\documentclass[aos]{imsart}
\setpkgattr{copyright}{text}{}
\setpkgattr{issuedata}{text}{}

\RequirePackage{amsthm,amsmath,amsfonts,amssymb}
\RequirePackage[numbers,sort]{natbib}
\RequirePackage[colorlinks,citecolor=blue,urlcolor=blue]{hyperref}
\RequirePackage{graphicx}
\usepackage{babel}
\pdfoutput=1

\usepackage[usenames,dvipsnames]{color}
\RequirePackage[OT1]{fontenc}
\usepackage{amsmath}
\usepackage{amssymb,amscd,amsfonts,amsthm,mathtools}
\usepackage[capitalize]{cleveref}
\usepackage{thmtools}
\usepackage{booktabs}
\usepackage{tikz}
\usepackage{nccmath}
\usetikzlibrary{calc}

\startlocaldefs
  \usepackage{tabto}
  \TabPositions{1.5em}
  \theoremstyle{plain}
  \newtheorem{theorem}{Theorem}
  \newtheorem{lemma}[theorem]{Lemma}
  \newtheorem{proposition}[theorem]{Proposition}
  \newtheorem{corollary}[theorem]{Corollary}
  \theoremstyle{definition}
  
  \declaretheorem[style=definition,numbered=no,name=Remark]{remark}
  \declaretheorem[style=definition,numbered=no,name=Convention]{convention}
  \declaretheoremstyle[postheadspace=.4em,headfont=\scshape]{examplestyle}
  \declaretheorem[style=examplestyle,numbered=no,name=Examples]{examples}
  \declaretheorem[style=examplestyle,numbered=no,name=Example]{example}
  \newcounter{exampleno}
  \renewcommand{\theexampleno}{{\textcolor{black}{(}\roman{exampleno}\textcolor{black}{)}}}
  \newcommand{\exitem}[1]{\noindent\refstepcounter{exampleno}\label{#1}\theexampleno}
  \crefalias{exampleno}{example}
  \DeclareMathOperator{\tsum}{{\textstyle\sum}}
  \DeclareMathOperator{\msum}{{\medmath\sum}}
  \DeclareMathOperator{\medcup}{\mathbin{\scalebox{1.5}{\ensuremath{\cup}}}}%
\endlocaldefs

\begin{document}

\begin{abstract}
  A distributional symmetry is invariance of a distribution under a group of transformations.
  Exchangeability and stationarity are examples.
  We explain that a result of ergodic theory provides a
  law of large numbers: If the group satisfies suitable conditions, expectations 
  can be estimated by averaging over subsets of transformations,
  and these estimators are strongly consistent. 
  We show that, if a mixing
  condition holds, the averages also satisfy a central limit theorem,
  a Berry-Esseen bound, and concentration. These are
  extended further to apply to
  triangular arrays, to randomly subsampled
  averages, and to a generalization of U-statistics. As applications, we obtain new results on exchangeability, random fields, network models,
  and a class of marked point processes. We also establish asymptotic normality of the empirical entropy for a large class of processes. Some known results are recovered as special cases, and can hence be interpreted as an outcome of symmetry.
  The proofs adapt Stein's method.
\end{abstract}

\begin{frontmatter}
  \title{Limit theorems for invariant distributions}
  \begin{aug}
    \author{\fnms{Morgane\ }\snm{Austern}\ead[label=e1]{}}
    \and
    \author{\fnms{Peter\ }\snm{Orbanz}\corref{}\ead[label=e2]{}}
    \affiliation{Harvard University and University College London}
  \end{aug}

  \maketitle
\end{frontmatter}

\newcommand{\myitem}{{\raisebox{0.28ex}{\tiny$\bullet$}}}
\renewcommand{\labelitemi}{\myitem}
\newcommand{\myint}{\medint\int}
\renewcommand{\triangle}{\vartriangle}
\renewcommand{\emptyset}{\varnothing}

\def\group{\mathbb{G}}
\def\dgroup{d_{\text{\rm G}}}
\def\kword#1{\textbf{#1}}
\def\K{\mathcal{K}}
\def\xspace{\mathbf{X}}
\def\borel{\mathcal{B}}
\def\Folner{F{\o}lner }
\def\mean{\mathbb{E}}
\def\condind{{\perp\!\!\!\perp}}
\def\ie{i.e.\ }
\def\eg{e.g.\ }
\def\equdist{\overset{\scriptscriptstyle\smash{\text{\rm\tiny d}}\vphantom{x}}{=}}
\def\equas{=_{\text{\rm\tiny a.s.}}}
\def\braces#1{{\lbrace #1 \rbrace}}
\def\bigbraces#1{{\bigl\lbrace #1 \bigr\rbrace}}
\def\Bigbraces#1{{\Bigl\lbrace #1 \Bigr\rbrace}}
\def\simiid{\sim_{\mbox{\tiny iid}}}
\def\Law{\mathcal{L}}
\def\iid{i.i.d.\ }
\def\L{\mathbf{L}}
\def\INV{\mathbf{I}_{\group}}
\def\ind#1{\text{\tiny #1}}
\newcommand{\argdot}{{\,\vcenter{\hbox{\tiny$\bullet$}}\,}}
\def\A{\mathbf{A}}
\def\B{\mathbf{B}}
\def\cL{\overline{\mathbf{L}}}
\def\sigmaf{\sigma_{\! f}}
\def\dW{d_{\ind{\rm W}}}
\def\ERG{\mathbf{E}}
\def\pMeas{\mathcal{P}}
\def\Erdos{Erd\H{o}s}
\def\Renyi{R\'enyi\ }

\newcommand{\darrow}{\xrightarrow{\;\text{\rm\tiny d}\;}}
\newcommand{\cF}{\overline{\mathbb{F}}}

\section{Introduction}
Statistical models that can be characterized by symmetry, or transfor\-ma\-tion invariance,
include stationary processes \citep{Shields:1996}, 
graphon and graphex models of networks
\citep{Bickel:Chen:Levina:2011:1,Ambroise:Matias:2012:1,Klopp:Tsybakov:Verzelen:2017:1,Caron:Fox:2017:1,Veitch:Roy:2016:1,Borgs:Chayes:Cohn:Holden:2016:1},
the exchangeable random partitions that underpin much of Bayesian nonparametrics
\citep{James:Lijoi:Pruenster:2009,Pitman:2006}, and
rotation- and shift-invariant random fields
\citep{Bolthausen:1982:1,Jensen:Kuensch:1994:1}.
Examples from related fields are various models for relational data and preference prediction used in machine learning
\citep{Orbanz:Roy:2015},
point process representations of nearest neighbor methods and Voronoi tesselations
\citep{grain,merlin,Penrose:2007}, or self-similar
stochastic processes \citep{Kallenberg:2001}.
Recent advances in spin glass theory rely crucially
on exchangeable arrays \citep{Panchenko:2013}.

We consider estimation under such invariant models. For each example
above, a canonical estimator for expectations is known. We explain 
that these estimators are special cases of a general class of averages.
For such averages, the ergodic theorem
of Lindenstrauss \citep{Lindenstrauss:2001:1} provides what a statistician
would call a (strong) law of large
numbers. 
Starting from this result, we establish central limit theorems,
Berry-Esseen bounds, and a concentration inequality.
We then develop several applications in detail.

\subsection{Overview}
\label{sec:overview}
The remainder of this section is an informal summary of our
approach, and of the main results. For the purposes of this
introduction, we sidestep technicalities:
A key quantity
throughout is an infinite group $\mathbb{G}$. We
assume for now that $\group$ is countable, and postpone general
definitions to \cref{sec:background}.

Consider a random element $X$ of a space $\xspace$, and a real-valued function $f$.
Suppose the group $\group$ consists of measurable bijections
${\phi:\xspace\rightarrow\xspace}$. We can then transform $X$ by
$\phi$, where we use the notation $\phi(X)$ and $\phi X$ interchangeably.
The purpose of this work is to understand under what conditions the
expectation $\mean[f(X)]$ can be estimated by
\begin{equation}
  \label{eq:intro:0}
  \mathbb{F}_n(f,X):=\mfrac{1}{|\A_n|}\msum_{\phi\in\A_n} f(\phi X)\;,
\end{equation}
where ${\A_1,\A_2,\ldots}$ are finite subsets of
$\mathbb{G}$, and $|\argdot|$ denotes cardinality. (For uncountable groups, $\mathbb{F}_n$ integrates
over a compact set $\A_n$.)
Such averages occur in
dynamical systems 
\citep{Einsiedler:Ward:2011:1} and  
statistical mechanics \citep{Parisi:1988}.
Various examples are used in statistics:\nolinebreak
\begin{example}
  \exitem{ex:intro:Z}
  The window estimator for a random field on a grid
  \citep{Bolthausen:1982:1,Jensen:Kuensch:1994:1}. In this case,
  ${X=(X_{ij})_{i,j\in\mathbb{Z}}}$ is a collection of real-valued random variables.
  Let $f$ be a function that depends only on
  the value at the origin, so ${f(X)=g(X_{00})}$ for some
  function $g$.
  A transformation that shifts the grid is of the
  form ${\phi=(k,l)}$ for some ${k,l\in\mathbb{Z}}$.
  If we choose ${\A_n:=\braces{-n,\ldots,n}^2}$, then\nolinebreak
  \begin{equation}
  \label{intro:eq:rfield}
    \mathbb{F}_n(f,X)
    =
    \tfrac{1}{|\A_n|}\tsum_{(k,l)\in\A_n}f((X_{i+k,j+l})_{i,j\in\mathbb{Z}})
    =
    \tfrac{1}{(2n+1)^2}\tsum_{|i|,|j|\leq n}g(X_{ij})
  \end{equation}
  averages $g$ over all locations on the subgrid of radius $n$ around
  the origin. The group $\group$ is the group
  $\mathbb{Z}^2$ of all shifts, with addition as group operation.
\end{example}
\noindent
More generally, $X$ is a random object---such as a
random sequence, matrix, field, or graph---and $f$ is
a function that typically depends only on ``a small part'' of $X$. The group $\group$
is a set of transformations that ``move the domain'' of $f$
over $X$, and $\A_n$ contains those elements of $\group$ that cover
a suitably defined sample, whose size is a function of $n$. The next two examples choose
$\A_n$ as $\mathbb{S}_n$, the set of all
permutations of the set $\braces{1,\ldots,n}$.\nolinebreak
\begin{examples}
  \exitem{ex:intro:ex:sequence} The sample average over a random sequence
  ${X=(X_1,X_2,\ldots)}$. Consider a function
  ${f(X)=g(X_1)}$ of the first entry,
 and let each permutation ${\phi\in\mathbb{S}_n}$ transform $X$
 by permuting entries, ${\phi
   X\!:=(X_{\phi(1)},\ldots,X_{\phi(n)},X_{n+1},X_{n+2},\ldots)}$. Then
  \begin{equation}
  \label{intro:eq:sequence}
  \mathbb{F}_n(f,X)=\tfrac{1}{|\mathbb{S}_n|}\tsum_{\phi\in\mathbb{S}_n}f(\phi
  X)=\tfrac{1}{n!}\tsum_{\phi\in\A_n}g(X_{\phi(1)})=\tfrac{1}{n}\tsum_{i\leq n}g(X_{i})\;.
\end{equation}
 In this case, the group is ${\group=\cup_n\mathbb{S}_n}$,
the set of all finite permutations of $\mathbb{N}$.\\[-.7em]

\exitem{ex:intro:graph} The triangle density in network
analysis \citep{Bickel:Chen:Levina:2011:1,Ambroise:Matias:2012:1}. Here,
$X$ is a random undirected, simple graph with vertex set
  $\mathbb{N}$. Denote by ${X[i_1,\ldots,i_k]}$ the induced subgraph on the
vertices ${i_1,\ldots,i_k\in\mathbb{N}}$. Let $g$ be a function defined on
graphs with three vertices, and set ${f(X):=g(X[1,2,3])}$.
Suppose each
${\phi\in\mathbb{S}_n}$ transforms the graph by permuting the first $n$ vertices, so
${(\phi X)[1,2,\ldots]=X[\phi(1),\ldots,\phi(n),n+1,n+2,\ldots]}$. Then
$\mathbb{F}_n$ averages $g$ over all subgraphs of size 3 in the finite graph
${X[1,\ldots,n]}$:
\begin{equation*}
  \mathbb{F}_n(f,X)=\tfrac{1}{|\mathbb{S}_n|}\tsum_{\phi\in\mathbb{S}_n}g(X[\phi(1),\phi(2),\phi(3)])
  =\tfrac{1}{n(n-1)(n-2)}\tsum g(X[i,j,k])\;,
\end{equation*}
where the sum on the right runs over all distinct triples ${i,j,k\leq n}$.
\end{examples}

\subsubsection*{Tools from ergodic theory}
To characterize the behavior of $\mathbb{F}_n$, we borrow from
ergodic theory: Two key conditions are
\begin{equation}
  \label{intro:conditions}
  \text{(i)}\quad \phi X\equdist X
  \quad\text{ and }\quad
  \text{(ii)}\quad
  |\phi\A_n\cap\A_n|/|\A_n|\xrightarrow{n\rightarrow\infty}1
  \quad\text{ for all }\phi\in\group\;,
\end{equation}
where $\equdist$ is equality in distribution.
If (\ref{intro:conditions}i) holds, $X$ is called
$\group$-invariant.
If it also satisfies
\begin{equation}
    \label{intro:conditions:ergodic}
  P(X\in A)\in\braces{0,1}\;\;\text{ for every Borel set }A\text{
    with }\phi A = A\text{ for
    all }\phi\in\group\;,
\end{equation}
it is called $\group$-ergodic. (Uncountable groups require more
general formulations of (\ref{intro:conditions}ii) and
\eqref{intro:conditions:ergodic}, see \cref{sec:background}.) The
same terminology is applied to the distribution of $X$,
so a $\group$-ergodic probability measure is the law of
$\group$-ergodic random element, etc. \cref{tab:ergodic} lists examples.

To motivate the conditions informally, first observe that $\mathbb{F}_n$ attempts to estimate
$\mean[f(X)]$ from surrogate values $f(\phi X)$.
That should require ${\mean[f(X)]=\mean[f(\phi
  X)]}$, which is in turn implied by (\ref{intro:conditions}i).
Any valid estimator
$\mathbb{F}_n$ of $\mean[f(X)]$ must satisfy ${\mathbb{F}_n\approx\mean[f(X)]}$ in some suitable sense for large enough $n$, so it must also satisfy
\begin{equation*}
  \mathbb{F}_n(f,X)\approx\mean[f(X)]
  =\mean[f(\phi X)]\approx
  \mathbb{F}_n(f,\phi X)\;.
\end{equation*}
That is true if ${\phi\A_n\approx\A_n}$, which is guaranteed by
(\ref{intro:conditions}ii).
In statistics, this condition was first used by
Charles Stein, to characterize groups for which the Hunt-Stein theorem establishes
minimaxity of invariant tests \citep{Bondar:Milnes:1981:1}.
Ergodicity can be motivated as follows:
We hope to establish strong consistency of estimates, that is,
${\mathbb{F}_n(f,X)\rightarrow\mean[f(X)]}$ almost surely as ${n\rightarrow\infty}$.
That means the event ${\braces{\mathbb{F}_n(f,X)\rightarrow a}}$ must have
probability $1$ for ${a=\mean[f(X)]}$, and $0$ otherwise.
Since invariance implies ${\mean[f(\phi X)]=\mean[f(X)]}$, these
events are invariant sets for all ${a\in\mathbb{R}}$ (and indeed any invariant
measurable set can be characterized in this way for some $f$ and some invariant
$X$). In this sense, $\group$-ergodic distributions form a class for
which strong consistency might hold, provided one can establish a
suitable strong law of large numbers.

\begin{table}[b]
    \caption{}
    \label{tab:ergodic}
    \resizebox{\textwidth}{!}{
      \begin{tabular}{@{}llll@{}}
        $\group$-invariant objects $X$ & $\group$-ergodic objects &
        $\xi$ explained by
        & eq.\ \eqref{eq:intro:lindenstrauss} specializes to \\
        \midrule
        exchangeable sequences & \iid sequences & de Finetti's theorem \citep{Kallenberg:2001} & law of large numbers\\
        stationary Markov chain & irreducible chains
        \citep{Shields:1996} &
                                                          Rohlin's
                                                          source thm. \citep{Shields:1996} & Birkhoff's theorem \citep{Kallenberg:2001}\\
        exchangeable graphs & graphon models
                              \citep{Borgs:Chayes:Lovasz:Sos:Vesztergombi:2008,Diaconis:Janson:2007}
                                                & Aldous-Hoover thm.
                                               \citep{Kallenberg:2005}
        & graph limit convergence \citep{Borgs:Chayes:Lovasz:Sos:Vesztergombi:2008} \\
        graphs generated by & graphex models \citep{Caron:Fox:2017:1} & Kallenberg's represen-
        & empirical graphex \citep{Veitch:Roy:2016:1,Borgs:Chayes:Cohn:Holden:2016:1}\\
        {\ }inv. point processes & & {\ }tation theorem \citep{Kallenberg:2005} &\\
        exchangeable arrays & dissociated arrays
        \citep{Kallenberg:2005} & Aldous-Hoover
                                                   theorem &
                                                             Kallenberg's LLN \citep{Kallenberg:1999}
      \end{tabular}
  }
\end{table}

This law of large numbers is due to Lindenstrauss \citep{Lindenstrauss:2001:1}:
If (\ref{intro:conditions}ii) holds, and $X$ is $\group$-ergodic,\nolinebreak
\begin{equation}
  \label{eq:intro:lindenstrauss:ergodic}
  \mathbb{F}_n(f,X)\;\xrightarrow{n\rightarrow\infty}\;\mean[f(X)]
  \qquad\text{ almost surely }
\end{equation}
for any function $f$ with ${\mean[|f(X)|]<\infty}$.
The sets $\A_n$ must satisfy certain additional fine print, 
but they can
always be modified to do so if they satisfy
(\ref{intro:conditions}ii). \cref{theorem:lindenstrauss} in
\cref{sec:background} gives a proper statement.

The theorem can be extended to the $\group$-invariant case. The two cases are related by a property 
known as ergodic decomposition:
$\group$-invariant distributions are mixtures of $\group$-ergodic
ones. More formally, if $X$ is $\group$-invariant, there is a
random element $\xi$ of the set
of $\group$-ergodic distributions such that ${X|\xi\sim\xi}$
(see \cref{theorem:ergodic:decomposition} for details).
If $X$ is $\group$-invariant, \eqref{eq:intro:lindenstrauss:ergodic} becomes\nolinebreak
\begin{equation}
  \label{eq:intro:lindenstrauss}
  \mathbb{F}_n(f,X)\;\xrightarrow{n\rightarrow\infty}\;\mean[f(X)|\xi]=\myint f(x)\xi(dx)
  \qquad\text{ almost surely. }
\end{equation}
For example, a random sequence $(X_i)_{i\in\mathbb{Z}}$
is stationary if it is $\mathbb{Z}$-invariant (adding elements of
$\mathbb{Z}$ shifts the index set). In this case,
ergodic
decomposition becomes Rohlin's stationary source theorem \citep{Shields:1996},
and \eqref{eq:intro:lindenstrauss:ergodic} specializes to Birkhoff's ergodic theorem.
An exchangeable (i.e.\ permutation-invariant) sequence
is ergodic if it is i.i.d.---see \cref{ex:deFinetti} for details. Thus, ${X|\xi\sim\xi}$ means $X$ is 
``conditionally i.i.d.'', which is de Finetti's theorem, and
\eqref{eq:intro:lindenstrauss:ergodic} is the
strong law of large numbers.

\subsubsection*{Sketch of main results} Our results provide rates of
convergence for $\mathbb{F}_n$. 
Like certain convergence results for stationary processes,
they use a mixing condition to control dependence
within $X$:
A typical mixing condition for a discrete-time process
  ${(X_1,X_2,\ldots)}$ would be that any pair $(X_j,X_k)$, for
${j<k}$, is approximately independent of the tail ${(X_{k+n},X_{n+k+1},\ldots)}$ 
for large $n$
\citep{Billingsley:1995}. Informally, we replace the
tail by ${(f(\psi X))_{\psi\in G}}$, for a set
${G\subset\group}$, and require\nolinebreak
\begin{equation}
  \label{eq:intro:mixing}
  (f(\phi_1 X),f(\phi_2 X)) \;\condind\;
  (f(\psi X))_{\psi\in G}\;|\;\xi
  \quad\text{ approximately }
\end{equation}
whenever ${\phi_1,\phi_2\in\group}$ are far from $G$. The condition is
tailored to second-order results, hence the pair on the left.
Since ${X|\xi\sim\xi}$, conditional independence given $\xi$ suffices.
\cref{sec:mixing} gives a precise definition.

Our first result is a central limit theorem: If ${\mean[|f(X)|^{2+\varepsilon}]<\infty}$ for some
${\varepsilon>0}$, and the conditional mixing property above holds, then
\begin{equation*}
  \sqrt{|\A_n|}\;\bigl(\mathbb{F}_n(f,X)-\mean[f(X)|\xi]\bigr)\;\xrightarrow{\;\rm{\tiny
      d}\;}\;\eta
  Z\quad\text{ for }Z\sim N(0,1)\;.
\end{equation*}
The asymptotic variance $\eta^2$ is a random variable, independent of
$Z$, and constant if $X$ is $\group$-ergodic. That is \cref{theorem:CLT}. If ${\mean[|f(X)|^{4+2\varepsilon}]<\infty}$,
\cref{theorem:BE} bounds the approximation error as 
\begin{equation*}
  \dW\Bigl(
          {\textstyle\frac{\sqrt{|\A_n|}}{\eta}}\bigl(\mathbb{F}_n(f,X)-\mean[f(X)|\xi]\bigr),Z
          \Bigr)
          \;\leq\;
          u(\A_n,\eta)
\end{equation*}
for a suitable function $u$ and the Wasserstein distance $\dW$. This
result generalizes the Berry-Esseen theorem.
In either case, the moment condition can be relaxed to
${\varepsilon=0}$, at the price of stronger mixing.

In statistics, asymptotic normality results are often applied to
quantify uncertainty. \cref{CI} shows that, if $z_{1-\frac{\alpha}{2}}$ is the
${(1-\alpha)}$-quantile of the standard normal distribution, 
\begin{equation*}
  \limsup_{n\rightarrow \infty}   P\Big(\mathbb{E}[f(X)|\xi]\in \big[\mathbb{F}_n(f,X)- z_{1-\frac{\alpha}{2}}\mfrac{\hat\eta}{\sqrt{|\A_n|}},\mathbb{F}_n(f,X)+ z_{1-\frac{\alpha}{2}}\mfrac{\hat\eta}{\sqrt{|\A_n|}}\big]\Big)\le \alpha
\end{equation*}
under the distribution $P$ of $X$, where $\hat{\eta}$ is an empirical variance that can be
computed from a sample of size $n$. In other words, the
interval estimate above is a consistent confidence interval.

In \cref{sec:g}, we generalize $\mathbb{F}_n$ along three lines:
(i) $f$ and $X$ may change with $n$.
(ii) Averages may be subsampled or randomized. In the simplest case, that means
replacing $\A_n^{k_n}$ by a random subset $\smash{\widehat{\A}_n}$, and
generalizing $\mathbb{F}_n$ to
\begin{equation*}
  \widehat{\mathbb{F}}_n(f_n,X_n)=\tfrac{1}{|\widehat{\A}_n|}\tsum_{\phi\in\hat{\A}_n}f_n(\phi X_n)-\mean[f_n(X_n)|\xi_n]\;.
\end{equation*}
More generally, $\widehat{\mathbb{F}}_n$ is defined by a
random measure $\mu_n$ on $\A_n$, so that random subsets are the special
case where $\mu_n$ is uniform on $\smash{\widehat{\A}_n}$.
(iii) Each $\phi$ may be substituted by a vector of transformations,
with some number $k_n$ of elements, replacing $\group$ by
$\group^{k_n}$ and $\A_n$ by $\A_n^{k_n}$.
Our main results are a central limit theorem (\cref{theorem:CLT:g}) and a Berry-Esseen bound (\cref{theorem:BE:g}) for
$\widehat{\mathbb{F}}_n$.
We use the result for $k_n$-tuples to formulate a class of generalized U-statistics
(\cref{corollary:U:statistics}).

Since certain
asymptotic properties of i.i.d.\ sequences generalize to
$\group$-invariant objects, it is natural to ask whether 
finite-sample properties do so, too. \cref{sec:concentration} gives a concentration inequality of the form
\begin{equation*}
  \mathbb{P}\bigl(\widehat{\mathbb{F}}_n(f,X)\geq t\bigr)\;\leq\; 2e^{-\omega_n t^2}\qquad\text{ for all }t>0\;,
\end{equation*}
for certain constants $\omega_n$, 
where $\mathbb{P}$ denotes probability under the joint distribution of
$X$ and the (possibly randomized) average $\widehat{\mathbb{F}}_n$.

\subsection{Applications}

The remaining sections apply the theorems summarized above to
obtain new results for a number of specific models, and also highlight
how certain known results can be phrased as instances of invariance.
We consider two specific types of invariance---stationarity and
exchangeability---in some detail because of their importance to
statistics, but also discuss other applications, to point
processes and entropy.

For stationary random fields, the group $\group$ plays a dual role, as index set
and a set of shifts. Substituting a field indexed
by the grid ${\mathbb{Z}^d}$ into \cref{theorem:CLT}
recovers Bolthausen's central limit theorem
\citep{Bolthausen:1982:1}. Substituting other groups generalizes
this result, as \cref{corollary:Bolthausen} illustrates for a
continuous field on ${\group=\mathbb{R}^d}$.
\cref{theorem:BE} is a Berry-Essen bound for Bolthausen's theorem.
If $\group$ is uncountable, the estimator $\mathbb{F}_n$ becomes an integral.
In applications where computing the integral
is not be feasible, it can be discretized to a sum, by applying
\cref{theorem:CLT:g,theorem:BE:g}, with
$\mu_n$ chosen as almost surely discrete. Making $\mu_n$ 
non-random makes the discretization deterministic.
\cref{corollary:data:augmentation} illustrates both cases, again for a
continuous random field. 

Modeling assumptions made in statistics often imply some form of
invariance under the permutation group $\mathbb{S}_\infty$. Examples
include i.i.d.\ sequences, Bayesian models appealing to de
Finetti's theorem (which generate exchangeable sequences),
stochastic block models and graphon models (which generate
exchangeable graphs), and finite and Dirichlet process mixtures
(which generate exchangeable partitions).
\cref{theorem:exchangeable} is a general central limit theorem and
Berry-Esseen bound. It shows that 
any ${\mathbb{S}_\infty}$-invariant random object $X$ satisfies
\begin{equation*}
  \sqrt{n}\bigl(\tfrac{1}{n!}\tsum_{\phi\in\mathbb{S}_n}f(\phi X)-\mean[f(X)|\mathbb{S}_\infty]\bigr)
  \;\xrightarrow{\;d\;}\;
    \eta Z\qquad\text{ where }\eta\,\condind\, Z\sim N(0,1)\;.
\end{equation*}
The result does not require a mixing condition.

Some models do not make an invariance assumption on the data source,
but rather use invariant random objects as approximations or latent
variables. One example are nonparametric stochastic block models
that increase the number of ``communities'' in the model with sample size
\citep{Choi:Wolfe:Airoldi:2012:1}.
An observed graph with $n$ vertices
is explained by an exchangeable graph
$X_n$, whose distribution changes with $n$.
\cref{hypo} shows how to estimate a statistic,
where we choose the triangle density for illustration. Informally,
\begin{equation*}
  \mfrac{\sqrt{n}}{\eta_n}
  \bigl(
  \text{empirical triangle density}(n)-
  \text{population triangle density under }X_n\bigr)
  \;\xrightarrow{\;d\;}\;
  Z\;,
\end{equation*}
where $\eta_n$ is determined by the law of $X_n$.
Another example are graphex random graphs
\citep{Caron:Fox:2017:1,Borgs:Chayes:Cohn:Holden:2016:1,Veitch:Roy:2016:1}, which are not themselves
exchangeable, but generated by a latent point process with an
exchangeability property.
\cref{sec:graphex} explains how to apply \cref{theorem:exchangeable}
to such models by exctracting an exchangeable surrogate object.
\cref{result:graphex} is an example: The standard estimator for the graphex
equivalent of the edge density satisfies
\begin{equation*}
  \sqrt{s}
  \bigl(
  \text{empirical graphex subgraph density}(s)-
  \text{graphex subgraph density}\bigr)
  \;\xrightarrow{\;d\;}\;
  \eta Z\;,
\end{equation*}
where $s$ is the relevant notion of sample size.

In \cref{sec:marked:pp}, we obtain a central limit theorem and
Berry-Esseen bound for so-called random geometric measures, 
which have been used to study problems such as
nearest-neighbor methods and tessalations.
These are point processes that allow points to depend on neighboring points,
where the neighborhood is defined by an ``observation
window'' shifted over regions of the sample space.
Representing these shifts as elements of a group makes
\cref{theorem:CLT,theorem:BE} applicable.

\cref{sec:entropy} concerns entropy: The entropy of a stochastic process
is defined as a limit of so-called empirical entropies, computed from
the first $n$ values of the process. 
This definition can be extended to certain invariant random objects, by 
defining the $n$th empirical entropy using the transformation set $\A_n$.
The fact that the limit exists is, in the classical case, known as the
Shannon-McMillan-Breiman theorem. 
Lindenstrauss \citep{Lindenstrauss:2001:1} has generalized it
to the invariant case. We show asymptotic normality:
Conditions under which 
\begin{equation*}
  \sqrt{|\A_n|}
  \bigl(
  \text{empirical entropy}(n)-\text{entropy}
  \bigr)
  \quad\xrightarrow{\;d\;}\quad
  \eta Z
  \qquad\text{ as }n\rightarrow\infty\;
\end{equation*}
holds are given in \cref{theorem:entropy}.

\section{Background and definitions}
\label{sec:background}

Throughout, $\group$ is a group, with
identity element $e$. By $\xspace$, we always mean a standard Borel
space, with Borel $\sigma$-algebra $\borel(\xspace)$, and by
$\pMeas(\xspace)$ the space of probability measures on $\xspace$,
topologized by weak convergence.
For a random element $X$ of
$\xspace$, and ${p>0}$, define the norm
${\|f\|_p:=\mean[|f(X)|^p]^{1/p}}$ for 
measurable functions
${f:\xspace\rightarrow\mathbb{R}}$. The set of functions with
${\|f\|_p<\infty}$ is denoted 
${\L_p(X)}$. By ${f\in\L_p(X)}$, we refer to a function $f$, rather
than an equivalence class.

\subsection{Conditions on the group}

To explain the estimator \eqref{eq:intro:0} for an uncountable group $\group$,
we must define a topology and a measure on $\group$. Finite sets then generalize to
compact ones, and sums over group elements to integrals.
To cohere with group structure, the topology 
must make the group operation continuous.
If that is the case, and 
the topology is locally compact, second-countable, and Hausdorff, or lcscH, then
$\group$ is a \kword{lcscH group}. If $\group$ is countable, the discrete topology is lcscH,
and $\group$ is a \kword{discrete group}.
We always equip $\group$ with its Borel $\sigma$-algebra $\borel(\group)$.
On every lcscH group, there is a
$\sigma$-finite measure $|\argdot|$ that satisfies
\begin{equation}
  \label{Haar:measure}
  |\phi^{-1} A|=|A|\qquad\text{ for all }\phi\in\group\text{ and }A\in\borel(\group)\;,
\end{equation}
called a \kword{Haar measure}. It is unique up
to positive scaling, so ${c|\argdot|}$ is again a Haar measure for ${c>0}$ \citep{Kallenberg:2001}.
If a set ${A\subset\group}$ is compact, then
${|A|<\infty}$. Informally, Haar measures generalize volume,
and \eqref{Haar:measure} shows that a set can be shifted without changing
its volume. Examples of Haar measures are Lebesgue measure
on the groups $(\mathbb{R}^r,+)$, for ${r\in\mathbb{N}}$, or counting measure (cardinality) on a
discrete group. Our results do not assume a specific scaling $c$, but in
examples 
we always choose $|\argdot|$ as cardinality if $\group$ is discrete.

Like volume, distance can be defined in a shift-invariant way: If $\group$ is lcscH, there
exists a metric $d$ on $\group$
that is \kword{left-invariant},\nolinebreak
\begin{equation}
  \label{eq:left:invariant}
  d(\phi^{-1}\argdot,\phi^{-1}\argdot)=d(\argdot,\argdot) \qquad\text{ for all }\phi\in\group\;.
\end{equation}
We write ${\B_t(\phi):=\braces{\psi\in\group|d(\psi,\phi)\leq t}}$ for
a metric ball centered at $\phi$, and abbreviate by ${\B_t:=\B_t(e)}$ a metric
ball around the identity.
One can always choose a left-invariant metric on $\group$ such that $\B_n$
``grows evenly'' with $n$,\nolinebreak
\begin{equation}
  \label{eq:metric:condition}
  \tfrac{|\B_{n+1}\setminus \B_n|}{|\B_n\setminus \B_{n-1}|} = O(1)\;,
\end{equation}
see \citep{Loeh:2017}.
If $G$ and $A$ are sets in $\group$, we write ${GA:=\braces{\phi\psi|\phi\in G,\psi\in A}}$.
A \kword{\Folner sequence} is a sequence of compact sets ${\A_1,\A_2,\ldots\subset\group}$ such that\nolinebreak
\begin{equation}
  \label{eq:Folner}
  \frac{|G\A_n\cap\A_n|}{|\A_n|}\xrightarrow{n\rightarrow\infty}1
  \qquad\text{ for every compact }G\subset\group\;.
\end{equation}
If $\group$ is discrete, its compact sets are the finite sets, and \eqref{eq:Folner} is equivalent to (\ref{intro:conditions}ii).
A lcscH group that contains a \Folner sequence is called \kword{amenable}
\citep{Einsiedler:Ward:2011:1}.
A \Folner sequence is \kword{tempered} if
\begin{equation}
  \label{eq:Shulman}
  \bigl|{\medmath\bigcup}_{k<n}\A_k^{-1}\A_n\bigr|\leq c|\A_n|\qquad\text{ for some }c>0\text{ and all }n\in\mathbb{N}\;.
\end{equation}
Not every \Folner sequence is tempered, but every lcscH group containing a \Folner sequence also contains
a tempered \Folner sequence \citep[][Proposition
  1.4]{Lindenstrauss:2001:1}.
\begin{convention}
We use the shorthand \kword{nice group} for an amenable lcscH
group $\group$ equipped with a metric $d$ satisfying \eqref{eq:left:invariant} and \eqref{eq:metric:condition}.
\end{convention}
\begin{examples}
  \exitem{ex:action:sginf:sequences} The group
  $\mathbb{S}_{\infty}$ of all permutations of $\mathbb{N}$ with finite support:
  Define $\mathbb{S}_{n}$ as the group of permutations of $\braces{1,\ldots,n}$,
  and ${\mathbb{S}_{\infty}:=\cup_{n\in\mathbb{N}}\mathbb{S}_{n}}$.
  The canonical metric on $\mathbb{S}_\infty$ is
    \begin{equation}
    d(\phi,\phi'):=\min\braces{n\in\mathbb{N}\,|\,\phi(n,n+1,\ldots)=\phi'(n,n+1,\ldots)}\;.
  \end{equation}
  The sequence ${(\mathbb{S}_n)}$ is a tempered \Folner sequence:
  Each ${\phi\in\group}$ is in $\mathbb{S}_n$ for $n$
  sufficiently large, so
  ${\phi\mathbb{S}_n\cap\mathbb{S}_n=\mathbb{S}_n}$ eventually, and
  (\ref{intro:conditions}ii)
  holds. Since ${\mathbb{S}_k^{-1}\mathbb{S}_n=\mathbb{S}_n}$
  whenever ${k\leq n}$, the sequence is tempered.
  \\[-.7em]

\exitem{ex:action:Zd:fields}
The shifts of the $r$-dimensional grid $\mathbb{Z}^r$ form the group ${(\mathbb{Z}^r,+)}$:
An element $\mathbf{j}$ of the group shifts a grid point $\mathbf{i}$ to ${\mathbf{i}+\mathbf{j}}$.
Its canonical metric
\begin{equation}
  \label{metric:Zr}
  d(\mathbf{i},\mathbf{j})=\min_{k\leq r}|i_k-j_k|
\end{equation}
is left-invariant and satisfies \eqref{eq:metric:condition}.
The balls ${\B_n=\braces{-n,\ldots,n}^r}$, for ${n\in\mathbb{N}}$, form a tempered \Folner sequence,
and so do the sets ${\braces{1,\ldots,n}^r}$.
\\[-.7em]

\exitem{ex:Rd}
Similarly, ${(\mathbb{R}^r,+)}$ is the shift group of $\mathbb{R}^r$. Lebesgue measure is a Haar
measure, Euclidean distance is a left-invariant metric satisfying \eqref{eq:metric:condition},
and the balls ${\B_n}$ and the sets ${[0,n]^r}$ both form tempered \Folner sequences.
\end{examples}
Recall from the introduction that $|\A_n|$ can be interpreted as
sample size. If $\group$ is compact, ${|\A_n|\leq|\group|<\infty}$.
It is hence essential for asymptotics that $\group$ is not
compact.
Examples of nice, non-compact groups include the groups above,
the group ${(\mathbb{R}_{>0},\cdot)}$ (which
characterizes self-similarity of stochastic processes),
the group of translations and rotations of a Euclidean space,
and discrete and continuous Heisenberg groups
\citep{Bump:Diaconis:Hicks:Miclo:Widom:2017}.
See \citep{Loeh:2017,Einsiedler:Ward:2011:1} for more.

\subsection{Invariance and ergodicity}
We now let elements of $\group$ transform elements of a space
$\xspace$. We must specify what that means: Permuting a matrix,
say, could mean permuting rows, or columns, or
entries. Such a specification is called an action:
A \kword{measurable action} of $\group$ on $\xspace$ is a jointly
measurable map
${(\phi,x)\mapsto T_{\phi}(x)}$ that satisfies
\begin{equation}
  \label{eq:action}
  T_e(x)=x \quad\text{and}\quad
  T_{\phi\phi'}(x)=T_\phi(T_{\phi'}(x))
  \quad\text{ for }x\in\xspace\text{ and }\phi,\phi'\in\group\,.
\end{equation}
The conditions ensure that the set of transformations $T_\phi$ defined by $\group$
on $\xspace$ is itself a group.
We usually simplify notation and write ${\phi(x):=T_\phi(x)}$.
A random element $X$ of $\xspace$ with distribution $P$ is \kword{$\group$-invariant} if
\begin{equation*}
  \phi(X)\;\equdist\;X\qquad\text{ or equivalently }\qquad
  P=P\circ\phi^{-1}\qquad\text{ for all }\phi\in\group\;.
\end{equation*}
We then call $P$ a $\group$-invariant measure. A Borel set ${A\in\borel(\xspace)}$
is \kword{almost invariant} if ${P(\phi A\vartriangle A)=0}$ for all ${\phi\in\group}$ and all $\group$-invariant $P$,
where $\vartriangle$ denotes symmetric difference.
The almost invariant sets form a $\sigma$-algebra $\sigma(\group)$, and
we abbreviate conditioning on $\sigma(\group)$ as
\begin{equation*}
  \mean[\argdot|\group]:=\mean[\argdot|\sigma(\group)]
  \quad\text{ and }\quad
  P(\argdot|\group):=P(\argdot|\sigma(\group))\;.
\end{equation*}
A probability measure is \kword{$\group$-ergodic}
if it is $\group$-invariant and ${P(A)\in\braces{0,1}}$ for all
${A\in\sigma(\group)}$. This condition is
equivalent to \eqref{intro:conditions:ergodic} if $\group$ is
countable \citep{Einsiedler:Ward:2011:1}.
A random element is $\group$-ergodic if its
distribution is.

\subsection{Estimation}
We now come to the general form of the estimator \eqref{eq:intro:0}.
For a group $\group$ acting measurably on $\xspace$, a \Folner
sequence $(\A_n)$ on $\group$, and a Borel function $f$ on $\xspace$, define
\begin{equation*}
  \label{eq:F}
  \mathbb{F}_n(f,x):=\mfrac{1}{|\A_n|}\myint_{\A_n}f(\phi x)|d\phi|\;.
\end{equation*}
If $\group$ is discrete, $\mathbb{F}_n$ simplifies to the sum \eqref{eq:intro:0}.
The cornerstone of our work is a result of Lindenstrauss, which
concluded a long line of work
by Ornstein, Weiss, and others \citep[e.g.][]{Weiss:2003:1}.\nolinebreak
\begin{theorem}[E.\ Lindenstrauss \citep{Lindenstrauss:2001:1}]
  \label{theorem:lindenstrauss}
  If a random element $X$ of a standard Borel space is invariant under
  a measurable action of a nice group, and if $(\A_n)$ is a tempered
  \Folner sequence, then
  \begin{equation}
    \label{eq:lindenstrauss}
    \mathbb{F}_n(f,X)\;\xrightarrow{n\rightarrow\infty}\;\mean[f(X)|\group]
    \quad\text{ almost surely for all }f\in\L_1(X)\;,
  \end{equation}
  where ${\mean[f(X)|\group]=\mean[f(X)]}$ almost surely if $X$ is ergodic.
\end{theorem}
Where convenient, we center $\mathbb{F}_n$ around the limit as
\begin{equation}
  \label{F:centered}
  \overline{\mathbb{F}}_n(f,X):=\mathbb{F}_n(f,X)-\mean[f(X)|\group]\;.
\end{equation}
The next result gives an interpretation of the limit:
If $X$ is invariant, it can be generated by selecting an ergodic
measure $\xi$ at random, and then drawing $X$ from $\xi$. The
limit $\mean[f(X)|\group]$ is the expectation of $f$ under the
instance of the latent measure $\xi$ that has generated $X$.
\begin{theorem}[Ergodic decomposition, {Varadarajan \citep{Varadarajan:1963}}]
  \label{theorem:ergodic:decomposition}
  If a lcscH group $\group$ acts measurably on a standard Borel
  space $\xspace$, the set of $\group$-invariant probability measures is 
  convex. Its set of extreme points is the set $\ERG$ of
  $\group$-ergodic measures, and is measurable in $\pMeas(\xspace)$.  A random element $X$ of $\xspace$
  is $\group$-invariant if and only if
  \begin{equation}
    \label{eq:ergodic:decomposition}
    P[X\in\argdot|\group]=\xi(\argdot) \qquad\text{almost surely}
  \end{equation}
  for a random element $\xi$ of $\mathbf{E}$.
  The law of $\xi$ is uniquely determined by that of $X$.
\end{theorem}
Thus, conditioning
on $\sigma(\group)$ means conditioning on $\xi$.
Another implication is that ${\xi=P}$ almost surely if $P$ is itself ergodic,
and therefore
\begin{equation*}
  \mean[f(X)|\group]=\myint f(x)d\xi(x)=\mean[f(X)]
    \quad\text{ if }X\text{ is ergodic.}
\end{equation*}
Taking expectations on both sides of 
\eqref{eq:ergodic:decomposition}
shows that $P$ is $\group$-invariant if and only if
\begin{equation*}
  P(X\in\argdot)=\myint_{\ERG} m(\argdot)\mathbb{P}(\xi\in dm)\;.
\end{equation*}
That provides a more geometric interpretation: Recall that every
element of a polytope in Euclidean space is a convex combination of
extreme points. The integral 
similarly represents $P$ as a generalized convex combination, or barycenter,
of extreme points. Compare this 
to the theorems of 
Krein-Milman and Choquet, which generalize the same property of 
polytopes to certain compact convex
sets \citep{Alfsen:1971}:
\cref{theorem:ergodic:decomposition} makes stronger requirements on
the elements of the convex set (they are $\group$-invariant
measures), but does not require compactness.

\begin{examples}
\exitem{ex:deFinetti}
Let $\xspace$ be the space ${\mathbb{R}^{\mathbb{N}}}$ of real-valued
sequences. Define an action of the permutation group $\mathbb{S}_\infty$ as
${\phi(x):=(x_{\phi(1)},x_{\phi(2)},\ldots)}$, for ${x\in\xspace}$ and ${\phi\in\mathbb{S}_{\infty}}$.
An \kword{exchangeable sequence} is a $\mathbb{S}_\infty$-invariant random sequence
${X=(X_i)_{i\in\mathbb{N}}}$. It is ergodic if and only if it is i.i.d., a fact known as the
Hewitt-Savage 0--1 law \citep{Kallenberg:2001}.
It follows that $\xi$ factorizes as
${\xi=\xi_0^{\otimes\mathbb{N}}}$, for some random probability measure
$\xi_0$ on $\mathbb{R}$. \cref{theorem:ergodic:decomposition} then takes the form
\begin{equation*}
  P(X\in\argdot)=\myint_{\pMeas(\mathbb{R}^\mathbb{N})}m(\argdot)\mathbb{P}(\xi\in dm)=\myint_{\pMeas(\mathbb{R})}m_0^{\otimes\mathbb{N}}(\argdot)\mathbb{P}(\xi_0\in dm_0)\;,
\end{equation*}
which is de Finetti's theorem \citep{Kallenberg:2001}.
Let ${f(x)=g(x_1)}$ be a function of the first sequence entry, as in \eqref{intro:eq:sequence}.
\cref{theorem:lindenstrauss} becomes
\begin{equation*}
  \mfrac{1}{n!}\msum_{\phi\in\mathbb{S}_n}f(\phi X)=\mfrac{1}{n}\msum_{i\leq n}g(X_i)\xrightarrow{n\rightarrow\infty}\myint_{\mathbb{R}} g(x_1)\xi_0(dx_1)
  \quad\text{a.s.}
\end{equation*}
For ergodic $X$, this is the strong law of large numbers
for i.i.d.\ sequences.\\[-.7em]

\exitem{ex:Zd:fields}
Fix ${r\in\mathbb{N}}$, and set ${\xspace=\mathbb{R}^{\mathbb{Z}^r}}$.
An element ${x=(x_{i})_{i\in\mathbb{Z}^r}}$ of $\xspace$ is
hence a scalar field on an $r$-dimensional grid. Define an action of
${\group=\mathbb{Z}^r}$ on $\xspace$ as
${\phi(x):=(x_{i+\phi})_{i\in\mathbb{Z}^r}}$ for ${\phi\in\mathbb{Z}^r}$.
A \kword{stationary random field} is a $\mathbb{Z}^r$-invariant random element $X$ of $\xspace$.
Recall from
\cref{ex:action:Zd:fields} that ${\A_n=\braces{-n,\ldots,n}^r}$ defines a \Folner sequence.
Write ${\Omega_n:=\braces{-n,\ldots,n}^r}$ to distinguish the subset $\Omega_n$ of the \emph{index} set ${\mathbb{Z}^r}$ from the subset $\A_n$ of
the \emph{group} ${\mathbb{Z}^r}$. In this case, \eqref{eq:Folner} can
be rephrased in terms of the index set:
Since ${\Omega_n:=\A_n(0,\ldots,0)}$,
  \begin{equation*}
    |\partial\Omega_n|\,/\,|\Omega_n|\xrightarrow{n\rightarrow\infty} 0
  \quad\text{ where }\quad
  \partial\Omega_n=\Omega_n\!\setminus\!\Omega_{n-1}\;.
\end{equation*}
  In this form, the condition is well-known in statistics \citep{Bolthausen:1982:1,Jensen:Kuensch:1994:1}.
  For a function ${f(x)=g(x_{0,\ldots,0})}$ at the origin,
$\mathbb{F}_n$ is given by \eqref{intro:eq:rfield}.
We also noted already that $\A_n$ can
alternatively be chosen as ${\braces{1,\ldots,n}^r}$.
For the case ${r=1}$ of stationary sequences, \cref{theorem:lindenstrauss} then takes the form
${n^{-1}\sum_{i=1}^ng(X_i)\rightarrow\mean[g(X_1)|\group]}$, which is
Birkhoff's ergodic theorem \citep{Shields:1996}.
\end{examples}

\def\mysetminus{\!\setminus\!}

\section{Conditional mixing}
\label{sec:mixing}

This section formalizes the mixing condition
sketched in \eqref{eq:intro:mixing}.  The label ``mixing'' is used for
a range of conditions, whose common denominator is
typically that they quantify dependence
using terms of
the form ${|P(A)P(B)-P(A\cap B)|}$. Their
strengths and purposes vary---Bradley \citep{bradley2005basic},
for example, surveys an extensive list of mixing conditions for
stationary processes.
Our notion of mixing resembles that used in random field asymptotics
\citep{Georgii:2011,Bolthausen:1982:1}.
Ergodic theory defines mixing conditions to verify ergodicity, which are typically
much weaker \citep{Einsiedler:Ward:2011:1}.
Consider ${f\in\L_1(X)}$, and
a set ${G\subset\group}$.
The events in
$\xspace$ that can be formulated in terms of
${(f(\phi X))_{\phi\in G}}$ form the $\sigma$-algebra
\begin{equation*}
\sigmaf(G):=\sigma(f\circ\phi,\phi\in G)=\sigma\bigl({\textstyle\bigcup_{\phi\in
  G}}(f\circ\phi)^{-1}\borel(\mathbb{R})\bigr)\;,
\end{equation*}
where $\borel(\mathbb{R})$ is the Borel $\sigma$-algebra of $\mathbb{R}$.
Write ${\B_t(G):=\cup_{\phi\in G}\B_t(\phi)}$. The set of
group elements whose distance from $G$ exceeds $t$ is
${\group\setminus\B_{t}(G)}$.
The relevant set of events is then
\begin{equation*}
  \mathcal{C}(t):=\bigbraces{
    (A,B)\in\sigmaf(\phi_1,\phi_2)\otimes\sigmaf(G)\big\vert
    G\subset\group, \phi_1,\phi_2\in\group\mysetminus\B_{t}(G)}\;.
\end{equation*}
The \kword{mixing coefficient} for $f$ and $P$ is the function
\begin{equation*}
  \alpha(t):=\sup_{(A,B)\in\,\mathcal{C}(t)}|P(A)P(B)-P(A\cap
  B)|\quad\text{ for }t>0\;,
\end{equation*}
and $P$ is \kword{mixing} with respect to $f$
if ${\alpha(t)\rightarrow 0}$ as ${t\rightarrow\infty}$.
Similarly,
\begin{equation*}
  \alpha(t|\group):=\sup_{(A,B)\in\,\mathcal{C}(t)}\mean[|P(A|\group)P(B|\group)-P(A\cap
  B|\group)|] \quad\text{ for }t>0\;
\end{equation*}
is the \kword{conditional mixing coefficient}, and $P$ is
\kword{conditionally mixing} if ${\alpha(t|\group)\rightarrow 0}$ as
${t\rightarrow\infty}$. Both coefficients are decreasing in $t$,
since ${\mathcal{C}(t_1)\subset\mathcal{C}(t_2)}$ if ${t_1\leq t_2}$.
\begin{lemma}
  \label{lemma:hypotheses}
  The mixing coefficients satisfy ${\alpha(k|\group)\leq 4\alpha(k)}$ for all ${k\in\mathbb{N}}$.
\end{lemma}
Thus, mixing implies conditional mixing.
The first example below shows that the converse need not be true.
The second example describes a case where both properties hold.
\begin{examples}
  \exitem{ex:mixing:exchangeable}
  Any exchangeable sequence ${X=(X_1,X_2,\ldots)}$ is
  conditionally mixing with respect to ${f:(x_1,x_2,\ldots)\mapsto
    x_1}$: By de Finetti's theorem, its entries are conditionally independent.
  For subsets ${F,G\subset\mathbb{N}}$, that implies
  \begin{equation*}
      (X_i)_{i\in F}
      \;\condind\;
      (X_j)_{j\in G}\;|\;\group
      \qquad\text{ if }\min_{i\in F,\,j\in G}|i-j|\geq 1\;,
  \end{equation*}
  and hence ${\alpha(k|\group)=0}$ for all ${k\in\mathbb{N}}$. It need
  not be mixing: Draw once from a
  random variable $Y$,
  and set ${X_i:=Y}$ for all ${i\in\mathbb{N}}$. Then $X$ is exchangeable, but
  dependence of $X_1$ and $X_i$ does not diminish as $i$ grows.\\[-.7em]

  \def\i{\mathbf{i}}
  \def\j{\mathbf{j}}
  \exitem{ex:mixing:field} Let ${X=(X_\i)_{\i\in\mathbb{Z}^r}}$ be a
  stationary random field with the Markov property: For each ${\i\in\mathbb{Z}^r}$,
  ${X_\i\,\condind\,(X_\j)_{\j\in\mathbb{Z}^d\setminus\braces{\i}}\,|\,(X_\j)_{\j\in\B_1(\i)}}$.
  If $X$ satisfies the so-called Dobrushin condition, namely
  \begin{equation*}
      \vartheta:=\sup_{\i|d(\i,0)=1}\sup_{A,B\in\borel(\xspace)}|P(X_0\in A|X_\i\in B)-P(X_0\in A)|\;\leq\;\tfrac{1}{2r}\;,
  \end{equation*}
  it is mixing with respect to all coordinate
  functions:
There are positive constants $c_1$ and $c_2$ such that
${\alpha(k)\leq c_1e^{-c_2 k}}$ for all ${k\in\mathbb{N}}$ \citep[e.g.][8.28]{Georgii:2011}.
By \cref{lemma:hypotheses}, that also implies conditional mixing.
In general, if $(X_i)$ is a stationary sequence,
$\alpha(\cdot|\group)$ can be bounded by the classical $\alpha$-mixing
coefficients \cite[e.g.][]{bradley2005basic}.\\[-.7em]

\exitem{} If $X$ is conditionally mixing for $f$, it also is for
${g\circ f}$, for any function $g$.
\end{examples}

\section{Basic limit theorems}
\label{sec:main}
The central limit theorem requires conditional mixing and a second-moment 
condition. The strength of each can be traded off against the other:
The next two theorems assume either
\begin{align}
  \label{H1}
  \text{(i)}&\;\;
  \mean[f(X)^2]<\infty
  &
  \text{(ii)}&\;\;
  \alpha(K|\group)=0
  \quad\text{ for some }K\in\mathbb{N}\;,
  \intertext{or that there exists an ${\varepsilon>0}$ such that}
  \label{H2}
  \text{(i)}&\;\;
  \mean[f(X)^{2+\varepsilon}]<\infty
  &
  \text{(ii)}&\;\;
  \myint_{\group}\alpha(d(e,\phi)|\group)^{\frac{\varepsilon}{2+\varepsilon}}|d\phi| < \infty\;,
\end{align}
where $e$ is the identity element of $\group$.
If $\group$ is discrete,
(\ref{H2}ii) simplifies to
 \begin{equation*}
    \tsum_{n\in\mathbb{N}}|\B_{n+1}\!\setminus\!\B_n|\,\alpha(n|\group)^{\frac{\varepsilon}{2+\varepsilon}}<\infty\;.
\end{equation*}
We note only en passant that the quantity
${|\B_{n+1}\!\setminus\!\B_n|}$ plays a crucial role in group
theory, where it is known as the growth rate of $\group$ \citep{Loeh:2017}.

\begin{theorem}
  \label{theorem:CLT}
  Let $\group$ be a nice group with tempered \Folner sequence
  ${(\A_n)}$, acting measurably on a standard Borel
  space $\xspace$. If a $\group$-invariant random element $X$ of
  $\xspace$ and a function ${f:\xspace\rightarrow\mathbb{N}}$ satisfy either \eqref{H1} or
  \eqref{H2}, then
  \begin{equation}
    \label{eq:CLT}
    \sqrt{|\A_n|}\;\bigl(
    \mathbb{F}_n(f,X)-\mean[f(X)|\group]\bigr)
    \quad\xrightarrow{\;\;\text{d}\;\;}\quad\eta Z\qquad\text{ for }Z\sim N(0,1)\;.
  \end{equation}
  The asymptotic variance ${\eta^2}$ is a random variable distributed as 
  \begin{equation}
    \label{eq:variance}
    \eta^2\equdist\myint_{\group}\eta^2(\phi)|d\phi|
    \quad\text{ for }\quad
    \eta^2(\phi):=\mean[f(X)f(\phi X)|\group]\;,
  \end{equation}
  and satisfies ${\eta^2<\infty}$ almost surely. It is independent of
  $Z$, and constant almost surely if $X$ is $\group$-ergodic.
\end{theorem}
The rate of convergence in Lindenstrauss' theorem is
thus ${|\A_n|^{-\frac{1}{2}}}$, and depends only on the \Folner sequence.
The action does not affect the rate, but the mixing coefficient and
constants. The ergodic decomposition property is visible in
the independence of $\eta$ and $Z$:
\cref{theorem:ergodic:decomposition} shows
${\mean[\argdot|\group]=\mean[\argdot|\xi]}$, so $\eta$
is a function of $\xi$, and constant if $X$ is $\group$-ergodic. Informally, the randomness of $Z$
is due to ${X|\xi}$, that of $\eta$ is due to $\xi$.

In statistical terms, $\mathbb{F}_n(X,f)$ is an estimate of
${\mean[f(X)|\group]}$ computed from a sample of size
$|\A_n|$. \cref{theorem:lindenstrauss} shows this
estimator is (strongly) consistent, and \cref{theorem:CLT} provides
the rate of convergence and shows the estimation error is
asymptotically normal. That can be used to obtain a consistent
confidence interval, as the next result shows. The additional condition on $\eta$ ensures---in
the non-ergodic case, where $\eta$ is not almost surely constant---that
its law does not place too much mass very close to $0$, which could lead
to effectively degenerate behavior even if ${\eta>0}$ almost surely.
\begin{theorem}
  \label{CI}
  Assume the conditions of \cref{theorem:CLT}, and additionally that
  ${P(\eta<t)\rightarrow 0}$ if ${t\searrow 0}$. Let $(b_n)$ be an 
  increasing sequence of positive integers satisfying
  \begin{equation*}
    \text{(i) }\;b_n\rightarrow\infty
    \qquad
    \text{(ii) }\;
    |\B_{b_n}|=o(\sqrt{|\A_n|})
    \qquad
    \text{(iii) }\;
    |\A_n\setminus \B_{b_n}\A_n|=o(|\A_n|)\;,
  \end{equation*}
  and define the empirical  variance
  \begin{equation*}
    \hat{\eta}_n^2:=\frac{1}{|\A_n|}\myint_{\A_n}\myint_{\B_{b_n}(\phi)}\big(f(\phi
    X)-\mathbb{F}_n(f,X)\big)\big( f(\phi'
    X)-\mathbb{F}_n(f,X)\big)|d\phi'||d\phi|\;.
  \end{equation*}
  For any ${\alpha\in (0,1)}$, let $z_{1-\frac{\alpha}{2}}$ be the positive scalar satisfying ${P(|Z|> z_{1-\frac{\alpha}{2}})=\alpha}$.
  Then
  \begin{equation*}
    \limsup_{n\rightarrow \infty}P\Big(\mathbb{E}(f(X)|\group)\in \big[\mathbb{F}_n(f,X)\pm z_{1-\frac{\alpha}{2}}\mfrac{\hat{\eta}_n}{\sqrt{|\A_n|}}\big]\Big)\;\le\; \alpha\;.
  \end{equation*}
\end{theorem}
The left- and right-hand side in \eqref{eq:CLT} can be compared in
terms of the Wasserstein distance $\dW$. For two random
elements $Y$ and $Y'$ of $\mathbb{R}$, this is
\begin{equation*}
  \dW(Y,Y'):=\sup_{h\in\mathcal{L}}|\mean[h(Y)]-\mean[h(Y')]|\;,
\end{equation*}
where $\mathcal{L}$ are the Lipschitz functions on
$\mathbb{R}$ with Lipschitz 
constant 1 \citep[e.g.][]{Ross:2011:1}.
We denote normalized moments of $f$ by
\begin{equation*}
  s_p:=\mean\bigl[\bigl|\tfrac{f(X)}{\eta}\bigr|^p\bigr]^{\frac{1}{p}}=
  \bigl\|\tfrac{f(X)}{\eta}\bigr\|_p
  \quad\text{ for }p>0\;.
\end{equation*}
The bound on $\dW$ depends both on the value of the integral in
(\ref{H2}ii), and on the decay of its tail, and we define
\begin{equation}
  \label{eq:def:tau}
  \tau(b):=\myint_{\group\setminus
    \B_b}\alpha(d(e,\phi)|\group)^{\frac{\varepsilon}{2+\varepsilon}}|d\phi|
  \qquad\text{ for }b\geq 0\;.
\end{equation}
Condition (\ref{H2}ii) then amounts to ${\tau(0)<\infty}$.
The next result generalizes the Berry-Esseen theorem; it quantifies
the speed of convergence in \cref{theorem:CLT}, and the coverage of the
confidence interval.
\begin{theorem}
  \label{theorem:BE}
  Assume the conditions of \cref{theorem:CLT}, with $\eta$ defined as
  in \eqref{eq:variance}, and let $Z$ be a standard normal variable.
  If \eqref{H1} holds, and ${K\in\mathbb{N}}$ is the smallest number
  for which
  ${\alpha(K|\group)=0}$, then
  \begin{equation*}
    \dW\Bigl(\tfrac{\sqrt{|\A_n|}}{\eta}\,\cF_n(f,X),Z\Bigr)
   \;\leq\;
   \kappa s_2^2\frac{|\A_n\triangle \B_K\A_n|}{|\A_n|}
   +   \kappa\frac{\max(s_4^3,1)|\B_K|^2}{\sqrt{|\A_n|}}
  \end{equation*}
  for a positive constant $\kappa$.
  If $f$ satisfies \eqref{H2} for some ${\varepsilon>0}$,
  \begin{align*}
    \dW\Bigl(\tfrac{\sqrt{|\A_n|}}{\eta}\,\cF(f,X),Z\Bigr)
    \;&\leq\;
    \kappa s_{2+\varepsilon}^2\frac{|\A_n|-|\A_n\cap\B_{b_n}\A_n|}{|\A_n|}\\
    &+\;
    \kappa 
    \max(s_{4+2\varepsilon}^3,1)\tau(0)\Bigl(
    \tau(b_n)+\frac{|\B_{b_n}|}{\sqrt{|\A_n|}}\Bigr)
  \end{align*}
  for a positive constant ${\kappa}$, and any sequence ${b_1<b_2<\ldots}$ of positive scalars.
\end{theorem}
The choice of
${(b_n)}$ trades off $|\B_b|$, which increases with $b$,
against $\tau(b)$, which decreases.

\begin{example}\exitem{}
Let $X$ be an i.i.d.\
sequence, and hence exchangeable and ergodic.
For ${f\in\L_2(X_1)}$, we have
${\alpha(1|\group)=0}$,
and \cref{theorem:CLT} is the elementary central limit theorem.
\cref{theorem:BE} is the Berry-Esseen bound {\citep[e.g.][]{Ross:2011:1}}:
Hypothesis \eqref{H1} holds, the first term of the bound satisfies ${\A_n\!\triangle\!\B_1\A_n=O(1/n)}$, and the second term 
collapses
to ${1/\sqrt{n}}$.
\end{example}
A less elementary application is a real-valued random field
${(X_{\phi})_{\phi\in\group}}$ that is stationary,
i.e.\ invariant under the group ${\group}$ acting on the index set
$\group$. For the groups $\mathbb{Z}^r$ and
$\mathbb{R}^r$, for instance, substituting into \cref{theorem:CLT} yields:
\begin{corollary}
  \label{corollary:Bolthausen}
  Let ${X=(X_{\phi})_{\phi\in\group}}$ be a stationary random
  field, and $f$ a real-valued function that satisfies \eqref{H2}. If
  ${\group=(\mathbb{Z}^r,+)}$ for some ${r\in\mathbb{N}}$,
  \begin{equation*}
    \sqrt{n^r}\Bigl(\tfrac{1}{n^r}\msum_{\mathbf{i}\in\braces{0,\ldots,n}^r}f(X_{\mathbf{i}})\,-\,\mean[f(X)|\mathbb{Z}^r]\Bigr)
    \quad\darrow\quad
    \eta Z\qquad\text{ as }n\rightarrow\infty\;
  \end{equation*}
  for ${\eta^2:=\sum_{\mathbf{i}\in\mathbb{Z}^r}\mean\bigl[f(X_0)f(X_{\mathbf{i}})\big|\mathbb{Z}^r\bigr]}$.
  If ${\group=(\mathbb{R}^r,+)}$ instead, then
  \begin{equation*}
    \sqrt{n^r}\Bigl(\tfrac{1}{n^r}\myint_{[0,n]^r}f(X_t)|dt|\,-\,\mean[f(X)|\mathbb{R}^r]\Bigr)
    \quad\darrow\quad
    \eta Z\qquad\text{ as }n\rightarrow\infty\;,
  \end{equation*}
  where
  ${\eta^2:=\int_{\mathbb{R}^r}\mean\bigl[f(X_0)f(X_t)\big|\mathbb{R}^r\bigr]|dt|}$. In
  either case, ${\eta\,\condind\,Z}$.
\end{corollary}

The case ${\group=\mathbb{Z}^r}$ is Bolthausen's central limit theorem \citep{Bolthausen:1982:1}.
Thus, \cref{theorem:CLT} implies a generalization of Bolthausen's theorem to
random fields indexed by nice groups, as the second case illustrates.
If $X$ satisfies the condition ${\vartheta<1/(2r)}$ in
\cref{ex:mixing:field}, it is conditionally mixing with respect to each
coordinate function, and the corollary holds for all functions
${f(X)=g(X_{0})}$ with ${g\in\L_{2+\varepsilon}(X_0)}$.1

If we quantify the approximation error using \cref{theorem:BE}, additional
properties of the group play a role, and we hence consider a specific
class:
${(\mathbb{Z}^r,+)}$ is a so-called finitely generated nilpotent group of rank
${r}$. Such groups are nice, and each contains a finite set called a generator. The
minimal number of elements of this set required to transform one group element into another is
a metric, the word metric, whose metric balls $\B_n$ satisfy \eqref{eq:Folner} and
${1/|\B_n|=O(n^{-r})}$. We refer to \citep{Loeh:2017} for details.
Substituting into \cref{theorem:BE} yields:
\begin{corollary}
  \label{corollary:nilpotent}
  Let $\group$ be a finitely generated, nilpotent group of rank
  ${r\in\mathbb{N}}$, and set ${\A_n:=\B_n}$ for the word metric of
  a finite generator. If there exist ${\varepsilon,\delta>0}$ such that
  ${\alpha(k|\group)=O(k^{-(r+\delta)})}$ and ${f(X)/\eta\in\L_{4+2\varepsilon}(X)}$,
  then
  \begin{equation*}
    \dW\Bigl(
    {\textstyle\frac{\sqrt{|\A_n|}}{\eta}}(\mathbb{F}_n(f,X)-\mean[f(X)|\group]),Z
    \Bigr)
    =
    O(n^{-r\delta/(2(r+\delta))})
    \;\text{ for }
    Z\sim N(0,1)\;,
  \end{equation*}
  where ${\eta}$ is defined as in \cref{corollary:Bolthausen} and
  independent of $Z$.
\end{corollary}
For ${\group=\mathbb{Z}^r}$, the unit coordinate vectors
in $\mathbb{Z}^r$ are a finite generator, and the word
metric it defines is \eqref{metric:Zr}.

\section{Generalized limit theorems}
\label{sec:g}

\def\hA{\widehat{\A}}

\def\yspace{\mathbf{Y}}
\def\x{\mathbf{x}}
\def\bphi{\boldsymbol{\phi}}
\def\bpsi{\boldsymbol{\psi}}
\def\hF{\widehat{\mathbb{F}}}
\def\supp{\text{supp}}
\def\gen#1{\langle #1\rangle}
\def\pr{\text{\rm pr}}
\def\cR{\widehat{\mathbb{F}}}

This section extends our main theorems to a
generalized version of the estimator $\mathbb{F}_n(f,X)$.
We begin with an informal overview; proper definitions follow in \cref{sec:g:definitions}.
The generalized estimator we will define is of the form
\begin{equation*}
  \frac{1}{\mu_n(\A_n^{k_n})}\int_{\A_n^{k_n}}f_n(T_n(\bphi,X_n))\mu_n(d\bphi)\;,
\end{equation*}
and again involves a random quantity, now denoted $X_n$,
a real-valued function $f_n$, and a group action $T_n$. Additionally, 
$\mu_n$ is a random measure, and ${k_n\in\mathbb{N}}$.
The estimator combines three separate extensions of $\mathbb{F}_n$:
\begin{itemize}
  \item {\em Triangular arrays}.
    We permit the function $f_n$ and the law of $X_n$ to
    depend on $n$. That generalizes an invariant random object $X$ in a similar
    way as triangular arrays generalize i.i.d.\ sequences
    \citep[e.g.][]{Kallenberg:2001}.
    Changing $X_n$ with $n$ may involve changing the sample space $\xspace_n$ and the action $T_n$. 
    An application example is a nonparametric network model in
    \cref{sec:adaptation}, which uses $m(n)$ parameters to explain an observed graph of size $n$.
    In this case,
    $f_n$ and $\xspace_n$ are fixed, but $m(n)$, and hence the distribution of $X_n$, depends on $n$.\\[-.5em]

\item {\em Randomization}. The set $\A_n$ may be randomized, which we formalize as a random
measure $\mu_n$ on $\A_n$. For example, if 
${\phi_{n1},\ldots,\phi_{nj_n}}$ are 
sampled with replacement from $\A_n$,
\begin{equation*}
  \mu_n:=j_n^{-1}\tsum_{i\leq j_n}\delta_{\phi_{ni}}
  \quad\text{ yields the average }\quad
  j_n^{-1}\tsum_{i\leq j_n}f(\phi_{ni} X)\;.
\end{equation*}
More generally,
if $\group$ is countable, $\mu_n$ may generate subsets (sampling
without replacement), multisets (sampling with replacement), or sets
of weighted points. In the uncountable case, $\mu_n$ may be
discrete (which discretizes the integral in $\mathbb{F}_n$ to a
sum), or generate uncountable subsets. An illustration is
\cref{corollary:data:augmentation}, which subsamples a rotation
group.\\[-.5em]

\item {\em U-statistics}.
  Consider a function ${g:\mathbb{R}^k\rightarrow\mathbb{R}}$ and a
  random sequence $(Y_1,Y_2,\ldots)$ in
  $\mathbb{R}$.
  A U-statistic can be defined in several equivalent ways (see
  \cref{sec:u-stat}), one of which is
  \begin{equation*}
    n^{-k}\tsum_{i_1,\ldots,i_k\leq
      n}g(Y_{i_1},\ldots,Y_{i_k})
    \;.
  \end{equation*}
  It can be expressed in terms of shifts: If
  ${f:(y_i)\mapsto y_1}$ is the first coordinate function and shifts in
  the set ${\A_n:=\braces{0,1,\ldots,n-1}}$ act on the index set of $Y$ by
  addition, we have
  \begin{equation*}
    n^{-k}\tsum_{i_1,\ldots,i_k\leq
      n}g(Y_{i_1},\ldots,Y_{i_k})={|\A_n|^{-k}}\tsum_{\phi_1,\ldots,\phi_k\in\A_n}g(f(\phi_1Y),\ldots,f(\phi_kY))\;.
  \end{equation*}  
  If we instead choose $g$ as a function 
  ${g:\xspace_n^{k_n}\rightarrow\mathbb{R}}$ and replace the set of shifts by
  a subset $\A_n$ of a general nice group $\group$, we
  obtain a generalized U-statistic
  \begin{equation*}
    \mfrac{1}{|\A_n|^{k_n}}\myint_{\A_n^{k_n}}g(\phi_1X_n,\ldots,\phi_kX_n)|d\bphi|^{\otimes k}
  \end{equation*}
  for tuples ${\bphi=(\phi_1,\ldots,\phi_{k_n})}$.
  To average over tuples, we must choose $T_n$ as an action of $\group^{k_n}$, and we permit the dimension
  $k_n$ to grow with $n$.   Similarly as elementary U-statistics, these generalized U-statistics
  are asymptotically
  normal under suitable conditions (\cref{corollary:U:statistics}). A variant of this idea is used
  in the proof of \cref{theorem:exchangeable}, to approximate permutations by tuples
  of shifts.
\end{itemize}
Since these generalizations can be used in combination
with each other, we formulate a central
limit theorem (\cref{theorem:CLT:g}) and a Berry-Esseen bound
(\cref{theorem:BE:g}) simultaneously for all three.
The conditions of \cref{theorem:CLT,theorem:BE}---invariance, a moment condition, and
conditional mixing---are still applicable in principle, but become
rather restrictive in the general case, and we introduce the following relaxations:
\begin{itemize}
\item If $T_n$ is an action of $\group^{k_n}$,
  \cref{theorem:CLT} requires invariance under all tuples
  ${\bphi=(\phi_1,\ldots,\phi_{k_n})}$. A simple example shows how 
  strong this assumption is: 
  If $(Y_i)_{i\in\mathbb{Z}}$ is stationary, the U-statistic above
  involves the random field ${(g(Y_{i_1},\ldots,Y_{i_{k_n}}))_{i_1,\ldots,i_{k_n}\in\mathbb{Z}}}$,
  but this field is not invariant under shifts in
  $\mathbb{Z}^k$. To obtain a more suitable condition, we observe
  that the field is invariant under ``diagonal'' shifts
  \begin{equation*}
    (g(Y_{i_1},\ldots,Y_{i_{k_n}}))\quad\mapsto\quad(g(Y_{i_1+j},\ldots,Y_{i_{k_n}+j}))\qquad\text{
      for }j\in\mathbb{Z}\;.
  \end{equation*}
  More generally, if $(j_1,\ldots,j_n)$ is any fixed tuple, applying
  this tuple as a shift may change the distribution, but the
  shifted field $(g(Y_{i_1+j_1},\ldots,Y_{i_k+j_k}))$ is again invariant
  under diagonal shifts. The notion of variance assumed in this
  section, defined in \eqref{eq:generalized:invariance},
  generalizes this property from $\mathbb{Z}$ to a general
  group $\group$.\\[-.5em]
\item
  Recall that conditional mixing formulates conditions 
  on pairs $(\phi,\phi')$ in $\group$ that are far away from a set
  $G$. For tuples, this condition becomes stronger as $k_n$ grows---loosely 
  speaking because distances are larger in high dimensions.
  The marginal mixing condition defined in \cref{sec:mixing:mar} measures entry-wise
  distances (which tend to be smaller).\\[-.5em]
\item The bound on moments is relaxed to uniform integrability,
  similar to conditions assumed by central limit theorems for
  triangular arrays.
\end{itemize}
Randomization requires an additional condition: To guarantee
convergence, $\mu_n$ must not concentrate on an 
``unrepresentatively small'' part of $\A_n$. \cref{sec:wellspread}
makes that precise.

\subsection{Definitions}
\label{sec:g:definitions}

Let ${0<k_1\leq k_2\leq\ldots}$ be integers.
For each ${n\in\mathbb{N}}$, let $X_n$ be a random element of a standard Borel space $\xspace_n$,
and ${f_n:\xspace_n\rightarrow\mathbb{R}}$ a measurable function. 
If $\group$ is a nice group with Haar measure ${|\argdot|}$, the product
space ${\group^{k_n}}$ is a
nice group with Haar measure ${|\argdot|^{\otimes k_n}}$.
Similarly, if $(\A_n)_n$ is a tempered \Folner sequence in $\group$, so is
${(\A_i^{k_n})_{i\in\mathbb{N}}}$ in $\group^{k_n}$.
To randomize averages, let
$\mu_n$ be a random measure on $\group^{k_n}$ that satisfies
\begin{equation}
  \label{eq:random:measure}
  \text{(i) }\;\mu_n \text{ is }\sigma\text{-finite }
  \qquad
  \text{(ii) }\;\mu_n(\A_n^{k_n})>0\qquad\text{ almost surely.}
\end{equation}
(Formally, we equip the set of $\sigma$-finite measures on $\group^{k_n}$ with the
$\sigma$-algebra generated by the maps ${\mu\mapsto\mu(A)}$, for all Borel sets
${A\subset\group^{k_n}}$.
By a random measure, we mean a random element of this space \citep[e.g.][]{Kallenberg:2001}.)
Let ${T_n:\group^{k_n}\!\times\xspace_n\rightarrow\xspace_n}$
be a measurable action of $\group^{k_n}$, and write
\begin{equation*}
  \boldsymbol{\phi}x\,:=\,T_n(\phi_1,\ldots,\phi_{k_n},x)\quad\text{ for }x\in\xspace_n\text{ and }\boldsymbol {\phi}=(\phi_1,\ldots,\phi_{k_n})\in\group^{k_n}\;.
\end{equation*}
The \kword{diagonal action} associated with $T_n$ consists of all transformations\nolinebreak
\begin{equation}
  \label{eq:diagonal:action}
    (\phi,\ldots,\phi)x=T_n(\phi,\ldots,\phi,x)\quad\text{ for }\phi\in\group\;.
\end{equation}
The notion of invariance assumed in this section is
\begin{equation*}
  T_n((\phi,\ldots,\phi),T_n(\boldsymbol{\psi},X_n))\equdist T_n(\boldsymbol{\psi},X_n)\quad\text{ for every }\phi\in\group\text{ and }\boldsymbol{\psi}\in\group^{k_n}
\end{equation*}
or equivalently, in more concise notation,
\begin{equation}
  \label{eq:generalized:invariance}
  (\phi,\ldots,\phi)\boldsymbol{\psi}X_n\equdist\boldsymbol {\psi}X_n\quad\text{ for every }\phi\in\group\text{ and }\boldsymbol{\psi}\in\group^{k_n}\;.
\end{equation}
That is a stronger requirement than diagonal invariance, but weaker
than $T_n$-invariance.

To define conditioning, we denote by $\sigma_n(\group)$ the $\sigma$-algebra
\begin{equation*}
  \sigma_n(\group):=\braces{A\subset\xspace_n\text{ Borel}\,|\,(\phi,\ldots,\phi)A=A\text{ for all }\phi\in\group}\;,
\end{equation*}
and abbreviate ${\mean[\argdot|\group]:=\mean[\argdot|\sigma_n(\group)]}$ and ${P(\argdot|\group)=P(\argdot|\sigma_n(\group))}$.
We then consider the random, conditionally centered average
\begin{equation}\label{eq:gen:mor}
  \cR_n(f_n,X_n):=\frac{1}{\mu_n(\A_n^{k_n})}\int_{\A_n^{k_n}}f_n(\bphi X_n)\;-\;\mean[f_n(\bphi X_n)|\group]~\mu_n(d\bphi).
\end{equation}
If ${k_n=1}$, and ${\mu_n(\argdot)=|\argdot|}$ for all $n$, and if all
$X_n$ and all $\xspace_n$ are identical, we recover
${\sigma_n(\group)=\sigma(\group)}$ and 
${\cR_n=\cF_n}$.

\subsection{Marginal mixing}\label{sec:mixing:mar}
To formulate a suitable mixing condition, we modify the definitions in \cref{sec:mixing}:
Again consider two elements
${\bphi}$ and ${\bphi'}$ and a subset $G$, now all in ${\group^{k_n}}$.
We measure how close the entries $\phi_i$ and $\phi_k'$ are
to the remaining entries of $\bphi$ or $\bphi'$, or to any entry of vectors in
$G$. To do so, we define the set of ``all other'' entries,
\begin{equation*}
  \mathcal{E}_{i,k}(\bphi,\bphi',G):=\braces{\phi_j|j\neq
    i}\cup\braces{\phi_j'|j\neq k}\cup\braces{\pi_j|\pi\in G,j\leq
    k_n}\;.
\end{equation*}
In terms of the metric $d$ on $\group$, the shortest distance from
$\phi_i$ or $\phi'_k$ to any of these is
\begin{equation}
  \label{def:set:delta}
  \delta_{i,k}(\bphi,\bphi',G):=\inf\braces{d(\braces{\phi_i,\phi'_k},\psi)\,|\,\psi\in\mathcal{E}_{i,k}}\;.
\end{equation}
For the given function $f_n$, we then define the set of events
\begin{equation*}
  \mathcal{C}_{i,k}(t):=
  \bigcup\sigma_{f_n}(\bphi)\otimes
  \sigma_{f_n}(\bphi')\otimes\sigma_{f_n}\!(G)
\end{equation*}
where the union runs over all pairs $(\bphi,\bphi')$ 
and all measurable sets ${G}$ in ${\group^{k_n}}$ with ${\delta_{i,k}(\bphi,\bphi',G)\!\geq \!t}$.
Recall that the conditional mixing coefficient was defined
in terms of the conditional $P(\argdot|\group)$. Using Lindenstrauss' theorem,
the latter can be written as
\begin{equation*}
    P(A|\group)
    \;=\;
    \mean[\,\mathbb{I}\braces{X\!\in\! A}|\group]
    \;=\;
    \lim_{m\rightarrow\infty}\mfrac{1}{|\A_m|}\myint_{\A_m}\mathbb{I}\braces{\phi X\in A}|d\phi|\;.
\end{equation*}
To measure the effect of transforming only by coordinates $i$ and $k$,
we substitute this by
\begin{equation*}
    P_{i,k}(A,A')
    \;:=\;
    \lim_{m\rightarrow\infty}\medmath{\frac{1}{|\A_m|}}\myint_{\A_m}\mathbb{I}\braces{e_{i,\psi} X_n\in A, e_{k,\psi}X_n\in A'}|d\psi|\;,
  \end{equation*}
  where ${e_{i,\psi}:=(e,\dots,e,\psi,e,\dots,e)}$ has $k_n$
  dimensions and $\psi$ is the $i$th coordinate.
We then define the \kword{marginal
  mixing coefficient}
\begin{equation*}
\alpha_n(t|\group):=\sup_{i\leq k_n}\sup_{(A,A',B)\in\mathcal{C}_{i,k}(t)}|P(A,A',B|\group)-\mean[P_{i,k}(A,A')\mathbb{I}\braces{X_n\in B}|\group]|\;.
\end{equation*}
Choosing $(k_n,f_n,X_n)$ as $(1,f,X)$ for all $n$ recovers ${\alpha_n(\argdot|\group)=\alpha(\argdot|\group)}$.

  Applying conditional mixing to tuples would measure distance between
  $(\bphi,\bphi')$ and $G$ in the product space metric on
  $\group^{k_n}$. Marginal mixing weakens the condition
  by replacing this metric by \eqref{def:set:delta}. Loosely speaking,
  since
  \eqref{def:set:delta} tends to be
  smaller, the condition ${\delta_{i,k}(\bphi,\bphi',G)\geq t}$ then
  tends to exclude more triples ${(\bphi,\bphi',G)}$ in
  the definition of $\mathcal{C}_{i,k}$ than the metric would, which results in a smaller supremum $\alpha_n$.
This intuition can be turned into a precise statement if
we make definitions comparable, by considering processes of the form ${X_n=f_n(f(\phi_1X),\ldots,f(\phi_{k_n}X))}$:
\begin{proposition}\label{prop_mix_b} Let $X$ be
  $\group$-invariant,
  ${f\in\L_1(X)}$, and set ${\xspace_n=\mathbb{R}^{k_n}}$. Then the conditional
  mixing coefficient of $(f(\phi X))_{\phi\in\group}$ and 
  the marginal mixing coefficient of ${(f_n(f(\phi_1X),\dots,f(\phi_{k_n}X)))_{\bphi\in
      \group^{k_n}}}$ satisfy ${\alpha_n(\argdot|\group)\le \alpha(\argdot|\group)}$.
\end{proposition}

\subsection{Spreading conditions for randomization}
\label{sec:wellspread}

The random measure $\mu_n$ should not
concentrate on a subset of $\A_n^{k_n}$ that is ``too small''.
That is formalized as follows: 
For ${A\in\borel(\group^{2k_n})}$ and any measure $\nu$ on $\group^{k_n}$,
define
\begin{equation*}
  \mathbb{T}_n(A,\nu):=\frac{1}{\nu(\A_n^{k_n})^{2}}
 {\myint_{\A_n^{2k_n}}}\mathbb{I}((\bphi,\bpsi)\in A)\nu(d\bphi)\nu(d\bpsi)\;.
\end{equation*}
Consider the random variable
\begin{equation*}
  \Gamma_n^2(A,\bphi):=\frac{1}{\mathbb{T}_n(A,|\argdot|^{\otimes k_n})\mu_n(\A_n^{k_n})}\int_{\A_n^{k_n}}\mathbb{I}((\bphi,\bpsi)\in
  A))\mu_n(d\bpsi)\;.
\end{equation*}
Informally, one would expect the integrals
\begin{equation*}
\begin{split}
  \frac{1}{\mu_n(\A_n^{k_n})}\int_{\A_n^{k_n}}\Gamma^2_n(A,\bphi)\mu_n(d\bphi)\,=\,\frac{\mathbb{T}_n(A,\mu_n)}{\mathbb{T}_n(A,|\argdot|^{\otimes k_n})}
\end{split}
\end{equation*}
to be bounded if $\mu_n$ spreads out its mass sufficiently.
As $\mu_n$ might be discrete even if the Haar measure is not,
bounds should be formulated only in terms of
``sufficiently large'' sets $A$. We define the family of such sets as 
\begin{equation*}
  {\Sigma}_n:=
  \bigl\{A\in \mathcal{B}(\group^{2k_n})\,\big\vert\, A \text{ is connected and }\, |\text{\rm pr}_{k}(A)|\ge 1 \text{ for all } k\le 2k_n\bigr\}\;,
\end{equation*}
where $\text{pr}_k$ denotes projection on the $k$th coordinate.
A weak notion of boundedness suffices for asymptotic normality:
We call the sequence $(\mu_n)$ \kword{well-spread}
if the variables $\Gamma_n^2$ are uniformly integrable for large sets,
\begin{equation*}
\sup_n\sup_{A\in \Sigma_n}\Big\|\frac{1}{\mu_n(\A_n^{k_n})}\int_{\A_n^{k_n}}\Gamma_n^2(A,\phi)\mathbb{I}(|\Gamma_n^2(A,\phi)|\ge \beta)d\mu_n(\phi)\Big\|_{1}\,\xrightarrow{\beta \rightarrow \infty}\, 0\;.
\end{equation*}
A Berry-Esseen bound requires a stricter bound and a fourth-order
condition: We similarly define
\begin{equation*}
  \mathbb{T}^*_n(A,\nu) :=
  \frac{1}{\nu(\A_n^{k_n})^4}\int_{\A_n^{4k_n}}\mathbb{I}((\bphi_1,\bphi_2,\bphi_3,\bphi
  _4)\in A)\nu^{\otimes 4}(d\bphi_1,d\bphi_2,d\bphi_3,d\bphi_4)\;,
\end{equation*}
now for subset ${A}$ of and a measure $\nu$ on
$\group^{4k_n}$, and 
\begin{equation*}
  {\Sigma}^*_n:=
  \bigl\{A\in \mathcal{B}(\group^{4k_n})\,\big\vert\, A \text{ is connected and }\, |\text{\rm pr}_{k}(A)|\ge 1 \text{ for all } k\le 4k_n\bigr\}\;.
\end{equation*}
We call $(\mu_n)$ \kword{strongly well-spread} if 
\begin{equation*}
  \mathcal{S}:=\sup_{n}{\mathcal{S}^n}<\infty
  \qquad\text{ where }\qquad
  \mathcal{S}^n :=\sup_{A\in \Sigma_n^*}\Big\|
  \frac{\mathbb{T}^*_n(A,\mu_n)}{\mathbb{T}^*_n(A,|\argdot|^{\otimes k_n})}\Big\|_{1}
\;,
\end{equation*}
with \kword{spreading coefficient} $\mathcal{S}$. 
Since the existence of higher moments implies uniform integrability, strongly well-spread
implies well-spread.
Either condition can be applied to a single random measure $\mu$,
by setting ${\mu_n:=\mu}$ for all $n$.\nolinebreak
\begin{examples}
\exitem{} Let $\Pi$ be a Poisson point process on $\group^k$, for some ${k\in\mathbb{N}}$. Then the random measure
${\mu(\argdot):=|\Pi\cap\argdot|^{\otimes k}}$ is strongly well-spread if
\begin{equation*}
  \sup_{A\in\borel(\group^k),\,|A|^{\otimes k}<\infty}\frac{\mean\bigl[|\Pi\cap A|^{\otimes k}\bigr]}{|A|^{\otimes k}}\;<\;\infty\;.
\end{equation*}
\exitem{} Let $\group$ be discrete.
For each $n$, let $\Pi_n$ be a point process on $\group^{k_n}$ with
\begin{equation*}
  \Pi_n\cap\A_n^{k_n}\,\big\vert\, (|\Pi_n\cap\A_n^{k_n}|=m)
  \quad\equdist\quad
  (\Phi_1,\ldots,\Phi_m)\quad\text{ for all }m\in\mathbb{N}\;,
\end{equation*}
where the $\Phi_i$ are drawn uniformly with or without replacement from $\A_n^{k_n}$.
The sequence defined by ${\mu_n(\argdot):=|\Pi_n\cap\argdot|^{\otimes k_n}}$ is strongly
well-spread.
\end{examples}

\subsection{Results}
If the dimension $k_n$ grows with $n$, we must quantify how much $f_n$
changes with $n$:
For ${p>0}$ and ${i\leq k_n}$, define
\begin{equation*}
  c_{i,p}(f_n)\;:=\;
  \sup_{\psi\in\group,\bphi\in\group^{k_n}}
  \tfrac{1}{2}\,
  \|\, f_n\circ\bphi-f_n\circ(e,\ldots,e,\psi,e,\ldots,e)\bphi\,\|_{p}
\end{equation*}
where $\psi$ is the $i$th coordinate.
Hypotheses \eqref{H1} and \eqref{H2} are then replaced by one of the following conditions:
Either
\begin{align}
  \label{H1:g}
  \text{\rm (i)}\;&
  \sup_n \alpha_n(K|\group)=0
  \qquad\qquad\quad\text{\rm (ii)}\;
  \sup_n\tsum\displaylimits_{i\leq k_n}c_{i,2}(f_n)<\infty\\
  \text{\rm (iii)}&\;
  \bigl(f_{n}(\boldsymbol{\phi} X_n)^2\bigr)_{n\in\mathbb{N},\phi\in\group^{k_n}}
  \text{ is uniformly integrable }\nonumber
\end{align}
holds for some ${K\in\mathbb{N}}$, or
\begin{align}
  \label{H2:g}
  \text{\rm(i)}\;&
  \sup_n\medint\int_{\group} \alpha_n(d(e,\phi)|\group)^{\frac{\epsilon}{2+\epsilon}}|d\phi|<\infty
  \qquad
  \text{\rm(ii)}\;
  \sup_n\tsum\displaylimits_{i\leq k_n}\!c_{i,2+\varepsilon}(f_n)<\infty\\
  \text{\rm(iii)}&\;
  \bigl(f_{n}(\boldsymbol{\phi} X_n)^{2+\varepsilon}\bigr)_{n\in\mathbb{N},\phi\in\group^{k_n}}
  \text{ is uniformly integrable}\nonumber
\end{align}
holds for some ${\varepsilon>0}$. In either case,
(iii) implies (ii) if the sequence $(k_n)$ is bounded.
To assemble the asymptotic variance, set
\begin{equation*}
  \hF_{\infty,i}(\psi):=\lim_{m\rightarrow\infty}
  \mfrac{1}{|\A_m|^{k_n}}  \!\!\!
  \medint\int_{\bphi\in\A_m^{k_n}}\limits\!\!\!\! f_n((\bphi_1,\ldots,\bphi_{i-1},\psi,\bphi_{i+1},\ldots,\bphi_{k_n})X_n)
   |d\bphi|^{\otimes k_n-1}\;.
\end{equation*}
Let $\mu_n^i$ be the $i$th coordinate marginal of $\mu_n$, scaled to $\mu_n^i(\A_n)=\sqrt{|\A_n|}$,
\begin{equation*}
  \mu_n^{i}(\argdot):=\mfrac{\sqrt{|\A_n|}}{\mu_n(\A_n^{k_n})}\mu_n(\A_n,\dots,\A_n,\argdot,\A_n,\dots,\A_n)\;,
\end{equation*}
and set
\begin{equation*}
  \widehat{\eta}_{nm}:=\tsum\displaylimits_{i,j\leq k_n}{\medint\iint_{\;\;\phi\A_n,\psi\in \B_m(\phi)}\limits}
  \mean[\,\hF_{\infty,i}(e)\hF_{\infty,j}(\phi^{-1}\psi)\,|\,\group]\,
  \mu_n^i(d\phi)\mu_n^j(d\psi)\;.
\end{equation*}
The central limit theorem then takes the following form:
\begin{theorem}
  \label{theorem:CLT:g}
  Let $(X_n)$ be invariant in the sense of \eqref{eq:generalized:invariance} for each $n$,
  and let ${(\mu_n)}$ be well-spread and independent of ${(X_n)}$.
  Assume either condition \eqref{H1:g} or \eqref{H2:g} holds. If
  ${k_n=o(|\A_n|^{\frac{1}{4}})}$,
  and if the limits
  \begin{equation}
    \label{eq:variance:g}
    \widehat{\eta}_{nm}\xrightarrow{\;\;p\;\;}\eta_m\quad\text{ as }n\rightarrow\infty
    \qquad\text{ and }\qquad
    \eta_{m}\xrightarrow{\;\mathbf{L}_2\;}\eta\quad\text{ as }m\rightarrow\infty
  \end{equation}
  exist, then
  \begin{equation*}
    \sqrt{|\A_n|}\,\hF_n(f_{n},X_n)\xrightarrow{\;\;\text{d}\;\;}{\eta}Z\qquad\text{ as }n\rightarrow\infty\;,
  \end{equation*}
  for a standard normal variable $Z$ that is independent of $\eta$.
\end{theorem}
The Berry-Esseen bound in \cref{theorem:BE} generalizes
similarly:
\begin{theorem}
  \label{theorem:BE:g}
  Assume the conditions of \cref{theorem:CLT:g} hold, require that ${(\mu_n)}$ is
  strongly well-spread, and define $\eta$ as in \eqref{eq:variance:g}.
  If condition \eqref{H1:g} holds for some ${K\in\mathbb{N}}$,
  \begin{align*}
   \dW\Bigl(\mfrac{\sqrt{|\A_n|}}{\eta}\,\widehat{\mathbb{F}}_n(f_n,X_n),Z\Bigr)
   \;\leq\;
   \kappa\frac{k_n^2(\mathcal{S}^n\wedge 1)((\sum_{i}c_{i,4})^3\wedge 1)|\B_{K}|^2}{\sqrt{|\A_n|}}
   +
   \Bigl\|\frac{\widehat{\eta}_{n,K}^{\,2}-\eta^2}{\eta^2}\Bigr\|\;,
  \end{align*}
  for a positive constant $\kappa$.
  If \eqref{H2:g} holds instead, set
  \begin{equation*}
    \mathcal{R}_n(b):=\tsum_{t\geq b}|\B_{t+1}\setminus\B_t|\alpha_n(t|\group)^\frac{\varepsilon}{2+\varepsilon}
      \quad\text{ for }b\in\mathbb{N}\;,
  \end{equation*}
  and fix any sequence ${0<b_1<b_2<\ldots}$ of integers. Then
  \begin{align*}
    \dW\Bigl(\mfrac{\sqrt{|\A_n|}}{\eta}\,\widehat{\mathbb{F}}_n(f_n,X_n),Z\Bigr)
    \;&\leq\;
    \kappa \mathcal{R}_n(b_n)(\msum_ic_{i,2+\epsilon})^2(\mathcal{S}^n\!\!\wedge\! 1)
    +
    \Bigl\|\frac{\widehat{\eta}_{n,b_n}^{\,2}-\eta^2}{\eta^2}\Bigr\|\\
    &+\;
    \kappa((\msum_{i}c_{i,4+2\epsilon})^3\!\wedge\! 1)(\mathcal{S}^n\!\!\wedge\! 1)\mathcal{R}_n(0)
    \frac{k_n^2|\B_{b_n}|}{\sqrt{|\A_n|}}
  \end{align*}
  for a positive constant ${\kappa}$.
\end{theorem}
If we choose ${(k_n, X_n,\hF_n)}$ as ${(1,X,\mathbb{F}_n)}$ for all $n$,
the conditions specialize to \eqref{H1} and \eqref{H2}, and the results to
\cref{theorem:CLT,theorem:BE}.

\subsection{Generalized U-statistics}\label{sec:u-stat}
The generalized notion of invariance defined in \eqref{eq:generalized:invariance}
allows us to formulate a useful generalization of U-statistics,
denoted $X_{\bpsi}$ in the next result. Substituting these into
\cref{theorem:CLT:g} shows they are asymptotically normal:
\begin{corollary}
  \label{corollary:U:statistics}
  Consider a $\group$-invariant random element $Y$ of $\xspace$, a function ${h:\xspace^k\rightarrow\mathbb{R}}$, and set
  ${X_{\bpsi}:=h(\psi_1 Y,\ldots,\psi_k Y)}$ for ${\bpsi\in\group^k}$. Suppose there is an ${\varepsilon>0}$ for which
  the conditional mixing coefficient of $Y$ satisfies ${\int_{\group}\alpha^{\frac{\varepsilon}{2+\varepsilon}}(d(e,\phi)|\group)|d\phi|<\infty}$,
  and ${(X_{\bpsi}^{2+\varepsilon})_{\bpsi\in\group^k}}$ is uniformly integrable. Then
  \begin{equation*}
    |\A_n|^{\frac{1}{2}-k}\myint_{\A_n^{k}}(X_{\bpsi}-\mathbb{E}[X_{\bpsi}|\group^k])|d\bpsi|^{\otimes k}\;\darrow\;\eta Z
  \end{equation*}
  for ${\eta\,\condind\, Z}$ and ${Z\sim N(0,1)}$. If we
  denote
    \begin{equation*}
    H_i(\phi)\!:=\!\lim_{m\rightarrow
      \infty}\tfrac{1}{|\A_m|^{k-1}}\!\!\!\!\int\limits_{\A_m^{k-1}}\!\!\!\!X_{\psi_1,\dots,\psi_{i-1},\phi,\psi_{i+1},\dots,\psi_{k_n}}|d\psi_1|\!\cdots\!|d\psi_{i-1}||d\psi_{i+1}|
    \!\cdots\!|d\psi_k|\,,
  \end{equation*}
  the asymptotic variance is ${\eta^2=\sum_{i,j\le k}\int_{\group}{\rm Cov}[H_i(e),H_j(\phi)|\group]|d\phi|}$.
\end{corollary}
To clarify the relationship to U-statistics, recall that
a U-statistic for an i.i.d.\ sequence ${(Y_i)_{i\in\mathbb{Z}}}$ is
usually defined in one of two ways, namely\nolinebreak
\begin{equation*}
  U_n:={\textstyle\binom{n}{k}^{-1}}\!\!\msum\displaylimits_{\phi\in\mathbb{S}_n}h(Y_{\phi(1)},\ldots,Y_{\phi(k)})
  \;\;\text{ or }\;\;
  V_n:=\mfrac{1}{n^k}\!\!\msum\displaylimits_{i_1,\ldots,i_k\leq n}h(Y_{i_1},\ldots,Y_{i_k})\;.
\end{equation*}
The definitions are equivalent, in the sense that
${\sqrt{n}(U_n-V_n)\rightarrow 0}$ in probability \citep{Serfling:1980}.
The corollary shows ${n^{-1/2}(V_n-\mean[V_n])\darrow\eta Z}$, if we choose ${\group}$ as ${\mathbb{Z}}$ and ${\A_n}$ as ${\braces{1,\ldots,n}}$.
Although ${h(Y_{i_1},\ldots,Y_{i_k})}$ the relaxed invariance \eqref{eq:generalized:invariance},
it is not $\mathbb{Z}^k$-invariant, since arbitrary shifts may break independence of $(Y_i)$ by duplicating indices.

\section{Concentration}
\label{sec:concentration}

The theorems above show that certain asymptotic properties of i.i.d.\ processes generalize to symmetric random objects.
We show next that certain finite-sample properties generalize similarly. We use the definitions of \cref{sec:g}, but somewhat
restrict the spaces and functions involved:
Fix two Borel spaces $\xspace$ and $\yspace$, two sequences $(f_n)$ and $(g_n)$ of measurable functions
${f_n:\xspace\rightarrow\yspace}$ and
${g_n:\yspace^{k_n}\rightarrow\xspace}$, and let $(X_n)$ be a sequence of $\group$-invariant
random elements of $\xspace$. We consider concentration for quantities of
the form ${g_n(f_n(\phi_1X_n),\ldots,f_n(\phi_{k_n}X_n))}$. To this
end, define
\begin{equation*}
  Y^n:=(Y_\phi^n)_{\phi\in\group}\quad\text{ where }\quad Y^n_\phi:=f_n(\phi X_n)\;.
\end{equation*}
That implies ${(Y^n_{\phi})\equdist(Y^n_{\psi\phi})}$ for
${\psi\in\group}$.
We again work with (conditionally) centered
averages:
For ${\bphi=(\phi_1,\ldots,\phi_{k_n})}$, set
\begin{equation*}
  h_n(\boldsymbol{\phi}X_n):=g_n(Y^n_{\phi_1},\dots,Y^n_{\phi_{k_n}})-\mean[g_n(Y^n_{\phi_1},\dots,Y^n_{\phi_{k_n}})|\group]\quad\text{
    for }\bphi\in\group^{k_n}\;.
\end{equation*}
The average ${\widehat{\mathbb{F}}_n}$, as defined in the previous
section, is then 
\begin{equation*}
  \widehat{\mathbb{F}}_n(h_n,X_n)=\frac{1}{\mu_n(\A_n^{k_n})}\int_{\A_n^{k_n}}h_n(\boldsymbol{\phi}X_n)\mu_n(d\bphi)\;.
\end{equation*}
\begin{example}
  \exitem{example:concentration:stationary}
  Let ${X:=(X_i)_{i\in\mathbb{Z}}}$ be a stationary, real-valued
  process, so ${\xspace=\mathbb{R}^\mathbb{Z}}$. Choose all $f_n$ as the coordinate function
  ${(x_i)\mapsto x_0}$ at index $0$.
  If $(g_n)$ is any sequence of measurable functions
  ${g_n:\mathbb{R}^k\rightarrow \mathbb{R}}$,
  we obtain random fields 
  ${Y^n=(g_n(X_{i_1},\dots,X_{i_k}))_{i_1,\dots,i_k\in \mathbb{Z}}}$,
  and
  \begin{equation*}
    \widehat{\mathbb{F}}_n(h_n,X_n)=\mfrac{1}{n^k}\msum_{i_1,\dots,i_k\le
      n}\bigl(g_n(X_{i_1},\dots, X_{i_k})-\mathbb{E}(g_n(X_{i_1},\dots,
    X_{i_k}))\bigr)\;.
  \end{equation*}
  As we had already observed in the introduction of \cref{sec:g},
  each $Y^n$ is invariant under
  the diagonal action
  ${Y^n\mapsto(g_n(X_{i_1+\phi},\dots,X_{i_k+\phi}))}$, for ${\phi\in\mathbb{Z}}$,
  since $X$ is stationary.
\end{example}

\def\y{\mathbf{y}}
A function ${f:\yspace^{k}\rightarrow\mathbb{R}}$ is \kword{self-bounded} if
there are constants ${\delta_1,\ldots,\delta_{k}}$, the \kword{self-bounding coefficients}, such that
\begin{equation*}
  \mfrac{1}{2}|f(\x)-f(\x')|\;\leq\;\msum_{i\leq k}\delta_i\mathbb{I}\braces{x_i\neq x_i'}
  \qquad\text{ for all }\x,\x'\in\yspace^{k}\;,
\end{equation*}
see \eg \citep{Boucheron:Lugosi:Massart:2013:1}.
We call $f$ \kword{uniformly $\L_1$-continuous} in $\group$ if
\begin{equation*}
  \sup_{\substack{\bphi,\bpsi\in\group^{k}\\d(\phi_i,\psi_i)\leq\epsilon\text{ for }i\leq k}}
  \|f(\boldsymbol{\phi}X_n)-f(\boldsymbol{\psi}X_n)\|_{1}\;\longrightarrow\;0\;\quad\text{ as }\epsilon\rightarrow 0\;.
\end{equation*}
We measure interactions within a process ${Y=(Y_{\phi})_{\phi\in\group}}$ as follows:
Write $\mathcal{L}$ for the law of a random variable,
${\|\argdot\|_{\text{\rm\tiny TV}}}$ for the total variation norm, and
abbreviate ${Y_{\neq\phi}:=(Y_{\psi})_{\psi\neq\phi}}$.
If $\group$ is countable, define
\begin{equation*}
  \Lambda[Y]:=\tsum_{\phi\in\group\setminus\braces{e}}\!
  \sup_{\substack{\x,\y\in\yspace^{\group}\\ \x_{\neq\phi}=\y_{\neq\phi}}}\!
  \|\mathcal{L}(Y_e|Y_{\neq e}=\x_{\neq e})-\mathcal{L}(Y_e|Y_{\neq
    e}=\y_{\neq e})\|_{\text{\rm\tiny TV}}
\end{equation*}
If $\group$ is uncountable, we discretize:
For ${\epsilon>0}$, a set 
${C\subset\group}$ is an \kword{$\epsilon$-net} if
\begin{equation*}
  \text{(i)}\;
  e\in C
  \quad
  \text{(ii)}\;
  d(\phi,\phi')\geq\epsilon\;\text{ for }\phi,\phi'\in C\text{ distinct}
  \quad
  \text{(iii)}\;
  \medcup\displaylimits_{\phi\in C}\!\!B_{\epsilon}(\phi)=\group\;.
\end{equation*}
A decreasing sequence of nets is a sequence
$(C_{i})_{i\in\mathbb{N}}$, where $C_i$ is an $\epsilon_i$-net
and ${\epsilon_i\rightarrow
  0}$. Define
\begin{equation*}
  \rho[Y]\;:=\;\sup\Big(1-\lim_{i\rightarrow\infty}\frac{1-\Lambda[(Y_{\phi})_{\phi\in C_{i}}]}{|\B_{\epsilon_i}|}\Big)
\end{equation*}
where the supremum is taken over all decreasing sequences of
nets for which the limit on the right exists.
Discretizing continuous processes on nets is a standard tool in
the context of concentration inequalities \citep[see][Chapter 13]{Boucheron:Lugosi:Massart:2013:1}.
Note that ${\rho=\Lambda}$ if the group is discrete.
For ${\group=\mathbb{Z}}$, it is known as the
Dobrushin interdependence coefficient \citep{Stroock:Zegarlinski:1992:1}.
A continuous example is a Markov process
${Y=(Y_t)_{t\in\mathbb{R}}}$ on ${\group=\mathbb{R}}$,
where
\begin{equation*}
  \rho[Y]=\lim_{t\rightarrow\infty}\frac{1}{t}\sup_{x,y\in\mathbb{R}}\|\Law(Y_0|Y_t=x)-\Law(Y_0|Y_t=y)\|_{\text{\rm\tiny TV}}\;.
\end{equation*}

\begin{theorem} 
  \label{theorem:concentration}
  Let ${(\A_n)}$ be a tempered \Folner sequence in $\group$, let
  $(c_i)$ be the self-bounding coefficients of $h_n$, and require that $(h_n)$ is uniformly $\L_1$-continuous in $\group$.
  Define
  \begin{equation*}
    \tau_n:=
    \sup_{j\le k_n}\sup_{B\in \mathcal{B}( \group)}
    \frac{|\A_n|\mu_n(\A_n^{j-1}\!\times\! (B\cap\A_n)\!\times\! \A_n^{k_n-j}\,\big\vert\,\A_n^{k_n})
    }{|B\cap\A_n|} \;.
  \end{equation*}
  Then
  \begin{equation*}
    \mathbb{P}\bigl(\widehat{\mathbb{F}}_n(h_n,X_n)\geq t\bigr)\;\leq\;
    2\mathbb{E}\Big(\exp\Bigl(-\frac{(1-\rho[Y^n])|\A_n|}{(\sum_{i\leq k_n}c_i)^2\tau_n^2}\,t^2\Bigr)\Big)
    \qquad\text{ for all }t>0\;.
  \end{equation*}
\end{theorem}

The coefficients $\tau_n$ are only required if averages are randomized.
If $(\mu_n)$ is non-random, the statement can simplify considerably. For example:
\begin{corollary}\label{cor:con:con}
  If ${\mu_n=|\argdot|^{\otimes k_n}}$ almost surely for each $n$, then
  \begin{equation*}
    \mathbb{P}\bigl(\widehat{\mathbb{F}}_n(h_n,X_n)\geq t\bigr)\;\leq\;
    2\exp\Bigl(-\frac{(1-\rho_n)|\A_n|}{(\sum_{i\leq k_n}c_i)^2}t^2\Bigr)
    \quad\text{ for }t>0\text{ and }n\in\mathbb{N}.
  \end{equation*}
\end{corollary}

\begin{example}
  For illustration, compare to the i.i.d.\ case: Choose $X$, $f$ and
  $g$ as in \cref{example:concentration:stationary} at the beginning
  of this section, and assume additionally that the sequence $X$ is i.i.d. If $(c_i)$ are the self-bounding
  coefficients of $g$, the corollary shows
\begin{equation*}
  \mathbb{P}\bigl(\widehat{\mathbb{F}}_n(h,X)\geq t\bigr)\;\leq\;
  2\exp\Bigl(-\frac{nt^2}{(\sum_{i\leq k}c_i)^2}\Bigr)
  \quad\text{ for }t>0\text{ and }n\in\mathbb{N}\;,
\end{equation*}
which is a version of McDiarmid's inequality.
\end{example}

\cref{theorem:concentration} implicitly assumes fairly strong mixing:
If $\group$ is discrete, for example, then ${\alpha(n|\group)\leq
  c_1\alpha(n)\leq c_2\Lambda[X]}$ for some positive constants $c_1$ and $c_2$ and all ${n\in\mathbb{N}}$. The mixing condition is hence no weaker
than that required in the asymptotic case, and conditioning on
$X_{\neq e}$
in the definition of $\Lambda[X]$ means it is typically stronger.

\section{Approximation by subsets of transformations}
\label{sec:discretization}
According to \cref{theorem:CLT:g}, $\mathbb{F}_n$ may be computed
using only a subset of $\A_n$. 
We briefly discuss a few cases in more detail. First suppose we ``factor out'' a compact subgroup $\mathbb{K}$ of
$\group$ to obtain a subgroup $\mathbb{H}$, and then compute $\mathbb{F}_n$ using a \Folner sequence of $\mathbb{H}$.
For exchangeable sequences, factoring $\mathbb{S}_k$ out of $\mathbb{S}_\infty$ amounts to including only every $k$th
observation in the sample average, so rates slow by a constant. The general behavior is similar:
\begin{proposition}
  \label{result:discretization}
  Let $\group$ be generated by the union of a non-compact group $\mathbb{H}$
  and a compact group $\mathbb{K}$, and
  let ${(\A^{\mathbb{H}}_n)}$ be a \Folner sequence in $\mathbb{H}$. Then ${\A_n:=\A_n^{\mathbb{H}}\mathbb{K}}$ is a \Folner sequence in $\group$.
  If $X$ is $\group$-invariant, and ${f\in\L_2(X)}$ satisfies \eqref{H2} with respect to $\group$,
   there exist random variables ${\eta,\eta_{_\mathbb{H}}\in\L_2(X)}$
   and an independent variable ${Z\sim N(0,1)}$ such that
  \begin{equation*}
    \begin{split}
    \tfrac{1}{\sqrt{|\A_n^{\mathbb{H}}|}}\myint_{\A_n^{\mathbb{H}}}\bigl(f(\phi X)-\mean[f(X)|\group]\bigr)|d\phi|
    \;\xrightarrow{\;\tiny d\;}&\;\eta_{_\mathbb{H}}Z\\
    \text{and}\qquad\qquad\qquad\tfrac{1}{\sqrt{|\A_n|}}\myint_{\A_n}\bigl(f(\phi X)-\mean[f(X)|\group]\bigr)|d\phi|
    \;\xrightarrow{\;\tiny d\;}&\;\eta Z\;.\qquad\qquad\qquad
    \end{split}
  \end{equation*}
  where $\eta$ is defined by \eqref{eq:variance:g}. 
  The ratio
  ${\beta:=\sqrt{|\mathbb{K}|}\,\frac{\eta_{_\mathbb{H}}}{\eta}}$, and
  hence $\eta_{\mathbb{H}}$,
  is given by
  \begin{equation*}
    \beta^2-1=\frac{1}{\eta^2}
    \myint_{\mathbb{H}}\myint_{\mathbb{K}}\mean[f(X)(f(\phi X)-f(\psi\phi X))|\group]|d\psi||d\phi|
    \quad\text{a.s.}
  \end{equation*}
\end{proposition}
For example, let ${X=(X_t)_{t\in\mathbb{R}^r}}$ be a continuous random
field that is both shift- and rotation invariant.
Thus, ${\group=\mathbb{R}^r\times\mathbb{O}_r}$, where $\mathbb{O}_r$ is the (compact) orthogonal group of order $r$.
Factoring out $\mathbb{O}_r$ means we average only over shifts. Convergence then slows by a factor
\begin{equation}
  \label{beta:random:field}
  \beta^2-1
  =
  \mfrac{1}{\eta^2}\,\mean\bigl[f(X)\myint_{\mathbb{R}^r}\myint_{\mathbb{O}_r}(f(X+\phi)-f(\theta X+\phi))|d\theta||d\phi|\bigr]\;.
\end{equation}
One might also discretize $\A_n$ (e.g.\ to avoid integration), or subsample it. For example: A tempered \Folner sequence in
${\mathbb{R}^r\times\mathbb{O}_r}$ is given by ${([-n,n]^r\times\mathbb{O}_r)_n}$ \citep{Loeh:2017}.
If we discretize $[-n,n]^r$ deterministically, and $\mathbb{O}_r$ at random, we obtain:\nolinebreak
\begin{corollary}
  \label{corollary:data:augmentation}
  Let ${X=(X_t)_{t\in\mathbb{R}^r}}$ be a random field invariant under
  rotations and translations of $\mathbb{R}^r$, and require \eqref{H2}.
  Fix ${m\in\mathbb{N}}$. For ${z\in\mathbb{Z}^r}$, let ${\Theta_1^z,\ldots,\Theta_m^z}$ be independent,
  uniform random elements of $\mathbb{O}_r $.
  Then
  \begin{equation*}
    \frac{1}{m\sqrt{(2n)^r}}
    \sum_{\substack{z\in\braces{-n,\ldots,n}^r\\j\leq m}}\bigl(
    f(\Theta^z_j(X+z))-\mean[f(X)|\group]
    \bigr)
    \;\xrightarrow{\;\tiny d\;}\;\eta_m Z
  \end{equation*}
  as ${n\rightarrow\infty}$,
  for an almost surely finite random variable $\eta_m\,\condind\,Z$.
  Relative to $\mathbb{F}_n$ defined by integration over the entire set ${[-n,n]^r\times\mathbb{O}_r}$,
  convergence
  slows by a coefficient ${\beta_m^2-1=(\beta^2-1)/(2m^2\eta_m^2)}$, where $\beta$ is given by
  \eqref{beta:random:field}.
\end{corollary}
If the random rotations $\Theta_j^z$ are not independent---for example, if one generates $m$ rotations once and uses them
repeatedly---the rate may slow.

\section{Applications I: Exchangeable structures}
\label{sec:applications}
One of the most common distributional symmetries is permutation
invariance, often referred to as exchangeability.
It can broadly be categorized into three types:
\kword{Finite exchangeability}
is invariance under $\mathbb{S}_n$, for some fixed ${n\in\mathbb{N}}$
\citep{Kallenberg:2005}.
This is an example of invariance under a compact group, and has no asymptotic theory.
Countably infinite
exchangeability, or henceforth simply \kword{exchangeability}, is
invariance under $\mathbb{S}_{\infty}$. This type is common in
statistics and probability. By \kword{uncountable exchangeability}, we
refer to invariance under permutation groups of uncountable sets. Such groups are not
nice, and Lindenstrauss' theorem is not applicable, but
\cref{sec:graphex} gives an example where reduction to our results is possible.\\[-.5em]

\subsection{Exchangeability}
\label{sec:exchangeable}

The next theorem adapts our results to exchangeable 
structures, including the examples in \cref{tab:exchangeability}.
In this case, the mixing condition can be eliminated.
\begin{theorem}
  \label{theorem:exchangeable}
  Let $X$ be a random element of a standard Borel space $\xspace$, and invariant under a measurable
  action of $\mathbb{S}_\infty$. Let $f$ be a function
  satisfying ${\mean[f(X)^2]<\infty}$ and
  \begin{equation}
    \label{eq:condition:limsup}
    \tsum_{i\in\mathbb{N}}{\textstyle\limsup_{j}}\|f(X)-f(\tau_{ij}X)\|_{2}\;<\;\infty\;,
  \end{equation}
  where $\tau_{ij}$ denotes the transposition of $i$ and $j$. As ${n\rightarrow\infty}$,
  \begin{equation}
    \label{eq:CLT:exchangeable}
    \sqrt{n}\,\cF_n(f,X)\;
    \;=\;
    \sqrt{n}\Bigl(\tfrac{1}{n!}\tsum_{\phi\in\mathbb{S}_n}f(\phi X)-\mean[f(X)|\mathbb{S}_\infty]\Bigr)
    \;\xrightarrow{\;d\;}\;
    \eta Z
  \end{equation}
  where ${Z\sim N(0,1)}$ is independent of $\eta$. Define
  \begin{equation*}
    \mathbb{F}^i(\phi):=\lim_{n\rightarrow \infty}\tfrac{1}{|\mathbb{S}^{i}_{n}|}\tsum_{\phi'\in\mathbb{S}^{i}_n}f(\phi'\phi X)
    \quad\text{ where }\quad\mathbb{S}_n^i:=\braces{\phi\in\mathbb{S}_n\,\vert\,\phi(i)=i}\;.
  \end{equation*}
  The asymptotic variance satisfies
  \begin{equation*}
    \eta^2=\tsum_{i,j\in\mathbb{N}}\text{\rm Cov}\big[\mathbb{F}^i(e),\mathbb{F}^j(\tau_{ij})\big|\mathbb{S}_{\infty}\big]
    <\infty \quad\text{a.s.}
  \end{equation*}
  If in addition $\mathbb{E}[f(X)^4/\eta^4]<\infty$  and
  ${\tsum_{i\in\mathbb{N}}\limsup_{j}\big\|\frac{f(X)-f(\tau_{ij}X)}{\eta}\big\|_{4}<\infty}$,
  the Wasserstein distance to the limit is
  \begin{equation*}
    \dW\bigl(\tfrac{\sqrt{n}}{\eta}\cF_n(f,X),Z\bigr)=O\Bigl(
    \min_{k\in \mathbb{N}}\Big[\tfrac{k^2}{\sqrt{n}}+\tsum_{i>k}\max\Big(\limsup_{j}\big\|\tfrac{f(X)-f(\tau_{ij}X)}{\eta}\big\|_{4},1\big)\Big]\Bigr)\;.
  \end{equation*}
\end{theorem}
Typically, $X$ is of the form ${(X_{t})_{t\in T}}$ for some countable set $T$,
and permutations act on $X$ by acting on $T$. If $f$ depends only on a
finite number of these indices---e.g. if $X$ is a random matrix and $f$ a
function of a finite number of entries---\eqref{eq:condition:limsup} always holds, although this
condition is far from necessary.
If $X$ is conditionally mixing for $f$, the result can be deduced from
\cref{theorem:CLT}. The proof of the general case
defines surrogate variables 
${X_n:=(f(\tau_{1,i_1}\circ\cdots\circ\tau_{k_nn,i_{k_n}}X))_{i_1,\ldots,i_{k_n}}}$
for a suitable sequence $(k_n)$, and applies an idea similar to the generalized U-statistics
of \cref{corollary:U:statistics}.\\[-.7em]

\textsc{Remark}. (a)
Our definition of exchangeability as an arbitrary action of
$\mathbb{S}_{\infty}$
permits trivial cases, for example:
Mapping each ${\phi\in\mathbb{S}_\infty}$ to the identity map of
$\xspace$ is a valid action.
It makes all distributions exchangeable, point masses are ergodic, and ${\mathbb{F}_n(f,X)=\mean[f(X)|\group]=f(X)}$ for all $n$.
(b) Exchangeability can also be defined as invariance under the group
  $\mathbb{S}(\mathbb{N})$ of all bijections of $\mathbb{N}$, as is
  often done in Bayesian statistics. This
  definition is equivalent to ours, in the sense that
  any measurable action of $\mathbb{S}(\mathbb{N})$ and its restriction to
  ${\mathbb{S}_{\infty}\subset\mathbb{S}(\mathbb{N})}$ have the same
  invariant and ergodic measures \citep{Maitra:1977}, but the group
  $\mathbb{S}(\mathbb{N})$ is not nice.

\begin{table}[t]
\begin{tabular}{lll}
random structure $X$ & ergodic structures & CLT \eqref{eq:CLT:exchangeable} due to \\
\midrule
exchangeable sequence \citep{Kallenberg:2005}  & i.i.d. sequences  & H. B\"uhlmann \citep{Buehlmann:1958:1}\\
exchangeable partition  \citep{Pitman:2006} & ``paint-box'' distributions &   \\
exchangeable graph  \citep{Diaconis:Janson:2007} & graphon distributions  &
Bickel et al.\ \citep{Bickel:Chen:Levina:2011:1}, Ambroise and Matias \citep{Ambroise:Matias:2012:1}\\
jointly exch. array  \citep{Kallenberg:2005} & dissociated arrays  &
Eagleson and Weber \citep{Eagleson:Weber:1978:1}, Davezies et al. \citep{Davezies:Haultfeuille:Guyonvarch:2021}\\
separately exch. array  \citep{Kallenberg:2005} & dissociated arrays  &  \\
\midrule\\
\end{tabular}
\caption{}
\label{tab:exchangeability}
\end{table}

\subsection{Jointly exchangeable arrays}
\label{sec:arrays}
We discuss one class of examples in \cref{tab:exchangeability}, the
jointly exchangeable arrays, in more detail. These are defined as follows:
A collection ${x=(x_{i_1,\ldots,i_r})_{i_1,\ldots,i_r\in N}}$ of
scalars is called an $r$-array indexed by ${N\subseteq\mathbb{N}}$.
The subarray indexed by ${M\subset N}$ is denoted $x[M]$.
We let permutations $\phi$ of $N$ act on $x$ by permuting each index
dimension separately, ${\phi(x):=(x_{\phi(i_1),\ldots,\phi(i_r)})}$.
A \kword{jointly exchangeable array} is a random array $X$
that is indexed by ${N=\mathbb{N}}$ and satisfies ${\phi(X)\equdist X}$ for
all ${\phi\in\mathbb{S}_{\infty}}$.

The ergodic exchangeable arrays can be characterized explicitly, by 
Aldous-Hoover theorem \citep{Kallenberg:2005}:
To keep notation simple, assume ${r=2}$. Then $X$ is $\mathbb{S}_{\infty}$-ergodic if and only if there is a measurable function
${h:[0,1]^3\rightarrow\mathbb{R}}$ such that\nolinebreak
\begin{equation}
  \label{eq:AH}
  X\;\equdist\; (h(U_i,U_j,U_{ij}))_{i,j\in\mathbb{N}}
  \quad\text{ where }(U_i,U_{ij})_{i,j\in\mathbb{N}}\simiid\text{Uniform}[0,1]\;.
\end{equation}
Thus, $X$ is $\mathbb{S}_{\infty}$-ergodic if $h$ is fixed, and
$\mathbb{S}_{\infty}$-invariant if $h$ is random.
For ${r>2}$, the function $h$ has additional arguments
\citep{Kallenberg:2005}.
Kallenberg \citep{Kallenberg:1999} first proved the relevant case of
Lindenstrauss' theorem: If $X$ is $\mathbb{S}_{\infty}$-ergodic,
\begin{equation*}
  \tfrac{1}{n!}\tsum_{\phi\in\mathbb{S}_n}f((X_{\phi(i_1),\ldots,\phi(i_r)})_{i_1,\ldots,i_r})\;\xrightarrow{n\rightarrow\infty}\;\mean[f(X)]\quad\text{a.s. for
  }f\in\L_1(X)\;.
\end{equation*}
Eagleson and Weber \citep{Eagleson:Weber:1978:1} proved an early 
version of
\eqref{eq:CLT:exchangeable} for such averages (under
stronger
conditions than \cref{theorem:exchangeable}). Under suitable
additional assumptions, one can obtain a uniform result
\citep{Davezies:Haultfeuille:Guyonvarch:2021}.

An \kword{exchangeable graph} is an exchangeable 2-array with binary
entries and almost surely zero diagonal
\citep{Diaconis:Janson:2007}. We interpret the array as
the adjacency matrix of a random graph with vertex set $\mathbb{N}$.
Since that makes the range of $h$ binary, 
one can eliminate one degree of freedom in the representation above:
An exchangeable graph is ergodic if and only if \eqref{eq:AH} holds for a measurable function ${w:[0,1]^2\rightarrow[0,1]}$ and
${h(u,v,z):=\mathbb{I}\braces{z\leq w(u,v)}}$. For undirected graphs,
$w$ can be chosen to satisfy ${w(u,v)=w(v,u)}$, and is called a
\kword{graphon} \citep{Borgs:Chayes:Lovasz:Sos:Vesztergombi:2008}.

For a finite graph $y$ with vertex set $\braces{1,\ldots,k}$
and the subgraph $X[1,\ldots,k]$ of $X$ on the same set, 
consider the subgraph probability ${t(y):=P(X[1,\ldots,k]=y)}$.
Some
authors interpret $t(y)$ as a moment statistic \citep{Bickel:Chen:Levina:2011:1}.
For ${n\geq k}$ and a graph $x$ with vertex set $\braces{1,\ldots,n}$, the \kword{homomorphism density}
${t(x,y):=1/n!\sum_{\phi\in\mathbb{S}_n}\mathbb{I}\braces{x[\phi(1),\ldots,\phi(k)]=y}}$
is the (normalized) number of times $y$ occurs as a subgraph of $x$
\citep{Borgs:Chayes:Lovasz:Sos:Vesztergombi:2008,Lovasz:Szegedy:2006}.
If $X$ is ergodic,
and a finite subgraph ${X[1,\ldots,n]}$ is observed as data,
substituting into Kallenberg's result above shows
\begin{equation}
  \label{eq:homomorphism:densities}
  t(X[1,\ldots,n],\argdot)\;\xrightarrow{n\rightarrow\infty}\;P(X[1,\ldots,k]=\argdot)\;=\;t(\argdot)
  \qquad\text{ almost surely}.
\end{equation}
In other words, the sample homomorphism density $t(X[1,\ldots,n],y)$ is a strongly consistent estimator
of $t(y)$.
Borgs et al.\ \citep{Borgs:Chayes:Lovasz:Sos:Vesztergombi:2008} and
Lov\'asz and Szegedy \citep{Lovasz:Szegedy:2006} have also obtained
\eqref{eq:homomorphism:densities}, using different arguments. For these estimators, \eqref{eq:CLT:exchangeable} is due to
Bickel et al.\ \citep{Bickel:Chen:Levina:2011:1} and
Ambroise and Matias \citep{Ambroise:Matias:2012:1}.

\subsection{Stochastic block models with a growing number of classes}
\label{sec:adaptation}

Suppose we choose $h$ in \eqref{eq:AH} as follows:
Fix some ${m\in\mathbb{N}}$. Choose a measurable function ${\pi:[0,1]\rightarrow\braces{1,\ldots,m}}$
and a symmetric function
${v:\braces{1,\ldots,m}^2\rightarrow[0,1]}$. For each ${i\leq m}$, set
${\pi_i:=\mathbb{P}(\pi(U)=i)}$,
where $U$ is uniform in $[0,1]$. We can read $(\pi_i)_{i\leq m}$
as a distribution on $m$ categories, and $v$ as a matrix
${(v(i,j))_{i,j\leq m}}$. Define a random undirected graph with vertex
set $\mathbb{N}$ as 
\begin{equation*}
  X(\pi,v)\;:=\; \bigl(\,\mathbb{I}\braces{U_{ij}<v(\pi(U_i),\pi(U_j))}\,\bigr)_{i<j\in\mathbb{N}}\;.
\end{equation*}
Since this is a special case of \eqref{eq:AH},
$X(\pi,v)$ is an ergodic exchangeable graph, represented by the
piece-wise constant graphon ${w=v\circ(\pi\otimes\pi)}$.
A family of such distributions, indexed by some range of pairs $(\pi,v)$, is a
\kword{stochastic block model} with $m$ classes \citep[e.g.][]{Ambroise:Matias:2012:1}.
Since each law is specified by a finite vector $(\pi_i)$ and matrix $v$, the model is parametric.
Nonparametric extensions let $m$ grow
with sample size \citep[e.g.][]{Choi:Wolfe:Airoldi:2012:1}:
Choose an increasing function ${m:\mathbb{N}\rightarrow\mathbb{N}}$ and a parameter sequence
$(\pi^n,v^n)_{n\in\mathbb{N}}$ such that ${X_n:=X(\pi^n,v^n)}$ has $m(n)$ classes. An
observed graph on $n$ vertices is then explained as the finite subgraph $X_n[1,\ldots,n]$.

In the nonparametric case, no asymptotic normality results seem to be
known, but can easily be obtained from our results.
Since $X_n$ changes with sample size, \cref{theorem:exchangeable} is not applicable,
but \cref{theorem:CLT:g,theorem:BE:g} can be used instead.
As a concrete
example, let $y$ be the complete graph on three vertices.
In this case, ${P(X_n[1,2,3]=y)}$ is often called the triangle
density.
If vertex $1$ is in class $i$, but the classes of $2$ and $3$ are unknown, the probability
that $X_n[1,2,3]$ is a triangle is
\begin{equation*}
  E_i(n)
  \;:=\;
  \mean[f(X_n)|\pi^n(U^n_1)=i]
  \;=\;
\tsum_{j\leq m(n)}\pi_j^{n}\bigl(v^n(i,j)\tsum_{k\leq m(n)}\pi^n_k v^n(i,k)v^n(j,k)\bigr)\;.
\end{equation*}
Applying \cref{theorem:CLT:g,theorem:BE:g} to
${f(x):=\mathbb{I}\braces{x[1,2,3]=y}}$ yields:
\begin{corollary}\label{hypo} Let $Z$ be a standard normal variable. As ${n\rightarrow\infty}$,
    \begin{equation*}
    \tfrac{\sqrt{n}}{\eta_n}
    \bigl(
  \tfrac{1}{n(n-1)(n-2)}\tsum\mathbb{I}\braces{X_n[i_1,i_2,i_3]=y}
    -P(X_n[1,2,3]=y)\bigr)
    \;\xrightarrow{\;d\;}\;
    Z\;,
  \end{equation*}
  where the sum runs over all distinct triples ${i_1,i_2,i_3\leq n}$, and
  \begin{equation*}
    \eta_n^2=\tsum_{i\leq m(n)}\pi_i^{n}E_i(n)\bigl(E_i(n)-\tsum_{j\leq m(n)} \pi_j^{n}E_j(n)\bigr)\quad\text{almost surely.}
  \end{equation*}
  The Wasserstein distance to the limit is ${O\bigl(\eta_n^{-3}n^{-\frac{1}{2}}\|f(X_n)\|_{4}^{\frac{3}{4}}\bigr)\;.}$
\end{corollary}
The simplest SBM is an \kword{\Erdos-\Renyi }(ER) graph, where
each edge is an independent Bernoulli variable with
success probability $p$, that is, ${v:=p}$
is constant. This model has been thoroughly studied, and we can relate
the corollary to some known results:
\begin{examples}
\exitem{} If $X$ is an ER graph, 
${t(X[1,\ldots,k],\argdot)}$ satisfies a degenerate central
limit theorem, with ${\eta=0}$, see \citep{Ambroise:Matias:2012:1}.
To see this in the corollary,
set ${X_n:=X}$ for all $n$. We can then consider the limit
$\eta_nZ$. Since
${E_i(n)}$ does not depend on $i$ nor $n$, we obtain
${\eta_n=0}$.\\[-.7em]

\exitem{} Let each $X_n$ be an ER graph, with edge probability
${p(n)}$, and let ${p(n)\rightarrow 0}$.
In principle, \cref{hypo} holds: The limiting
triangle density is $0$, and ${\eta_n=0}$.
However, more bespoke results rescale by ${1/\sqrt{p(n)}}$ to
make small-scale behavior visible \cite{Janson:Luczak:Rucinski:2000:1}.
These do not follow from \cref{theorem:CLT:g}, since the variables
${\mathbb{I}\braces{X_n[1,2,3]=y}/p(n)}$ are not uniformly integrable.
\end{examples}

\subsection{Separate exchangeability}\label{sep:sec}
A random $r$-array $X$ is \kword{separately
  exchangeable} if it is invariant under the action 
\begin{equation*}
    \bphi x := x_{\phi_1(i_1),\ldots,\phi_r(i_r)}
  \quad\text{ for all
  }x\in\xspace\text{ and }\bphi=(\phi_1,\ldots,\phi_r)\in\mathbb{S}_{\infty}^r\;.
\end{equation*}
Comparing to \eqref{eq:diagonal:action} shows
that joint exchangeability is the diagonal
invariance corresponding to separate exchangeability.
Some models for relational
data in machine learning assume separate exchangeability for
matrices whose rows and columns are indexed by distinct sets (e.g.\ consumers
and products), and joint exchangeability if the sets are identical (e.g.\
vertices of a graph) \citep{Orbanz:Roy:2015}.
Separate exchangeability is the stronger property, and results in a
faster rate and simpler asymptotic variance:
\nolinebreak
\begin{corollary}
  \label{corollary:separately:ex}
  Let $X$ be a separately exchangeable $r$-array, and let
  ${f\in\L_2(X)}$ be a function that satisfies \eqref{eq:condition:limsup}. As ${n\rightarrow\infty}$,
  \begin{equation*}
    \sqrt{n^{r}}\,\cF_n(f,X)\;
    \;=\;
    \sqrt{n^r}
    \Bigl(\tfrac{1}{(n!)^r}\tsum_{\phi\in\mathbb{S}_n^r}f(\phi X)-\mean[f(X)|\mathbb{S}^r_\infty]\Bigr)\;
    \xrightarrow{\;d\;}
    \;
    \eta Z\;,
  \end{equation*}
  where $Z$ is standard normal and independent of $\eta$.
  The asymptotic variance satisfies ${\eta^2=\text{\rm Var}[f(X)|\mathbb{S}^r_{\infty}]<\infty}$ almost surely.
\end{corollary}

\begin{example}
  \exitem{}
  The convergence rate for
  homomorphism densities is in general $n^{-1/2}$ if a graph is exchangeable, but
  ${n^{-1}}$ if it is \Erdos-R\'enyi \citep[e.g.][]{Ambroise:Matias:2012:1}. 
  \cref{corollary:separately:ex} shows that is a consequence of
  additional symmetries in ER
  graphs, since they are not only jointly 
  but even separately exchangeable. 
\end{example}

\subsection{Graphex models}
\label{sec:graphex}

\def\cadlag{c\`adl\`ag }
\def\Levy{L\'evy }

Caron and Fox \citep{Caron:Fox:2017:1} have proposed a class of random graphs that,
with extensions and refinements by other authors
\citep{Borgs:Chayes:Cohn:Holden:2016:1,Veitch:Roy:2016:1},
are referred to as \kword{graphex models}.
Recall from \eqref{eq:AH} how an ergodic exchangeable graph is generated
by a graphon ${w:[0,1]^2\rightarrow[0,1]}$ and independent uniform variables.
A graphex model is defined similarly, by a symmetric measurable function
${\omega\!:\mathbb{R}_{\geq 0}^2\rightarrow[0,1]}$ and
a unit-rate Poisson process
${\Pi=\braces{(U_1,V_1),(U_2,V_2),\ldots}}$
on ${\mathbb{R}_{\geq 0}^2}$. Let ${U_{ij}}$, for ${i\leq j\in\mathbb{N}}$, again
be i.i.d.\ uniform elements of $[0,1]$. Define a
random countable subset $X_{\omega}$ of $\mathbb{R}_{\geq 0}^2$ as
\begin{equation*}
  \label{graphex:generative}
  (V_i,V_j)\in X_{\omega}\;\iff\;U_{ij}<\omega(U_i,U_j)\;.
\end{equation*}
This set is interpreted as a graph, in which vertices $V_i$ and
$V_j$ are connected if the pair $(V_i,V_j)$ is in $X_\omega$. The set
$X_{\omega}$ thus functions as a form of adjacency matrix, but each
vertex is identified by the value $V_i$, rather than the index $i$.
A subgraph is not selected as an $n\times n$ submatrix,
but by placing a rectangle $[0,s)^2$ in the plane:
The subgraph $g_s(X_\omega)$ for ${s\in(0,\infty]}$ is
\begin{equation*}
  (i,j)\in g_s(X_{\omega}) \quad\Leftrightarrow\quad (V_i,V_j)\in X_{\omega}\cap[0,s)^2\;.
\end{equation*}

Suppose an instance ${g_s(X_{\omega})}$ with $N$ vertices is
observed. Veitch and Roy \citep{Veitch:Roy:2016:1} have shown that one can estimate the restriction
$\omega|_{[0,s]^2}$ of $\omega$, provided $s$ is known: Subdivide $[0,s)^2$ into
quadratic patches $I_{ij}$, and define a piece-wise constant function ${\hat{\omega}_s}$ on
$[0,s)^2$ by specifying its value on each patch as
\begin{equation*}
  \hat{\omega}_s\vert_{I_{ij}}\;:=\;\mathbb{I}\braces{(i,j)\in G}
  \qquad\text{ where }
  I_{ij}:=\bigl[\tfrac{i-1}{N}s,\tfrac{i}{N}s\bigr)\times\bigl[\tfrac{j-1}{N}s,\tfrac{j}{N}s\bigr)\;.
\end{equation*}
This estimator is consistent on bounded domains
${[0,t)^2}$, in the following sense: Regard $\hat{\omega}_s$ as a
function ${\mathbb{R}_{\geq 0}^2\rightarrow[0,1]}$, with constant value $0$
outside ${[0,s)^2}$. Generate $X_{\hat{\omega}_s}$ according to
\eqref{graphex:generative}, using a Poisson process and uniform variables
that are independent of $X_{\omega}$. Then
\begin{equation}
  \label{VR:estimator}
  g_t(X_{\hat{\omega}_s}) \;\xrightarrow{\;\text{\tiny d}\;}\;
  g_t(X_{\omega}) \qquad\text{ as }s\rightarrow\infty\;,
\end{equation}
for every fixed ${t\in(0,\infty)}$ \citep{Veitch:Roy:2016:1}.
If $f$ is a measurable function of finite graphs, the Veitch-Roy estimator of
$\mean[f(g_t(X_\omega))]$ is therefore
\begin{equation*}
  \hat{f}_s:=\mean[f(g_t(X_{\hat{\omega}_s}))\,|\,g_s(X_{\omega})]\;.
\end{equation*}
Distributional convergence \eqref{VR:estimator} implies ${\hat{f}_s\rightarrow\mean[f(g_t(X_\omega))]}$
a.s.\ for ${s\rightarrow\infty}$.

We illustrate how to obtain rates for a simple example:
Fix ${t>0}$. For a finite graph $g$, choose $f$ as
\begin{equation}
  \label{graphex:edge:density}
  f(g):=\tfrac{1}{t^2}|\text{edge set of }g|
  \quad\text{ hence }
  f(g_t(X_\omega))=\tfrac{1}{t^2}|X_\omega\cap [0,t)^2|\;.
\end{equation}
The function ${(\omega,t)\mapsto\mean[f(g_t(X_\omega))]}$ is then
similar to the edge density in a graphon model. Consider the random sets
\begin{equation*}
  \mathcal{V}_{mn}:= X_{\omega}\cap [m,m+1)\times[n,n+1)\quad\text{
        for }m,n\in\mathbb{N}\;.
\end{equation*}
If we choose ${s\in\mathbb{N}}$, we have
\begin{equation*}
  \hat{f}_s=\;\tfrac{1}{t^2}\tsum_{(i,j)\in g_s(X_\omega)}P((i,j)\in g_t(X_{\hat{\omega}_s})|g_s(X_{\omega}))=\tfrac{1}{s^2}\tsum_{m,n<s}|\mathcal{V}_{mn}|\;.
\end{equation*}
The random array $(|\mathcal{V}_{mn}|)_{m,n}$ is, by
construction of $X_\omega$, jointly exchangeable and ergodic, and
\cref{theorem:exchangeable} yields:
\begin{corollary}
  \label{result:graphex}
  Let ${\omega:\mathbb{R}_{\geq 0}^2\rightarrow[0,1]}$ be a measurable and
  symmetric function, and fix ${t>0}$. Define $f$ as in \eqref{graphex:edge:density}.
  Then, for ${Z\sim N(0,1)}$,
  \begin{equation*}
    \sqrt{s}\bigl(
    \hat{f}_s
    -
    \mean\bigl[f(g_t(X_\omega))\bigr]
    \bigr)
    \;\xrightarrow{\;\text{\tiny d}\;}\;
    \eta Z
    \qquad\text{ as }s\rightarrow\infty\;,
  \end{equation*}
  where ${\eta^2=4\,\text{\rm
      Cov}\bigl[|X_{\omega}\cap[0,1]^2|,|X_{\omega}\cap[0,1]\!\times\![0,2]|\big\vert\group\bigr]}$
  is a finite constant.
\end{corollary}
The random set $X_{\omega}$ is invariant under an uncountable permutation
group that transforms each axis $\mathbb{R}_{\geq 0}$
\citep{Caron:Fox:2017:1}, and is in fact ergodic \citep{Kallenberg:2005}.
That is an example of uncountable exchangeability, as described at the
beginning of this section.
The local counts ${|\mathcal{V}_{mn}|}$ are a device to
reduce uncountable to countable exchangeability, and hence to invariance under a nice group.

\section{Applications II: Marked point processes}
\label{sec:marked:pp}

\def\mspace{\mathcal{M}}

Random geometric measures are point processes whose behavior at a given
point may depend on points nearby. They originate from so-called germ-grain models in physics
\citep{grain}, and are used to study e.g.\ nearest neighbor methods and Voronoi tesselations
\citep{merlin,Penrose:2007}. \cref{theorem:CLT,theorem:BE} are
directly applicable.

\subsection{Setup}
\def\yspace{\mathbf{Y}}

\label{rgm:setup}
The following definitions are adapted from those of Penrose
\citep{Penrose:2007}, with some simplifications: Consider two
Polish spaces, $\xspace$ (which we think of as a set of points)
and $\yspace$ (a set of marks or covariates), both equipped with
their Borel $\sigma$-algebras.
Denote by $\mathbf{M}$ the space
of $\sigma$-finite measures on $\xspace\times\yspace$, equipped with the
$\sigma$-algebra generated by the evaluation maps, and by $\mathcal{F}$
the set of finite subsets of $\xspace\times\yspace$. Let 
${\mu:\xspace\times\yspace\times\mathcal{F}\rightarrow\mathbf{M}}$ be
a measurable map, and ${W\subset\xspace\times \yspace}$ a compact set.
Loosely speaking, $\mu$ assigns to each marked point
$(x,y)$ a measure $\mu(x,y,F)$ that depends on
a set $F$ of points near $x$, and on their marks. These nearby points are collected
by using $W$ as an observation window, which is moved over $\xspace\times \yspace$ by elements of
a group:
Let $\group$ be a nice group that
acts measurably on $\xspace$. 
We extend the action to one on $\xspace\times\yspace$ by defining
\begin{equation}
  \label{rgm:action}
  \phi(x,y):=(\phi(x),y)\qquad\text{ for all }\phi\in\group,\, (x,y)\in\xspace\times\yspace\;.
\end{equation}
For compact ${\A_1,\A_2,\ldots\subset\group}$, write
$\A_nW=\braces{(\phi(x),y)|\phi\in\A_n,(x,y)\in W}$.
If $\Pi$ is a point process on $\xspace\times\yspace$, then
\begin{equation}
  \label{eq:rgm}
  \nu_n(\argdot):=\frac{1}{|\A_n|}\tsum_{(x,y)\in\Pi_n}\;\mu(x,y,\Pi_n)(\argdot)\quad\text{
    for }\quad \Pi_n:=\Pi\cap \A_nW
\end{equation}
is a random measure on $\xspace\times\yspace$.
The sequence $(\nu_n)$ is called a \kword{random geometric measure}
if $\Pi$ is invariant under the action
\eqref{rgm:action}, and if the sets $\Pi_n$ are almost surely
finite. See \citep{Baryshnikov:Yukich:2005,Penrose:2007} 
for similar definitions.

\subsection{Asymptotic normality}
A central theme in the literature on random geometric measures is the limiting behavior of
statistics of the form
\begin{equation*}
  \nu_n(h):=\myint_{\xspace\times\yspace}h(x,y)\nu_n(dx,dy)\quad\text{ for
  }\quad h:\xspace\times\yspace\rightarrow\mathbb{R}\;.
\end{equation*}
Such results typically require that the window does not
collect any point more than once.
A simple condition that excludes such repetitions is as follows:
Require (i) that $\group$ contains a subgroup $\mathbb{H}$ such that
${\phi(W)\cap\psi(W)=\emptyset}$ for distinct
${\phi,\psi\in \mathbb{H}}$, and (ii) that ${\mathbb{H}W=\group W}$. 
Informally, ${\mathbb{H}W}$ ``tiles'' the set ${\group
  W\subset\xspace}$ of points reached by the window.
We also require that (iii) the set $\{\phi\in \group| \phi(W)\cap W\ne \emptyset\}$ is 
compact. If
${\group=\xspace=\mathbb{R}^2}$, for example, one might choose ${W=[-1,1]^2}$
and ${\mathbb{H}=\braces{(2i,2j)|i,j\in\mathbb{Z}}}$.
The relationship to our results becomes clear if we define
\begin{equation*}
  \begin{split}
  f_n(F)\;&:=\;\myint_{\xspace\times\yspace}h(x',y')\tsum_{(x,y)\in F\cap W}\;\mu(x,y,\Pi_n)(dx',dy')
  \\
  \text{and}\quad  f(F)\;&:=\;\myint_{\xspace\times\yspace}h(x',y')\tsum_{(x,y)\in F\cap W}\;\mu(x,y,\Pi)(dx',dy')
\end{split}\end{equation*}
for ${F\in\mathcal{F}}$, and observe that 
 ${  \nu_n(h)
  \approx\frac{1}{|\A_n\cap \mathbb{H}|}
\int_{\A_n\cap \mathbb{H}}f_n(\phi(\Pi))|d\phi|
  =
  \mathbb{F}_n(f_n,\Pi)}$.
We apply \cref{theorem:CLT} and \ref{theorem:BE}, and obtain:
\begin{proposition}
  \label{result:rgm}
  Require that (i)--(iii) above hold, and that the sets $\A_n$ in \eqref{eq:rgm} form a tempered
  \Folner sequence.
  For each $n$, let $\alpha^{(n)}(\argdot|\group)$ be the conditional mixing
  coefficient of $\Pi$ and ${f_n}$. If
  \begin{equation*}
    \sup_n\int_{\mathbb{G}}\alpha^{(n)}(d(e,\phi)|\mathbb{G})^{\frac{\varepsilon}{2+\varepsilon}}|d\phi|<\infty
    \quad\text{ and }\quad
    \|f_n(\Pi)^{2+\varepsilon}\|_1<\infty
  \end{equation*}
  holds for some ${\varepsilon\ge 0}$, then as ${n\rightarrow\infty}$,
      \begin{equation*}
    \label{grison}
    \sqrt{|\A_n\cap \mathbb{H}|}\bigl(\nu_n(h)-\mean[\nu_n(h)|\group]\bigr)\;\xrightarrow{\;d\;}\;\eta
    Z\quad\text{ for }Z\sim N(0,1)\;,
  \end{equation*}
  where ${\eta^2=\int_{ \mathbb{H}}\text{\rm Cov}[f(\Pi),f(\phi
      \Pi)|\mathbb{G}]|d\phi|}$ and ${\eta\,\condind\,Z}$. Moreover,
  \begin{equation*}
    \dW\Big(\mfrac{\sqrt{|\A_n\cap \mathbb{H}|}}{\eta}\bigl(\nu_n(h)-\mean[\nu_n(h)|\group]\bigr),Z\Big)=O\Big(\mfrac{1}{\sqrt{|\A_n\cap \mathbb{H}|}}\max\braces{1,\|f_n(\Pi)\|_4^3}\Big)\;.
  \end{equation*}
\end{proposition}

\subsection{Relationship to existing results}
Versions of the result above are known in the case where
$\xspace$ is $\mathbb{R}^r$, ${\group=\mathbb{R}^r}$ consists of shifts, and
$\A_n$ is the Euclidean ball $\B_n$ \citep{Penrose:2007,merlin,grain}.
These are not phrased in terms of conditional mixing, but
instead use a ``stabilization condition'' \citep[e.g.][]{Baryshnikov:Yukich:2005}. The next result translates
stabilization to mixing conditions. There is no standardized
definition of stabilization; the one we state below is similar to
Definition 2.4 of Penrose \citep{Penrose:2007}, which he calls \emph{power-law
stabilizing} of order $q$.
For $(x,y)\in\xspace\times\yspace$ and ${F\in\mathcal{F}}$, let $F_t$ be
the truncated set
$\braces{(\tilde{x},\tilde{y})\!\in\! F\,|\,d(x,\tilde{x})\leq t}$.
The \kword{stabilization radius} of $\mu$ is
\begin{equation*}
  R(x,y,F):=\inf\braces{t>0\,\vert\,\mu(x,y,F)=\mu(x,y, F_t)}\;,
\end{equation*}
where we use the convention $\inf\emptyset=\infty$. If
\begin{equation}
  \sup_{s>0} \sup_{(x,y)\in W}s^{q}P\bigl(
  R(x,y,\Pi)>s\bigr)\;<\;\infty\quad\text{ for some }q>1\;,
\end{equation}
$\mu$ is \kword{polynomially stable} of order ${q}$.
The condition implies conditional mixing
if the metric balls in $\group$ do not expand too quickly:\nolinebreak
\begin{proposition}
  \label{result:polynomially:stable}
  Let $\Pi$ be a Poisson process, and let $\mu$ be polynomially stable
  of order $q$. If the metric balls ${\B_n}$ in $\group$ satisfy ${\sup_{n\in\mathbb{N}}
    n^{-r}|\B_n|<\infty}$ for some ${r>0}$, then
  \begin{equation*}
    \sup_n \myint_{\mathbb{G}}
    \alpha^{(n)}(d(e,\phi)|\mathbb{G})^{\frac{\varepsilon}{2+\varepsilon}}|d\phi|<\infty
    \qquad\text{ whenever }q>\tfrac{2+\varepsilon}{\varepsilon}\,r\;.
  \end{equation*}
\end{proposition}
That holds in particular for the groups ${\mathbb{R}^r}$, 
since an $r$-dimen\-sional Euclidean ball has volume
${|\B_n|=(\sqrt{\pi}n)^r/\Gamma(\frac{r}{2}+1)}$.
Geometric group theory provides further examples:
A group that satisfies ${\sup_{n\in\mathbb{N}}
  n^{-r}|\B_n|<\infty}$
and is also finitely generated is said to be of
\kword{polynomial growth} \citep{Loeh:2017}. Nice groups
of polynomial growth include $\mathbb{Z}^d$, the groups 
in \cref{corollary:nilpotent},
or the discrete Heisenberg groups \citep[e.g.][]{Bump:Diaconis:Hicks:Miclo:Widom:2017}.

\section{Applications III: Entropy}
\label{sec:entropy}

The entropy of a stationary process is defined as a limit. This limit exists almost surely, by
the Shannon-McMillan-Breiman (SMB) theorem \citep{Shields:1996}.
It has a natural generalization to invariant processes \citep[e.g.][]{Einsiedler:Ward:2011:1},
which again converges almost surely \citep{Lindenstrauss:2001:1}.
An adaptation of \cref{theorem:CLT} gives conditions under which it is asymptotically normal.
In this section, we assume $\group$
is discrete, and \kword{finitely generated}, which means there is a finite
subset ${G\subset\group}$ such that $\group$ is the smallest group
containing $G$. That is, for example, true for $\mathbb{Z}^r$ (choose $G$
as the set of unit coordinate vectors), but not for $\mathbb{S}_{\infty}$.

\subsection{Entropy}
Let $Y$ be a discrete random variable with mass function
${p(k)\!:=P(Y=k)}$ for ${k\in\mathbb{N}}$. If ${Y_1,Y_2,\ldots}$ are i.i.d.\ copies of $Y$,
the law of large numbers guarantees almost sure convergence
\begin{equation*}
  -\tfrac{1}{n}\log(p(Y_1)\times\ldots\times p(Y_n))\;\xrightarrow{n\rightarrow\infty}
  -\mean[\log p(Y)]\;=:\;H[Y]\;.
\end{equation*}
The constant $H[Y]$ is the \kword{entropy} of $Y$ \citep{Kallenberg:2001}.
If ${X=(X_i)_{i\in\mathbb{Z}}}$ is a
stochastic process with values in the finite set $[K]$, the entropy
can be defined similarly: If $p_n$ is the joint mass function of ${(X_1,\ldots,X_n)}$,
and $X$ is stationary and ergodic, there is
a constant ${h[X]\geq 0}$ such that
\begin{equation}
  \label{entropy:stationary}
  -\tfrac{1}{n}\log
  p_n(X_1,\ldots,X_n)\;\xrightarrow{n\rightarrow\infty}\; h[X]
  \quad\text{ almost surely.}
\end{equation}
This is the SMB theorem, and $h[X]$ is again called the entropy, or the entropy rate \citep{Shields:1996}.
The term ${-\frac{1}{n}\log p_n}(X_1,\dots,X_n)$ is the \kword{empirical entropy}.

\subsection{Entropy of invariant distributions}
Let $\group$ be countable, and $(\A_n)$ a tempered \Folner sequence with ${|\A_n|/\log(n)\rightarrow\infty}$.
Let $X$ be a $\group$-ergodic random element of $\xspace$. To define entropy,
regard
${(\phi X)_{\phi\in\group}}$ as a stochastic process on the group, and discretize its
state space:
Choose a partition  ${\lambda:=(\lambda_1,\ldots,\lambda_K)}$ of $\xspace$ into
a finite number of Borel sets, and write ${\lambda(x)=k}$ if
${x\in \lambda_k}$.
Let ${p_n}$ be the
joint mass function of ${(\lambda(\phi X))_{\phi\in\A_n}}$.
Then
there is a constant ${h_{\lambda}[X]\geq 0}$ such that
\begin{equation*}
  h_{n}(\lambda,X):=\;-\tfrac{1}{|\A_n|}\log
  p_n\bigl((\lambda(\phi X))_{\phi\in\A_n}\bigr)
  \;\xrightarrow{n\rightarrow\infty}\;h_{\lambda}[X]
  \quad\text{ almost surely.}
\end{equation*}
This result is again due to Lindenstrauss \citep{Lindenstrauss:2001:1}.
To recover \eqref{entropy:stationary}, choose $X$ as a stationary process
${(X_i)_{i\in\mathbb{Z}}}$, and
${\lambda_k:=\braces{x=(x_i)_{i\in\mathbb{Z}}\,|\,x_0=k}}$.

\subsection{Asymptotic normality}
Suppose $\group$ admits a total order $\preceq$ that is left-invariant
(i.e.\ ${\phi\preceq\psi}$ if and only if ${\pi\phi\preceq\pi\psi}$ for ${\phi,\psi,\pi\in\group}$).
The process values indexed by a set ${G\subset\group}$ are predictive of
the value at $\phi$ if
\begin{equation*}
  L_{\phi}(G):=\log P[\lambda(\phi X)\,|\,\lambda(\psi X),
  \psi\in G]
\end{equation*}
is large, where $P$ denotes probability under the law of $X$. The scalar
\begin{equation*}
  \rho_m:=\sup_{A\subset\group}\|L_e(A)-L_e(A\cap\B_m)\|_2
\end{equation*}
measures how well the value at the identity is predicted by
values within a radius $m$.
Recall that the definition of mixing in \cref{sec:mixing}
uses pairs ${\phi_1,\phi_2}$ in $\group$. We extend it to $k$-tuples:
For ${k\in\mathbb{N}}$, define
\begin{equation*}
  \mathcal{C}(t,k):=\bigbraces{
    (A,B)\in\sigmaf(\phi_1,\ldots,\phi_k)\otimes\sigmaf(G)\big\vert
    G\subset\group, \phi_1,\ldots,\phi_k\in\group\mysetminus\B_{t}(G)}\;,
\end{equation*}
and ${\alpha(t,k):=\sup_{(A,B)\in\mathcal{C}(t,k)}|P(A,B)-P(A)P(B)|}$.
The mixing coefficient in \cref{sec:mixing} is hence ${\alpha(t)=\alpha(t,2)}$.

\begin{theorem}
  \label{theorem:entropy}
  Let $\group$ be a finitely generated, nice group with
  left-invariant total order, and let $X$ be $\group$-ergodic with
  ${\sup_{A\subset\group}\|L_e(A)\|_{2+\varepsilon}<\infty}$
  for some ${\varepsilon>0}$. Choose a tempered
  \Folner sequence satisfying
  ${|\A_n\!\triangle\!\B_{b_n}\A_n|\,/\,|\A_n|\rightarrow 0}$
  and ${\sqrt{|\A_{n}|}\rho_{b_n}\rightarrow 0}$, for some
  sequence $(b_n)$ of positive scalars. If
  \begin{equation}
    \label{eq:entropy:mixing}
    \tsum_{i\in\mathbb{N}}|\B_i|\min_{m\leq i}\bigl(\rho_m+\alpha(i-m,|\B_m|)^{\frac{\varepsilon}{2+\varepsilon}}\bigr)<\infty
  \end{equation}
  holds for the mixing coefficient of the function ${f:=\lambda}$, then
  \begin{equation*}
    {\sqrt{|\A_n|}}
    \bigl(
    h_n(\lambda,X)-h_{\lambda}[X]
    \bigr)
    \quad\xrightarrow{\;d\;}\quad
    \eta Z
    \qquad\text{ as }n\rightarrow\infty\;,
  \end{equation*}
  where the asymptotic variance is independent of $Z$ and satisfies
    \begin{equation*}
    \eta^2 =
    \tsum_{\phi\in\group}
    \text{\rm Cov}
    \bigl[
      L_e(\braces{\psi\preceq e}),
      L_{\phi}(\braces{\psi\preceq
          \phi})
    \bigr]<\infty
    \qquad\text{ almost surely.}
  \end{equation*}
\end{theorem}
Condition \eqref{eq:entropy:mixing} can be interpreted as follows:
The proof represents $h_n$ as
\begin{equation*}
  \log
  p_n\bigl((\lambda(\phi X))_{\phi\in\A_n}\bigr)
=\tsum_{\phi\in \A_n}L_{\phi}(\{\psi\in \A_n\vert \psi\preceq \phi\})\;,
\end{equation*}
and approximates it by the average $\mathbb{F}_n$ of
${f'(X)=L_{\phi}(\{\psi\vert\psi\preceq\phi\}\cap \B(\phi,m))}$. The approx\-i\-ma\-tion error
is a function of $\rho_m$, and decreasing in $m$. Mixing, on the other hand, involves
tuples in $\B_m$, and since ${\alpha(\argdot,|\B_m|)}$
is non-decreasing in $|\B_m|$, a smaller $m$ means better mixing.
Informally, dependence within the process is both beneficial (it makes
predicting one value from others easier) and detrimental (it reduces mixing).\nolinebreak

\begin{remark}
(a) Left-invariance of the order is not required for asymptotic normality,
but simplifies $\eta$. Provided it holds, $\eta$ does not depend on the choice of $\preceq$.
(b) Examples of groups satisfying \cref{theorem:entropy} are
$(\mathbb{Z}^r,+)$ and the groups in \cref{corollary:nilpotent}, or
discrete Heisenberg groups \citep{Loeh:2017}. 
(c) Existence of a total order implies ${\phi^m\neq e}$ for all ${m\in\mathbb{N}}$, unless ${\phi=e}$.
In algebraic terms, $\group$ is torsion-free \citep{Loeh:2017}.
\end{remark}

\bibliographystyle{imsart-number} 
\bibliography{references}

\pagebreak

\def\independenT#1#2{\mathrel{\rlap{$#1#2$}\mkern2mu{#1#2}}}
\def\subL1{1}
\newcommand{\dself}{\partial}
\newcommand{\nrm}{\mfrac{1}{\eta_n}}

\appendix

\section{Proof Overview and Auxiliary Results}

The proofs are presented in three parts, for the basic limit theorems
in \cref{proofs:basic}, for the general ones
in \cref{proofs:general}, and for all other results in
\cref{proofs:other}.
The basic results (\cref{theorem:CLT,theorem:BE}) are
special cases of the general ones (\cref{theorem:CLT:g,theorem:BE:g}),
but we prove them first to clarify the approach. The
general proofs require changes, but follow the same layout.

\subsection{Proof overview}

The proofs of
the main results, \cref{theorem:CLT,theorem:BE,theorem:CLT:g,theorem:BE:g}, use
Stein's method \citep[e.g.][]{Ross:2011:1}: For the function class
\begin{equation}
  \label{stein_class}
  \mathcal{F}:=\bigbraces{t\in\mathcal{C}^2(\mathbb{R})\,\big\vert\,\|t\|_{\infty}\leq 1,\|t'\|_{\infty}\leq\sqrt{2/\pi},\|t''\|_{\infty}\leq 2}
\end{equation}
and a real-valued random variable $W$, Stein's inequality guarantees
\begin{equation}
  \label{eq:stein}
  \dW(W,Z)\;\leq\;\sup_{t\in\mathcal{F}}\bigl|\mean[Wt(W)-t'(W)]\bigr|
  \qquad\text{ for }Z\sim N(0,1)\;.
\end{equation}
The distance $\dW$ metrizes convergence in distribution for variables
with a first moment \citep[e.g.][]{Ross:2011:1}. One can therefore establish a
central limit theorem for a sequence $(W_n)$ of such variables by showing
${\dW(W_n,Z)\rightarrow 0}$, and hence by showing that the right-hand side of
\eqref{eq:stein} vanishes as ${n\rightarrow\infty}$.\\[.5em]
\emph{Basic case}. In broad strokes, \cref{theorem:CLT,theorem:BE} are proven as follows:\\[.2em]
\myitem\tab Choose ${W=W_n}$ as a suitably scaled version of
$\eta(n)^{-1}\mathbb{F}_n$, where $\eta(n)$ is a (for now unspecified)
positive random variable.\\[.2em]
\myitem\tab To upper-bound \eqref{eq:stein},
split $W$ at a cut-off distance $b_n$ in $\group$, into a short-range
and a long-range term. 
Adapting \eqref{eq:stein} to these modifications
yields a refined bound, in \cref{lemma:basic:stein:bound:initial}.
The leg work of the proof is then to control each term
in this bound.\\[.2em]
\myitem\tab Stein's method involves the notion of ``dependency
neighborhoods'' \citep{Ross:2011:1}: A set, say $\mathcal{N}(i)$, of
indices for a random variable $X_i$ such that ${X_i\condind X_j}$ if
${j\not\in\mathcal{N}(i)}$.
In our proofs, the neighborhood is the area within the
cut-off $b_n$, but terms inside and outside the neighborhood are not
completely independent. We hence bound long-range terms using conditional
mixing.
\\[.2em]
\myitem\tab Split $f$ into small and large values at a threshold $\gamma_n$. Since no fourth
moment is assumed, large values must be controlled explicitly. \\[.2em]
\myitem\tab The resulting bound is a function of $\eta(n)$. Choose
$\eta(n)$ as an approximation to the quantity $\eta$
defined in the statement of \cref{theorem:CLT}.\\[.2em]
The central limit theorem then follows by showing that the
bound vanishes as ${n\rightarrow\infty}$, and the Berry-Esseen bound
by additionally requiring a third and fourth moment, and substituting these into the bound.
\begin{remark}
Limit theorems for random structures often require a condition called stable convergence
\citep[e.g.][]{Hausler:Luschgy:2015}.
That is not required here; instead, the proof shows that ${\eta(n)^{-1}}\mathbb{F}_n$
converges to $Z$ conditionally on $\sigma(\group)$, which is then used
to obtain convergence of $\mathbb{F}_n$ to $\eta Z$. That is possible
because $\eta$ is constant given $\sigma(\group)$. In terms
of \cref{theorem:ergodic:decomposition}, $\eta$ is a
function of $\xi$, and hence $\sigma(\group)$-measurable.
\end{remark}

\noindent\emph{General case}. Proving \cref{theorem:CLT:g,theorem:BE:g} requires a number of modifications:\\[.2em]
\myitem\tab
Since the dimension $k_n$ of the group may grow with $n$,
we work with surrogate functions that depend only on the first
few entries of ${\boldsymbol{\phi}\in\group^{k_n}}$.
\\[.2em]
\myitem\tab Working in $\group^{k_n}$ complicates the dependency neighborhoods.
\\[.2em]
\myitem\tab
Since ${\widehat{\mathbb{F}}_n}$ is now random,
we must also control the probability of
selecting elements of the dependency neighborhood, using the spreading
conditions.

\subsection{Comments on other proof techniques}
Central limit theorems can be proven with a range of tools, including Fourier techniques, Lindeberg's replacement
trick, or martingale methods. Unlike Stein's method, these do not seem
adaptable to our problems. In the case of concentration, 
the Efron-Stein inequality and other standard techniques similarly fail.
There are several obstacles: \emph{(i) Topology of the group.} Many
  martingale proofs, and the Efron-Stein approach to concentration,
  combine observations
  into blocks, and control dependence between blocks via an
  isoperimetric argument (i.e.\ block boundaries are of negligible size).
  That applies to some groups, such as ${\group=\mathbb{Z}}$, but fails
  even for ${\group=\mathbb{Z}^2}$. Bolthausen \cite{Bolthausen:1982:1} used Stein's method
  to address an instance of this problem.
\emph{(ii) Lack of a total order.} Replacement arguments (e.g.\ Lindeberg's method and the Efron-Stein
  inequality) rely on the left-invariant total order of $\mathbb{Z}$ to
  replace random variables sequentially.
  That makes them inapplicable, for example, to permutation groups.
\emph{(iii) Group size}, since replacement arguments require countability.

\begin{remark}
Martingales are applicable if $\group$ contains compact subgroups ${\group_1\subset\group_2\subset\ldots}$
such that ${\group=\cup_n\group_n}$. That is the case for
$\mathbb{S}_{\infty}$, with ${\group_n=\mathbb{S}_n}$.
If so, ${(\group_n)}$ is a
\Folner sequence, and ${(\mathbb{F}_n)}$
is a reverse martingale adapted to the filtration
${\sigma(\group_1)\supset\sigma(\group_2)\supset\ldots}$. That implies \eqref{eq:lindenstrauss}. The
corresponding case of
\cref{theorem:CLT} (with more restrictive moment and mixing conditions) follows from the reverse martingale central limit theorem.
Such arguments are used in \citep{Lovasz:Szegedy:2006} for
convergence, and in \citep{Eagleson:Weber:1978:1} for asymptotic normality.
However, the method has limitations even for
${\group=\mathbb{S}_\infty}$. For example:
If $(X_i)$ is an exchangeable sequence and $h$ a function of two arguments, ${(h(X_i,X_j))_{ij}}$ is an exchangeable array,
but even with proper normalization, ${\sum_{i<j}h(X_i,X_j)}$ is not a reverse martingale unless $h$ is symmetric in its arguments.\\
\end{remark}

\subsection{Auxiliary results}
We begin with a result that allows us to bound the Wasserstein
distance $\dW$. Recall that $\mathcal{L}$ denotes the set of Lipschitz
functions with constant 1. It is a standard result that
\begin{equation}
  \label{eq:kantorovich:rubinstein}
  \dW(X,Y)\;=\;\sup_{h\in\mathcal{L}}|\mean[h(X)]-\mean[h(Y)]|\;=\;\inf \mean[|X'-Y'|]\;,
\end{equation}
where the infimum
is taken over all couplings $(X',Y')$ of $X$ and $Y$. This identity is sometimes
known as the Kantorovich-Rubinstein formula.
In analogy to $\dW$, we define the conditional (and hence random) distance
\begin{equation*}
 \dW(X,Y|\group)\;:=\;
 \sup_{h\in\mathcal{L}}|\mean[h(X)|\group]-\mean[h(Y)|\group]|\;.
\end{equation*}        
The next lemma shows how it relates to $\dW$.
\begin{lemma}
\label{lemma:jensen:wassertein}Let $X$ and $Y$ be random variables in
${\L_1(\mathbb{R})}$, defined on an abstract probability space ${(\Omega,\mathcal{A},\mathbb{P})}$. Then
\begin{align*}
\dW(X,Y|\group)= \inf\,\mean\big[|X'-Y'|\,\big|\group\big]\qquad\mathbb{P}\text{-a.s.,}
\end{align*}
where the infimum runs over all couplings $(X',Y')$ of the conditional
variables ${X|\sigma(\group)}$ and ${Y|\sigma(\group)}$, and 
\begin{equation*}
\dW(X,Y)\,\le\,\mathbb{E}[\dW(X,Y|\group)]\,=\,\|\dW(X,Y|\group)\|_1\;.
\end{equation*}
\end{lemma}
\begin{proof}
  Since both random variables are real-valued, we can choose regular
  conditional distributions $p$ for $X$ and $q$ for $Y$. That is,
  ${p_\omega=P(\argdot|\group)(\omega)}$ holds for $\mathbb{P}$-almost
  all ${\omega\in\Omega}$, the map ${\omega\mapsto p_\omega}$ is
  measurable, and the same holds for 
  $q$ and $Y$. We can then apply \eqref{eq:kantorovich:rubinstein}
  pointwise in $\omega$, which shows that $\mathbb{P}$-almost surely,
  \begin{equation*}
    \sup_{h\in\mathcal{L}}|\mean[h(X)|\group](\omega)-\mean[h(Y)|\group](\omega)|
    \;=\;
    \dW(p_\omega,q_\omega)
    \;=\;
    \inf\mean|X'-Y'|\;,
  \end{equation*}
  where the infimum is taken over all $(X',Y')$
  with marginal distributions $p_\omega$ and $q_\omega$.
  That shows the first identity. The second claim holds since 
\begin{align*}
  \dW(X,Y)\,=&\,\sup_{h\in \mathcal{L}}\big|\mean[h(X)]-\mean[h(Y)]\big|
  =\sup_{h\in \mathcal{L}}\big|\mean[\mean[h(X)|\group]-\mean[h(Y)|\group]]\big|\\
  &\le\mean[\sup_{h\in \mathcal{L}}\big|\mean[h(X)|\group]-\mean[h(Y)|\group]\big|]
  =\mean[\dW(X,Y|\group)]\;,
\end{align*}
where we have used the tower property and the relation
  ${\sup\mean\leq\mean\sup}$.
\end{proof}
Conditioning in $\dW$ lets us swap a random variable $Y$ (which in the
proofs will be the asymptotic variance) between arguments:
\begin{lemma}[Random scaling]
  \label{lemma:random:scaling}
  Let $X$, $Y$, and $Z$ be random variables in
  ${\L_2(\mathbb{R})}$, such that $Y$ is $\sigma(\group)$-measurable. If
  ${Y\geq c}$ almost surely for some ${c>0}$,
  \begin{equation*}
  \dW(X,Z/Y)\le\|\dW(XY,Z|\group)\|_{1}/c\;.
  \end{equation*}
\end{lemma}
\begin{proof}
 The second part of \cref{lemma:jensen:wassertein} shows that
 ${\dW(X,Z/Y)\le\|\dW(X,Z/Y|\group)\|_{1}}$. Fix any ${\epsilon>0}$. Since $Y$ is
 $\sigma(\group)$-measurable, there is a coupling  ${(X',Z')}$ of the
 conditional variables ${X|\sigma(\group)}$ and ${Z|\sigma(\group)}$  such that 
${\mean[|X'Y-Z'||\group]\leq \dW(XY,Z|\group)+\epsilon}$, now by the first
 part of \cref{lemma:jensen:wassertein}.
This coupling satisfies
  \begin{equation*}
    \mathbb{E}[|X'Y/Y-Z'/Y| |\group]\;\le\;\mean[|X'Y-Z'||\group]/c \;\le\; \bigl(\dW(XY,Z|\group)+\epsilon\bigr)/c\;.
\end{equation*}
Since $\epsilon$ is arbitrary, it follows that ${\dW(X,Z/Y)\le\|\dW(XY,Z|\group)\|_{1}/c}$.
\end{proof}
We must repeatedly use bounds of the form
${\|\mean[\argdot|\group]\|_1\lesssim\|\argdot\|_{\frac{2+\varepsilon}{2}}\alpha(k|\group)^{\frac{2}{2+\varepsilon}}}$
to ``separate off'' conditioning.
The next two lemmas capture all cases needed in the proofs, for both
${\varepsilon=0}$ and ${\varepsilon>0}$. The first version applies to 
conditional mixing. Recall this involves a pair ${\phi_1,\phi_2}$ of
distance at least $k$ from a set ${G\subset \group}$, which here is of
finite size $m$. The lemma shows that, if a transformation $\pi$ does not move the pair 
too close to $G$, 
the desired inequality holds.
\begin{lemma}[Conditional mixing bound]
  \label{lemma:mixing}
  Let $X$ be $\group$-invariant, $Y$ a real-valued random variable, 
  and ${h\!:\xspace^{k+2}\times\mathbb{R}\rightarrow\mathbb{R}}$ a
  measurable function 
  with ${\mean[|h(X,\ldots,X,Y)|]<\infty}$.
  Fix ${\phi_1,\phi_2,\psi_1,\ldots,\psi_m\in\group}$, and set
  \begin{equation*}
    \label{eq:lemma:mixing:H}
    H_{\tau}:=h(\psi_1X,\ldots,\psi_mX,\tau^{-1}\phi_1X,\tau^{-1}\phi_2X,Y)\quad\text{
      for }\tau\in\group\;.
  \end{equation*}
  Let $\pi$ be an element of $\group$. If
  \begin{equation*}
     Y\condind\, X\,|\,\sigma(\group)
     \quad\text{ and }\quad
  k\leq \min_{i\leq 2,j\leq m}d(\tau^{-1}\phi_i,\psi_j)
  \end{equation*}
  for both ${\tau=\pi}$ and the identity ${\tau=e}$, then
  \begin{equation*}
    \bigl\|\mean[H_{\pi}|\group,Y]-\mean[H_e|\group,Y]\bigr\|_1
    \;\leq\;
    4\|H_{\pi}-H_e\|_{\frac{2+\varepsilon}{2}}\alpha(k|\group)^{\frac{2}{2+\varepsilon}}
  \end{equation*}
        for any ${\varepsilon\geq 0}$.
\end{lemma}

\begin{proof}
  \emph{Case 1: ${\|H_{\pi}-H_e\|_{\infty}}$ finite}. We approximate
  $h$ by a step function\nolinebreak
  \begin{equation}
    \label{eq:lemma:mixing:hast}
    h^*(\argdot,\argdot,\argdot)=\tsum_{i=1}^Nc_i\mathbb{I}(
      \argdot\in A_i,\argdot\in B_i,\argdot\in C_i)\;,
  \end{equation}
  for some ${N\in\mathbb{N}}$, measurable sets
  $A_i$ in $\xspace^m$, $B_i$ in $\xspace^2$ and $C_i$
  in $\mathbb{R}$, and scalars ${|c_i|\leq\|h\|_{\infty}}$.
  Define $H^*_\tau$ analogously to $H_\tau$, by
  substituting $h^*$ for $h$. 
  Fix any ${\delta>0}$. Since $h$ is integrable, $h^*$ can be chosen to make ${\|h-h^*\|_1}$
  arbitrarily small, and hence such that 
  ${\|(H_{\pi}-H_e)-(H^{\ast}_{\pi}-H^{\ast}_e)\|_1\leq\delta}$.
  If we abbreviate
  \begin{equation*}
    \begin{split}
      I_i\,:=\;&\mathbb{I}_{A_i}(\psi_1X,\dots,\psi_kX)\,\mathbb{I}_{C_i}(Y)
      \,\bigl(\mathbb{I}_{B_i}(\phi_1X,\phi_2X)-
      \mathbb{I}_{B_i}(\pi^{-1}(\phi_1X,\phi_2X))\bigr)
    \end{split}
  \end{equation*}
  and ${E_i:=\mean[I_i|\group,Y]}$, we have
  ${\|\mathbb{E}[H^{\ast}_{\pi}-H^{\ast}_e|\group,Y]\|_1\le\sum_{i=1}^{N_{\delta}} |c_i|\|E_i\|_1}$
  for some ${N_\delta\in\mathbb{N}}$. Using the definition of conditional mixing,
  we have
  \begin{align}
    \label{eq:lemma:mixing:comp}
      \tsum_{i\leq N_{\delta}} |c_i|\|E_i\|_{1}\;&\le\;
      \mathbb{E}\big[\!\tsum_{{i|E_i> 0}}|c_i| |E_i|+\!\tsum_{{i|E_i\le 0}} |c_i||E_i|\bigr]
      \nonumber\\&\le\;
      \tsum_{{i|E_i> 0}}|c_i| \mathbb{E}[E_i]-\!\tsum_{{i|E_i\le 0}} |c_i|\mathbb{E}[E_i]
      \nonumber\\&\le\;
      \max_i|c_i|\bigl(\|\tsum_{i|E_i> 0}E_i\|_1
      +
      \|\tsum_{i|E_i\le 0}E_i\|_1\bigr)
  \\   &\le\; 2\|H_{\pi}-H_e\|_{\infty}\,\alpha(k|\mathbb{G})\;.\nonumber
  \end{align}
  Since the right-hand side does not depend on $\delta$ or
$h^{\ast}$, that implies
\begin{equation*}
\|\mathbb{E}[H_{\pi}-H_e|\group,Y]\|_1
\le 2  \|H_{\pi}-H_e\|_{\infty}~\alpha(k|\mathbb{G})\;.
\end{equation*}
\emph{Case 2: ${\|H_{\pi}-H_e\|_{\infty}}$ infinite}. For ${r>0}$, define
\begin{equation*}
  \Delta H:=H_{\pi}-H_e
  \qquad \Delta H_r:=\Delta H\cdot\mathbb{I}\braces{\Delta H\leq r}
  \qquad
  \overline{\Delta H_r}:=\Delta H-\Delta H_r\;.
\end{equation*}
The triangle inequality gives ${\|\mean[H_\pi-H_e|\group,Y]\|_1\leq
  \|\Delta H_r\|_1+\|\overline{\Delta H_r}\|_1}$, and case 1 above implies
${\|\Delta H_r\|_1\leq  2r\alpha(k|\mathbb{G})}$.
Since ${\|h\|_{\frac{2+\varepsilon}{2}}}$ is finite, we can assume
${\|\Delta H\|_{\frac{2+\varepsilon}{2}}\leq 1}$ without loss of generality. By H\"{o}lder's
inequality,
\begin{equation*}
  \|\overline{\Delta H_r}\|_1
  \;\leq\;
  \|\Delta
  H\|_{\frac{2+\varepsilon}{2}}\cdot \|\mathbb{I}\braces{\Delta H>r}\|_{\frac{2+\varepsilon}{\varepsilon}}
  \;\leq\;
  2r^{-\frac{\varepsilon}{2}}\;.
\end{equation*}
We hence obtain ${\|\mean[H_\pi-H_e|\group,Y]\|_1\leq2r\alpha(k|\mathbb{G})+2r^{-\frac{\varepsilon}{2}}
  =
  4\alpha(k|\group)^{\frac{\epsilon}{2+\varepsilon}}}$ by choosing
${r=\alpha(k|\mathbb{G})^{\frac{-2}{2+\varepsilon}}}$.
\end{proof}
The second version is the analogous result for marginal mixing
coefficients. As in \cref{sec:g}, ${e_{i,\tau}=(e,\ldots,e,\tau,e,\ldots,e)}$ denotes
a vector with $k_n$ entries and $\tau$ as the $i$th entry, and
$\delta_{i,j}$ is defined as in \eqref{def:set:delta}.
\begin{lemma}[Marginal mixing bound]
  \label{lemma:mixing:diagonal}
  Let $X$ be a random element of $\xspace_n$, invariant under
  the diagonal action of $\group^{k_n}$, and $Y$ a real-valued random variable.
  Let ${h\!:\xspace_n^{k+2}\times\mathbb{R}\rightarrow\mathbb{R}}$ be
  measurable,
  with ${\mean[|h(X,\ldots,X,Y)|]<\infty}$.
  Fix ${\bphi_1,\bphi_2,\bpsi_1,\ldots,\bpsi_m\in\group^{k_n}}$. For
  any ${i,j\leq k_n}$, set
  \begin{equation*}
    H_{\tau}^{ij}:=h(\boldsymbol{\psi}_1X,\ldots,\boldsymbol{\psi}_kX,e_{i,\tau}\boldsymbol{\phi}_1X,e_{j,\tau}\boldsymbol{\phi}_2X,Y)\quad\text{
      for }\tau\in\group\;,
  \end{equation*}
  where 
  Let $\pi$ be an element of $\group$. If
  \begin{equation*}
     Y\condind\, X\,|\,\sigma(\group)
     \quad\text{ and }\quad
     k\leq \delta_{ij}(e_{i,\tau}\boldsymbol{\phi}_1,e_{j,\tau}\boldsymbol{\phi}_2,\braces{\bpsi_1,\ldots,\bpsi_m})
  \end{equation*}
  for both ${\tau=\pi}$ and the identity ${\tau=e}$, then
  \begin{equation*}
    \bigl\|\mean[H^{ij}_{\pi}|\group,Y]-\mean[H^{ij}_e|\group,Y]\bigr\|_1
    \;\leq\;
    4\|H^{ij}_{\pi}-H^{ij}_e\|_{\frac{2+\varepsilon}{2}}\alpha_n(k|\group)^{\frac{2}{2+\varepsilon}}
  \end{equation*}
        for any ${\varepsilon\geq 0}$.
\end{lemma}
Since the proof is almost identical to that of \cref{lemma:mixing}, we only
highlight the requisite changes.
\begin{proof} 
  If ${\|H^{ij}_{\pi}-H^{ij}_e\|_{\infty}}$ is finite, again use
  \eqref{eq:lemma:mixing:hast}, now with measurable sets
  ${A_i}$ in $\xspace_n^m$ and $B_i$ in
  $\xspace_n^2$,
  and define $H^{ij*}_\tau$ by substituting $h^*$ for $h$. 
  For ${\delta>0}$ given, choose $h^*$ such that
  ${\|(H_{\pi}^{ij}-H^{ij}_e)-(H^{ij\ast}_{\pi}-H^{ij\ast}_e)\|_1\leq\delta}$.
  If we change the definition of $I_i$ to
  \begin{equation*}
      I_i\,:=\,\mathbb{I}_{A_i}(\boldsymbol{\psi}_1X,\dots,\boldsymbol{\psi}_kX)\,
      \mathbb{I}_{C_i}(Y)\,
      \bigl(\mathbb{I}_{B_i}(\boldsymbol{\phi}_1X,\boldsymbol{\phi}_2X) -
      \mathbb{I}_{B_i}(e_{i,\pi}\boldsymbol{\phi}_1X,e_{j,\pi}\boldsymbol{\phi}_2X)\bigr)\;,
  \end{equation*}
  repeating \eqref{eq:lemma:mixing:comp} shows
  ${\tsum_{i\leq N_{\delta}} |c_i|\|E_i\|_{1}\le
    2\|H^{ij}_{\pi}-H^{ij}_e\|_{\infty}\,\alpha_n(k|\mathbb{G})}$,
  and hence
  \begin{equation*}
    \|\mathbb{E}[H^{ij}_{\pi}-H^{ij}_e|\group,Y]\|_1
    \le 2
    \|H^{ij}_{\pi}-H^{ij}_e\|_{\infty}~\alpha_n(k|\mathbb{G})\;.
  \end{equation*}
  If ${\|H^{ij}_{\pi}-H^{ij}_e\|_{\infty}}$ is infinite, we set
  ${\Delta H:=H^{ij}_{\pi}-H^{ij}_e}$. 
  Repeating the argument in the previous proof shows 
  ${\|\Delta H_r\|_1\leq  2r\alpha_n(k|\mathbb{G})}$
  and ${\|\overline{\Delta H_r}\|_1\leq 2r^{-\frac{\varepsilon}{2}}}$
  for any ${r>0}$,
  and hence ${\|\mean[H^{ij}_\pi-H^{ij}_e|\group,Y]\|_1\leq
  4\alpha_n(k|\group)^{\frac{\epsilon}{2+\varepsilon}}}$.
\end{proof}
The next result is used to relate mixing to the growth of volume under
the metric $d$. We phrase this in terms of a generic function $g$,
which is later chosen as 
$t\mapsto\alpha(t|\group)^{\frac{\epsilon}{2+\epsilon}}$ in the basic
case, and $t\mapsto\alpha_n(t|\group)^{\frac{\epsilon}{2+\epsilon}}$
in the general case.

\begin{lemma}
  Let ${g:[0,\infty)\rightarrow[0,\infty)}$ be a measurable function. Then
  \begin{equation*}
    \frac{\sum_{i\ge m}|\B_{i+1}\setminus \B_{i}|g(i)}{\int_{\group\setminus\B_{m-1}}g(d(e,\phi)) |d\phi| }<\infty\qquad\text{ for all }m\in\mathbb{N}\;.
  \end{equation*}
\end{lemma}
\begin{proof}
  Abbreviate ${r:=\sup_i \frac{|\B_{i+1}\setminus \B_i|}{|\B_{i}\setminus \B_{i-1}|}}$. Then
  \begin{equation*}
    \begin{split}
      {\textstyle\sum_{i\ge m}}|\B_{i+1}\!\setminus\!\B_{i}| g(i)
      \le r{\textstyle\sum_{i\ge m}}|\B_{i}\!\setminus\!\B_{{i-1}}|g(i)
      \le r \myint_{\group\setminus\B_{m-1}}\!\!\!\!\!g(d(e,\phi))|d\phi| \;,
  \end{split}
  \end{equation*}
  where we have used \eqref{eq:metric:condition}.
\end{proof}
Finally, we note that assuming
${\mean[f(X)|\group]=0}$ incurs
no loss of generality:
\begin{lemma}[Conditional centering]
\label{lemma:centering}
  Let $X$ be $\group$-invariant, and ${p\geq 1}$. For any
  ${g\in\L_p(X)}$, the random function 
  ${f(\argdot):=g(\argdot)-\mean[g(X)|\group]}$ is $\sigma(\group)$-measurable random element of $\L_p(X)$.
  For all ${n\in\mathbb{N}}$, 
  \begin{equation*}
    \mathbb{F}_n(f,X)=\cF_n(g,X)
    \quad\text{ and }\quad
    \alpha_f(n|\group)=\alpha_g(n|\group)\quad\text{ almost surely,}
  \end{equation*}
  where $\alpha_\argdot(n|\group)$ is the conditional
  mixing coefficient defined by $(\argdot,X)$.
\end{lemma}
\begin{proof}
  For ${p\geq 1}$, $\L_p$-norms contract under conditioning \citep{Kallenberg:2001}. That makes
  $f$ a $\sigma(\group)$-measurable random element of
  ${\L_p(X)}$.
  Since
  ${f(\phi X)=g(\phi X)-\mean[g(X)|\group]}$
  for any ${\phi\in\group}$, we have
  ${\mathbb{F}_n(f,X)=\cF_n(g,X)}$. 
  To prove the second claim, consider events
  ${A\in\sigma_f(\{\phi_1,\phi_2\})}$ and ${B\in\sigma_f(G)}$, for any
  ${G\subset \group}$ and
  ${\phi_1,\phi_2\in\group\setminus\mathbf{B}_t(G)}$.
  Fix any ${\delta>0}$. 
  By definition of $\sigma_f$, we can choose sets
  ${S_i\in\sigma_g(\phi_1,\phi_2)}$, sets
  ${T_i\in\sigma(\group)}$, and constants ${c_i\in[0,1]}$ such that
  ${\|\tsum_ic_i\mathbb{I}(S_i,T_i)-\mathbb{I}(A)\|_1\le \delta}$.
  As the sets $T_i$ are in $\sigma(\group)$,  we have
     \begin{align*}
       &
       \|\msum_i c_i\big(\mathbb{P}(S_i,T_i,B|\group)-P(S_i,T_i|\group)P(B|\group)\big)\|_1
       \\&
       =\|\msum_i c_i \mathbb{I}(T_i)\big(\mathbb{P}(S_i,B|\group)-P(S_i|\group)P(B|\group)\big]\|_1
\le \alpha(t|\group)\;,
\end{align*}
where the final inequality uses the definition $\alpha$ and ${c_i\in[0,1]}$. As $\delta$ is arbitrary, this implies ${\|P(A,B|\group)-P(A|\group) P(B|\group)\|_1\le \alpha(t|\group)}$.
\end{proof}

\newpage

\section{Proofs of the basic limit theorems}
\label{proofs:basic}

We first adapt the upper bound on $\dW$ given by Stein's inequality to
our problem in \ref{app:A:bounds},
and then apply it to prove the limit theorems in \ref{app:A:CLT}.

\subsection{Bounds on the Wasserstein distance}
\label{app:A:bounds}

By \cref{lemma:centering}, it suffices to establish \cref{theorem:CLT,theorem:BE} for
elements $f$ of
\begin{equation*}
 \cL_p(X,\group):=
 \braces{f(\argdot)=g(\argdot)-\mean[f(X)|\group]\,\vert\,g\in\L_p(X)}\;.
\end{equation*}
Given ${f\in\cL_p(X,\group)}$, we choose the variable $W$ in Stein's inequality as
\begin{equation*}
  W \; := \;
  \mfrac{\sqrt{|\A_n|}}{\eta(n)}\mathbb{F}_n(f,X)
  \; = \;
  \mfrac{1}{\eta_n}\myint_{\A_n}f(\phi X)|d\phi|\quad
  \text{ where }\eta_n:=\eta(n)\sqrt{|\A_n|}\;.
\end{equation*}
Here ${\eta(n)}$ is for now any positive, $\sigma(\group)$-measurable random
variable, but will be chosen in the next section
as a specific approximation to the asymptotic variance.
For a fixed element ${\phi\in\group}$, conditional mixing allows us to control
dependence for elements $\phi'$ far away from $\phi$. To treat terms
close to $\phi$ separately, we choose ${b>0}$, and decompose $W$ into
long-range and short-range contributions,
\begin{equation*}
  W^{\phi}_{b}:=\mfrac{1}{\eta_n}\myint_{\A_n}\mathbb{I}\braces{d(\phi,\phi')\geq b}f(\phi' X)|d\phi'|
  \quad\text{ and }\quad
  \Delta^{\phi}_b:=W-W^{\phi}_{b}\;.
\end{equation*}
For our purposes, Stein's inequality then takes the following form.
\begin{lemma}
  \label{lemma:basic:stein:bound:initial}
  Assume the conditions of \cref{theorem:CLT}, and define
  $W$ as above, for a
  $\sigma(\group)$-measurable random element ${\eta(n)}$ of
  ${(0,\infty)}$. Then
   \begin{align}
    \label{eq:proof:basic:main}
    \mean[\dW(W, Z|\group)]\;&\le\;
    \sup_{t\in \mathcal{F}}\Big\|\mathbb{E}\big[\nrm \myint_{\A_n}f({\phi }X) t(W^{\phi}_{b})|d\phi|\big|\group\big]\|_1\notag
    \\ &+\;
    \sup_{t\in \mathcal{F}}
    \Big\|\mathbb{E}\big[\nrm\myint_{\A_n}f({\phi }X) (t(W)- t(W^{\phi}_{b})-\Delta^{\phi}_{b} t'(W))|d\phi||\group\big]\Big\|_{1}
    \\ &+\;
    \sqrt{\mfrac{2}{\pi}} \Big\|1-\nrm \mathbb{E}\bigl[\myint_{\A_n}f({\phi }X)  \Delta^{\phi}_{b}|d\phi|\,\big|\mathbb{G}\bigr]\Big\|_1\notag
    \\ &+\;
    \sqrt{\mfrac{2}{\pi}}  \Big\|\nrm\myint_{\A_n} f({\phi }X)  \Delta^{\phi}_{b}- \mathbb{E}[ f({\phi }X)  \Delta^{\phi}_{b}|\mathbb{G}]|d\phi|\Big\|_1 \notag\\
    & =: \text{\rm (a)} + \text{\rm (b)} + \text{\rm (c)} + \text{\rm (d)}\notag
   \end{align}
   where $Z$ is a standard normal variable and ${b>0}$.
\end{lemma}

\begin{proof} The triangle inequality yields
\begin{align*}
  \|&\mathbb{E}[Wt(W)- t'(W)|\group]\|_1\\
   &=\,
   \big\| \mathbb{E}\big[
    \myint_{\A_n}\!\!\!\tfrac{f({\phi }X)}{\eta_n}\bigl(
    t(W)\!-\!t(W^{\phi}_{b})\!+\!t(W^{\phi}_{b})\bigr)\!-\!t'(W)|d\phi|\big|\group\big]\big\|_1
    \\&
    \leq\, \big\|\mathbb{E}\big[
     \myint_{\A_n}\!\!\!\mfrac{f({\phi }X)}{\eta_n}(t(W)\!-\! t(W^{\phi}_{b}))\!-\!t'(W)
      |d\phi||\group\big]\big\|_1\notag\\
    &+\big\|\mathbb{E}\big[\myint_{\A_n}\!\!\!\mfrac{f({\phi }X)}{\eta_n} t(W^{\phi}_{b})|d\phi|\big|\group\big]\big\|_1\;.
\end{align*}
Using ${t\in\mathcal{F}}$, the first term can be bounded further as
\begin{align*}
  \big\| &\mathbb{E}\big[\!\!
    \myint_{\A_n}\!\!\!\!\!\mfrac{f({\phi }X)(t(W)- t(W^{\phi}_{b}))}{\eta_n}-t'(W)
    |d\phi|\big|\group\big]\big\|_1
  \\&\qquad\le\;
  \Big\|\mathbb{E}\big[ \!\!
      \myint_{\A_n}\!\!\!\tfrac{f({\phi }X)(t(W)- t(W^{\phi}_{b}))-\Delta^{\phi}_{b}t'(W)}{\eta_n}
      |d\phi|\big|\group\big]\!\Big\|_1
  \\&\qquad+\;
    \Big\|\mathbb{E}\big[t'(W)\bigl(1-
      \myint_{\A_n}\!\!\!\mfrac{f({\phi }X)}{\eta_n}\Delta^{\phi}_{b}
          |d\phi|\bigr)\big|\group\big]\Bigr\|_1
    \\&\qquad\le\;
     \Big\|\mathbb{E}\big[ \!\!
      \myint_{\A_n}\!\!\!\!\!\!\mfrac{f({\phi }X)(t(W)- t(W^{\phi}_{b}))-\Delta^{\phi}_{b}t'(W)}{\eta_n}
      |d\phi|\big|\group\big]\!\Big\|_1
     \\&\qquad +
   \sqrt{\mfrac{2}{\pi}} \bigl\|1\!-\!\mfrac{\mathbb{E}[\int_{\A_n}\!\!f({\phi }X)\Delta^{\phi}_{b}
       |d\phi||\group]}{\eta_n}\!\bigr\|_{1}
   \\&\qquad +
   \sqrt{\mfrac{2}{\pi}} \bigl\|\nrm\myint_{\A_n}\!\!\!f({\phi }X)\Delta^{\phi}_{b}
   -\mathbb{E}[f({\phi }X)\Delta^{\phi}_{b}|\group] |d\phi|\bigr\|_{1}\;.
\end{align*}
Substituting into the right-hand side of \eqref{eq:stein} yields the result.
\end{proof}

The main work of the proof is to control the terms
(a)--(d) in \cref{lemma:basic:stein:bound:initial}.
To handle large values of $f$, we split the function in its
range, into
\begin{equation}
\label{truncated:f}
  f^{< \gamma}(x):=f(x)\mathbb{I}\braces{|f(x)|<\gamma}
  \;\text{ and }\;
  f^{\ge \gamma}(x):=f(x)\mathbb{I}\braces{|f(x)|\ge\gamma}\;.
\end{equation}
The next result
refines the terms (a)--(d) using \cref{lemma:mixing}, and by handling
$f^{< \gamma}$ and $f^{\ge \gamma}$ separately.
\begin{lemma}
  \label{lemma:basic:stein:bound:refined}
  Require the assumptions of \cref{lemma:basic:stein:bound:initial}.
  Fix ${b>0}$ and ${\gamma>0}$, and let $\tau$ be defined as in
  \eqref{eq:def:tau}. Choose $p,q>0$ to be such that $\frac{1}{p}+\frac{1}{q}=1$. Then
\begin{align*}
&\|d(W,Z|\group)\|_1\leq\;
4\big\|\mfrac{f(X)}{\eta(n)}\big\|_{2+\varepsilon}^2
\tau(b)
+4|\B_{b}|
\big\|\mfrac{f^{\ge \gamma}(X)}{\eta(n)}\big\|_{{2+\varepsilon}}
\big\|\mfrac{f(X)}{\eta(n)}\big\|_{{2+\varepsilon}}
\\
&\qquad+\;
\mfrac{8|\B_{b}|}{\sqrt{|\A_n|}}\big\|\mfrac{f(X)}{\eta(n)}\big\|^2_{2q(1+\varepsilon/2)}
\big\|\mfrac{f^{<\gamma}(X)}{\eta(n)}\big\|_{p(1+\varepsilon/2)}
\myint_{\group}\alpha(d(e,\phi)|\group)^{\frac{\varepsilon}{2+\varepsilon}}d|\phi|
\\&\qquad+\;
\sqrt{2/\pi}\Bigl(\mathbb{E}\big[\big|\mfrac{\eta(n)^2-\eta_{b}^2 }{\eta(n)^2}\big|\big]
+
\big\|\mfrac{f(X)}{\eta(n)}\big\|^2_{2}\mfrac{|\A_n\triangle
  \B_{b}\A_n|}{|\A_n|}\Bigr)
\\&\qquad+\;
4\mfrac{|\B_{b}|}{\sqrt{|\A_n|}}
\big\|\mfrac{f^{<\gamma}(X)}{\eta(n)}\big\|_{{4+2\varepsilon}}^2
(\!\myint_{\group}\!\!\alpha(d(e,\phi)|\mathbb{G})^{\frac{\varepsilon}{2+\varepsilon}} |d\phi|)^{\frac{1}{2}}\;.
\end{align*}
\end{lemma}

\begin{proof}
To bound (a), fix any ${\delta>0}$. Then
\begin{align}
  \label{aux:bound:a}
  & \big\|\mathbb{E}\bigl[\nrm f({\phi }X)  t(W^{\phi}_{b})\big|\group\bigr]\big\|_1
  \le
  \tsum_{j\ge \lfloor|\B_{b}|/\delta\rfloor}\Big\|\mathbb{E}\Bigl[
    f(\phi
  X)\mfrac{t(W^{\phi}_{j\delta})-t(W^{\phi}_{(j+1)\delta})}{\eta_n}\Big|\group\Bigr]\Big\|_1
  \end{align}
  An application of \cref{lemma:mixing} to the summand gives
\begin{multline*}
   \Big\|\mathbb{E}\Bigl[\mfrac{f(\phi
  X)(t(W^{\phi}_{j\delta})-t(W^{\phi}_{(j+1)\delta}))}{\eta_n}\Big|\group\Bigr]\Big\|_1
  \leq
  4\Bigl\|\mfrac{f(\phi
  X)(t(W^{\phi}_{j\delta})-t(W^{\phi}_{(j+1)\delta}))}{\eta_n}\Bigr\|_{\frac{2+\varepsilon}{2}}
  \alpha(j\delta|\mathbb{G})^{\frac{\varepsilon}{2+\varepsilon}}\;.
\end{multline*}
By H\"older's inequality,
\begin{equation*}
  \Bigl\|\mfrac{f(\phi
    X)(t(W^{\phi}_{j\delta})-t(W^{\phi}_{(j+1)\delta}))}{\eta(n)}\Bigr\|_{\frac{2+\varepsilon}{2}}
  \leq
  \Big\|\mfrac{f(X)}{\eta(n)}\Big\|_{2+\varepsilon}\big\|t(W^{\phi}_{j\delta})-t(W^{\phi}_{({j+1})\delta})\big\|_{2+\varepsilon}
\end{equation*}
and since $t$ is Lipschitz,
${\|t(W^{\phi}_{j\delta})-t(W^{\phi}_{({j+1})\delta})\|_{2+\varepsilon}\leq\|W^{\phi}_{j\delta}-W^{\phi}_{({j+1})\delta}\|_{2+\varepsilon}}$.
In summary, the right-hand side of \eqref{aux:bound:a} is bounded by
\begin{equation*}
  \text{rhs \eqref{aux:bound:a}}
  \le 4 \sqrt{\mfrac{2}{\pi|\A_n|}}  \tsum_{j\ge
    \lfloor|\B_{b}|/\delta\rfloor}
  \Big\|\mfrac{f(X)}{\eta(n)}\Big\|_{2+\varepsilon}\big\|W^{\phi}_{j\delta}-W^{\phi}_{({j+1})\delta}\big\|_{2+\varepsilon} \alpha(j\delta|\mathbb{G})^{\frac{\varepsilon}{2+\varepsilon}}\;.
\end{equation*}
Since that holds for any ${\phi }\in\mathbb{G}$ and $\delta>0$, we conclude
\begin{equation*}
  \text{(a)}\;\le\;
  4\Big\|\mfrac{f(X)}{\eta(n)}\Big\|_{2+\varepsilon}^2
\myint_{\group\setminus\B_{b}} \alpha(d(e,\phi)|\mathbb{G})^{\frac{\varepsilon}{2+\varepsilon}}|d\phi|\;=\;
4\Big\|\mfrac{f(X)}{\eta(n)}\Big\|_{2+\varepsilon}^2 \tau(b)\;,
\end{equation*}
For (b), we decompose ${f=f^{<\gamma}+f^{\geq\gamma}}$. The triangle
inequality gives
\begin{align*}
  &\; \Big\|\mathbb{E}\big[\!\!\myint_{\A_n}\!\!\!f({\phi }X) \mfrac{t(W)- t(W^{\phi}_{b})-\Delta^{\phi}_{b} t'(W)}{\eta_n}|d\phi|\big|\group\big]\Big\|_1\\
  \le& \;
  \Big\|\mathbb{E}\big[\!\!\myint_{\A_n}\!\!\!f^{\ge\gamma}({\phi }X) \mfrac{t(W)- t(W^{\phi}_{b})-\Delta^{\phi}_{b} t'(W)}{\eta_n}|d\phi|\big|\group\big]\Big\|_1
  \\
  +& \;
  \Big\|\mathbb{E}\big[\!\!\myint_{\A_n}\!\!\!f^{<\gamma}({\phi }X) \mfrac{t(W)- t(W^{\phi}_{b})-\Delta^{\phi}_{b} t'(W)}{\eta_n}|d\phi|\big|\group\big]\Big\|_1 =: \text{(b1)}+\text{(b2)}\;.
\end{align*}
Since $t$ is an element of ${\mathcal{F}}$, it satisfies
\begin{equation}
  \label{eq:stein:taylor}
  |t(x+y)-t(x)-y t'(x)|\le
  2|y|\sup_{z\in [x,x+y]}|t'(z)|\quad\text{ for }x,y\in\mathbb{R}
\end{equation}
and ${\sup|t'(z)|\leq\sqrt{2/\pi}\leq 1}$. Choosing
${y=\Delta^{\phi}_{b}}$ yields
\begin{align*}
\text{(b1)} & \le 2  \Big\|\mathbb{E}\big[\nrm
\myint_{\A_n}\!\!\!|f^{\geq\gamma}(\phi X)|  |\Delta^{\phi}_{b}|~|d\phi|\big|\group\big]\Big\|_1
\\&\le  2 \Big\|\mfrac{f^{\geq\gamma}(X)}{\eta(n)}\Big\|_{{2+\varepsilon}}
\Big\|\mfrac{f(X)}{\eta(n)}\Big\|_{{2+\varepsilon}}\frac{\int_{\A_n^2}
  \mathbb{I}\braces{d({\phi },{\phi }')\le b}|d\phi||d\phi'|}{|\A_n|}
\\&\le 2|\B_{b}|
\Big\|\mfrac{f^{\geq\gamma}(X)}{\eta(n)}\Big\|_{{2+\varepsilon}}
\Big\|\mfrac{f(X)}{\eta(n)}\Big\|_{{2+\varepsilon}}
\end{align*}
For (b2), fix ${p,q>0}$ with ${1/p+1/q=1}$.
A Taylor expansion gives
\begin{equation*}
    |t(W)-t(W_b^{\phi})-\Delta_b^{\phi}t'(W)|\;\leq\;
    \frac{1}{2}\sup_{w}|t''(w)|(\Delta_b^{\phi})^2\;\leq\;(\Delta_b^{\phi})^2\;.
\end{equation*}
Substituting
${(\Delta_b^{\phi})^2=(\frac{1}{\eta_n}\int_{\A_n}\mathbb{I}\braces{d(\phi,\phi')\leq
    b}f(\phi' X)|d\phi'|)^2}$ into (b2) yields
\begin{align*}
\text{(b2)}
& \le
\Big\|\myint_{\A_n^3}\!\!\!\mfrac{\mathbb{E}\big[f^{<\gamma}(\phi X)
     \mathbb{I}\braces{d(\phi,\psi),d(\phi,\pi)\le b}
    f(\psi X)f(\pi X)\big|\group\big]}{\eta_n^3}
  |d\phi||d\psi||d\pi|\Big\|_1
  \\& \le \tfrac{8|\B_{b}|}{\sqrt{|\A_n|}}\big\|\tfrac{f(X)}{\eta(n)}\big\|^2_{{2q(1+\frac{\varepsilon}{2})}}
  \big\|\frac{f^{< \gamma}(X)}{\eta(n)}\big\|_{{p(1+\frac{\varepsilon}{2})}}  \myint_{\group}\alpha^{\frac{\varepsilon}{2+\varepsilon}}(d(e,\phi)|\group)d|\phi|\;.
\end{align*}
To bound (c), write ${\eta_b^2:=\int_{\phi\in \B_b}\eta^2(\phi)|d\phi|}$ again apply the triangle inequality, which
yields
\begin{align*}
  &\text{(c)}\cdot\sqrt{\mfrac{\pi}{2}} = \Bigl\|\mfrac{\eta(n)^2-\int_{\A_n^2}\frac{1}{|\A_n|}
    \mathbb{E}[\mathbb{I}\braces{d(\phi,\phi')\le b}
      f({\phi }X) f({\phi }'X)|\mathbb{G}]
    |d\phi||d\phi'|}{\eta(n)^2}\Bigr\|_1
  \\& \le
  \mathbb{E}\Big[\Big|\mfrac{\eta(n)^2-\eta_{b}^2 }{\eta(n)^2}\Big|\Big]
  +
  \Big\|\mfrac{\eta_{b}^2-\int_{\A_n^2}|\A_n|^{-1}\mathbb{E}[\mathbb{I}\braces{d(\phi,\phi')\le b}
      f({\phi }X) f({\phi }'X)|\mathbb{G}]
    |d\phi||d\phi'|}{\eta(n)^2}\Big\|_1
 \\&\le
   \mathbb{E}\Big[\Big|\mfrac{\eta(n)^2-\eta_{b}^2 }{\eta(n)^2}\Big|\Big]
   +
   \Big\|\mfrac{f(X)}{\eta(n)}\Big\|^2_{2}\mfrac{|\A_n\triangle \B_{b}\A_n|}{|\A_n|}.
\end{align*}
For (d), we again use ${f=f^{<\gamma}+f^{\geq\gamma}}$ and the triangle
inequality.
For a pair ${(\phi_1,\phi_2)}$ of group elements, abbreviate
\begin{equation*}
F_{\phi_1\phi_2}^{<\gamma}:=\mfrac{1}{\eta(n)^2}\bigl(f^{<\gamma}(\phi_1 X)f^{<\gamma}(\phi_2 X)-\mathbb{E}[f^{<\gamma}(\phi_1 X)f^{<\gamma}(\phi_2 X)|\group]\bigr)\;,
\end{equation*}
and define ${F_{\phi_1\phi_2}^{\geq\gamma}}$ as ${F_{\phi_1\phi_2}^{\leq\infty}}-{F_{\phi_1\phi_2}^{\leq\gamma}}$.
For any quadruple ${\phi_1,\ldots,\phi_4\in\group}$,
\begin{align*}
  \big\|\text{\rm Cov}[F_{\phi_1\phi_2}^{<\gamma},F_{\phi_3\phi_4}^{<\gamma}|\group]\big\|_1
  \le &\;
  4\Big\|\mfrac{f^{\leq\gamma}(X)}{\eta(n)}\Big\|^4_{4+2\varepsilon}
  \alpha
  \big(d((\phi_1,\!\phi_2),\!(\phi_3,\!\phi_4))|\group\big)^{\frac{\varepsilon}{2+\varepsilon}}
\end{align*}
holds by \cref{lemma:mixing}, which implies
\begin{align*}
  \Big\|\myint_{\A_n\times \A_n \B_{b}}\!\!\!F_{\phi_1\phi_2}^{<\gamma}\mfrac{|d\phi_1||d\phi_2|}{|\A_n|}\Big\|_1
  \le \mfrac{4|\B_{b}|}{\sqrt{|\A_n|}}
  \Big\|\mfrac{f^{<\gamma}(X)}{\eta(n)}\Big\|_{{4+2\varepsilon}}^2
  (\myint_{\group}\alpha(d(e,\phi)|\mathbb{G})^{\frac{\varepsilon}{2+\varepsilon}} |d\phi|)^{\frac{1}{2}}\;.
\end{align*}
For $f^{\geq\gamma}$, we obtain
\begin{equation*}
\Big\|\myint_{\A_n
\times \A_n \B_{b}}\!\!F_{\phi_1,\phi_2}^{\geq\gamma}\tfrac{|d\phi_1||d\phi_2|}{|\A_n|}\Big\|_{\subL1}
 \le2|\B_{b}|
 \Big\|\mfrac{f^{\ge\gamma}(X)}{\eta(n)}\Big\|_{2}\Big\|\mfrac{f(X)}{\eta(n)}\Big\|_{2}
 =:\text{(d')}\;,
\end{equation*}
and hence
\begin{align*}
  \text{(d)}\!\cdot\!\mfrac{\sqrt{\pi}}{\sqrt{2}}
  \le &\; 4\mfrac{|\B_{b}|}{\sqrt{|\A_n|}}
  \Big\|\mfrac{f^{<\gamma}(X)}{\eta(n)}\Big\|_{{4+2\varepsilon}}^2
  (\!\myint_{\group}\!\!\alpha(d(e,\phi)|\mathbb{G})^{\frac{\varepsilon}{2+\varepsilon}} |d\phi|)^{\frac{1}{2}}
+
  \text{(d')}\;.
\end{align*}
Rearranging terms within (a)+(b)+(c)+(d) yields the statement.
\end{proof}

\subsection{Proof of the limit theorems}
\label{app:A:CLT}
We first prove the central limit theorem under hypothesis \eqref{H2}.
The result under hypothesis \eqref{H1}, and the Berry-Esseen bound, then follow with
minimal adjustments.
\begin{proof}[Proof of \cref{theorem:CLT} and \cref{CI} assuming \eqref{H2}]
  Set ${S_n:=\sqrt{|\A_n|}\mathbb{F}_n(X)}$, and let ${Z\sim N(0,1)}$ be independent
  of ${(X,\eta)}$. We must show ${S_n\!\darrow\!\eta Z}$.
  By \cref{lemma:mixing},
  \begin{equation*}
    \begin{split}
      \|\eta^2\|_1&\le \int_{\group}\|\mathbb{E}[f(X)f(\phi X)|\group]\|_1 |d\phi|\\&\le \|f(X)\|_{2+\epsilon}\msum_{b\in \mathbb{N}}|\B_{b+1}\setminus \B_b|\alpha(b|\mathbb{G})^{\frac{\epsilon}{2+\epsilon}}<\infty\;,
    \end{split}
  \end{equation*} which shows ${\eta<\infty}$ almost surely.
  Since $\eta Z$ and ${S_n:=\sqrt{|\A_n|}\mathbb{F}_n(X)}$ have first
  moments, ${S_n\darrow\eta Z}$ holds
  if ${\dW(S_n,\eta Z)\rightarrow 0}$, as ${n\rightarrow\infty}$.
  
  To show that is the case,
  we may assume ${f\in\cL_1(X)}$, by
  \cref{lemma:centering}.
  We first choose suitable sequences $(\gamma_n)$ and $(b_n)$.
By definition, ${|\A_n|\rightarrow\infty}$. Set ${\gamma_n:=|\A_n|^{1/6}}$. That implies
${\gamma_n\rightarrow\infty}$, and hence ${\|f^{\geq\gamma_n}(X)\|_{2+\epsilon}\rightarrow 0}$. Since $|\A_n|$ diverges,
we can choose a divergent sequence $(b_n)$ such that
\begin{equation*}
  |\B_{b_n}|\leq|\A_n|^{1/12},\quad |\B_{b_n}|
  \|f^{\geq\gamma_n}(X)\|_2\qquad\text{ and }\qquad
  \mfrac{|\A_n\triangle \B_{b_n}\A_n|}{|\A_n|}\rightarrow 0\;.
\end{equation*}
Collecting terms in \cref{lemma:basic:stein:bound:refined}, we
then have
\begin{equation*}
  r_n:=
  \mfrac{|\B_{b_n}|\gamma_n^2}{\sqrt{|\A_n|}}+|\B_{b_n}|
  \|f^{\geq\gamma_n}(X)\|_2\,\rightarrow\, 0
  \quad\text{ and }\quad
  \tilde{r}_n:=\mfrac{|\A_n\triangle \B_{b_n}\A_n|}{|\A_n|}\,\rightarrow\,
  0\;.
\end{equation*}

The next step is to construct $\eta(n)$ in \cref{lemma:basic:stein:bound:refined}
as an approximation to $\eta$. Set
${u_n:=\max\braces{r_n,\tilde{r}_n,\tau(b_n)}^{1/8}}$ and
${v_n:=\max\braces{r_n,\tilde{r}_n,\tau(b_n)}}^{-1/2}$.
As ${n\rightarrow\infty}$, we hence have
${u_n\rightarrow 0}$ and ${v_n\rightarrow\infty}$, and observe that
\begin{align}
  \label{proof:CLT:u:v}&
  \text{(i)}\;\; u_n< v_n\text{ eventually }
  \quad
  \text{(ii)}\;\;
  \frac{v_n}{u_n^3}\,\bigl(r_n+\tilde{r}_n+\tau(b_n)\bigr)\rightarrow 0
  \quad
  \text{(iii)}\;\;
   v_n P(\eta<u_n)\rightarrow 0\;.
\end{align}
Set ${\eta(n):=
  \eta\mathbb{I}\braces{\eta\in[u_n,v_n]}
  +
  u_n\mathbb{I}\braces{\eta\not\in[u_n,v_n]}}$, and note that
${\eta(n)\,\condind\, Z}$.
Then
\begin{align*}
  \dW(S_n,\eta Z)
  \;&\le\;
  \dW(S_n,\eta(n)Z)+\dW(\eta(n)Z,\eta Z)\\
  \;&\le\;
  \dW(S_n,\eta(n)Z)+\|Z\|_1\|(\eta-u_n)\mathbb{I}\braces{\eta\not\in[u_n,v_n]}\|_1\;.
\end{align*}
Since we have already shown ${\|\eta^2\|_1<\infty}$, the last term satisfies
\begin{equation*}
  \|Z\|_1\|(\eta-u_n)\mathbb{I}\braces{\eta\not\in[u_n,v_n]}\|_1\rightarrow 0
  \qquad\text{ as } u_n\rightarrow 0
  \text{ and }v_n\rightarrow\infty\;.
\end{equation*}
It thus suffices to show ${\dW(S_n,\eta(n)Z)\rightarrow0}$. Using the Markov inequality we note that \begin{align}P\big(\eta\not\in [u_n,v_n]\big)\le P\big(\eta<u_n\big)+ \frac{\|\eta^2\|_1}{v_n^2}. \end{align}Using \cref{lemma:random:scaling} with $Y=\frac{1}{\eta(n)}$,
\begin{equation*}
  \dW(S_n,\eta(n)Z)
  \;\le\;
  v_n
  \mean\big[\dW\bigl(\mfrac{S_n}{\eta(n)},\,Z\big|\group\bigr)\big]\;,
\end{equation*}
since ${1/\eta(n)\geq 1/v_n}$.
Substituting ${W=\frac{S_n}{\eta(n)}}$ into
\cref{lemma:basic:stein:bound:refined} gives
\begin{align*}
  &v_n  \mean\big[\dW\bigl(\mfrac{S_n}{\eta(n)},\,Z\big|\group\bigr)\big]
  \le \frac{v_n}{u_n^2}\Bigl(5
  \big\|f(X)\big\|^2_{2+\varepsilon}\tau(b_n)\\&\qquad+4|\B_{b_n}|
  \|f^{\geq\gamma_n}(X)\|_{2+\varepsilon}
  \big\|f(X)\big\|_{2+\varepsilon}
  +
  \frac{8|\B_{b_n}|\big\|f(X)\big\|^2_{2+\varepsilon}\gamma_n\tau(0)}{u_n\sqrt{|\A_n|}}
  \\&\qquad+
  \sqrt{2/\pi}\big(u_n^2P(\eta\not\in [u_n,v_n])+\big\|f(X)\big\|^2_{2}\tilde r_n \big)+4\mfrac{|\B_{b_n}|\gamma_n^2\sqrt{\tau(0)}}{\sqrt{|\A_n|}}\Bigr)
  \\&\le
  \frac{8v_n}{\min(u_n^3,1)}\Bigl(
  \big\|f(X)\big\|^2_{2+\varepsilon}\tau(b_n)+\max(\big\|f(X)\big\|^2_{2+\varepsilon}\tau
  (0),1) [r_n+\tilde r_n]\Bigr)
  \\[.4em]&\qquad+v_nP(\eta<u_n)+\frac{\|\eta^2\|_1}{v_n}\;.
\end{align*}
This final bound vanishes as ${n\rightarrow\infty}$: The first
term by
\eqref{proof:CLT:u:v},
the second since ${u_n\rightarrow 0}$ and ${v_n\rightarrow\infty}$.
That shows ${\dW(S_n,\eta(n)Z)\rightarrow 0}$, which
implies ${\dW(S_n,\eta Z)\rightarrow 0}$ and completes the proof.
\end{proof}

Since the Berry-Esseen bound assumes a third and fourth moment,
it can be proven by applying \cref{lemma:basic:stein:bound:refined} directly:
\begin{proof}[Proof of \cref{theorem:BE}]
  The sequence $(b_n)$ is given by hypothesis. Fix any divergent sequence $(\gamma_n)$ in $(0,\infty)$. For each $\gamma_n$,
  \begin{equation*}
  \|f(X)\mathbb{I}\braces{|f(X)\leq\gamma_n|}\|_{3(1+\frac{\epsilon}{2})}
  \;\le\;
  \|f(X)\|_{3(1+\frac{\epsilon}{2})}\;.
\end{equation*}
We can hence apply \cref{lemma:basic:stein:bound:refined} with ${p=\frac{3}{2}}$ and ${q=3}$,
and \cref{theorem:BE} follows for ${n\rightarrow\infty}$.
\end{proof}
\begin{proof}[Proof of \cref{theorem:CLT} assuming \eqref{H1}]
  There is a finite ${k\in\mathbb{N}}$ such that
${\alpha(k|\group)=0}$. We can hence repeat the argument in the
above for
${b_1=b_2=\ldots:=k}$ and ${\varepsilon=0}$, which
again yields ${d_w(S_n,\eta(n)Z)\rightarrow 0}$ for ${n\rightarrow\infty}$.

\end{proof}

\subsection{Derivation of the confidence interval}
We now prove \cref{CI}. Using the centered average
$\overline{\mathbb{F}}_n$, the statement of the theorem can
be phrased as: Under the conditions of
\cref{theorem:CLT}, and assuming $\hat{\eta}_n$ is defined using a suitable sequence $(b_n)$, 
\begin{equation*}
\limsup_{n\rightarrow \infty} P\Big(\big|\overline{\mathbb{F}}_n
(f,X)\big|> \frac{\hat
  \eta_n}{\sqrt{|\A_n|}}z_{1-\frac{\alpha}{2}}\Big)\le\alpha\;.
\end{equation*}
Since we assume the hypothesis of \cref{theorem:CLT},
we can reuse part of its proof:
The scaled average $S_n$ in that proof
assumed ${\mean[f(X)|\group]=0}$. Since we now use $\overline{\mathbb{F}}_n$.
we have ${S_n=\sqrt{\A_n}\overline{\mathbb{F}}_n(f,X)}$.
We hence already know that ${\dW(S_n,\eta Z)\rightarrow 0}$ for the
asymptotic variance $\eta$, and the
key to obtaining a confidence interval is to show that this also implies
\begin{equation}
  \label{eq:proof:CI}
  \dW(S_n,\hat{\eta}_nZ)\rightarrow 0
\end{equation}
for the empirical variance $\hat{\eta}_n$. 
The proof has three steps:
\begin{enumerate}
\item We first show that \eqref{eq:proof:CI} holds if
${\|\eta^2-\hat{\eta}^2_n\|_1\rightarrow 0}$.
\item The main technical work is then to show 
  ${\|\eta^2-\hat{\eta}^2_n\|_1\rightarrow 0}$,
  which we do using similar arguments as the proof of the central limit
  theorem.
\item Given \eqref{eq:proof:CI}, we deduce the result.
\end{enumerate}

\begin{proof}[Proof of \cref{CI}]
  {\em Step 1}.
  Since $\hat{\eta}_n$ is, by its definition, independent of $Z$, the
  triangle inequality shows
  \begin{equation*}
  \dW(S_n,\hat \eta_n Z)\le \dW(S_n,\eta Z)+\dW(\eta Z,\hat \eta_n
  Z)\le \dW(S_n,\eta Z)+\|\eta-\hat \eta_n\|_1\|Z\|_1\;.
  \end{equation*}
  Since ${\dW(S_n,\eta Z)\rightarrow 0}$, \eqref{eq:proof:CI} holds if
  $\|\eta-\hat\eta_n\|_1\rightarrow 0$. For any $\epsilon>0$, we have
  \begin{align*}
    \|\eta-\hat\eta_n\|_1&\le \|(\eta-\hat\eta_n)\mathbb{I}(\max (\eta,\hat\eta_n\le \epsilon)\|_1+\|(\eta-\hat\eta_n)\mathbb{I}(\max (\eta,\hat\eta_n)>\epsilon)\|_1
    \\
    &\le
    2\epsilon+\frac{\|\eta^2-\hat\eta_n^2\|_1}{\epsilon}\;,
  \end{align*}
  so it suffices to establish $\|\eta^2-\hat\eta_n^2\|_1\rightarrow
  0$. \\[.5em]
  {\em Step 2}.
  We first observe that, similarly to the bound of term (c) in the proof
  of \cref{theorem:CLT}, 
  \begin{align*}
    &\Big\|\hat
    \eta_n^2-\frac{1}{|\A_n|}\int_{\A_n^2}\mathbb{I}(d(\phi,\phi')\le
    b_n) f(\phi
    X)f(\phi'(X)|d\phi'||d\phi|+\Big(\frac{1}{|\A_n|}\int_{\A_n}f(\phi
    X)|d\phi|\Big)^2\Big\|_1\\
    &\le \|f(X)\|_2^2\frac{|\A_n\triangle \A_n\B_{b_n}|}{|\A_n|}\;.
  \end{align*}
  Applying the triangle inequality to $\|\eta^2-\hat\eta_n^2\|_{1}$
  hence gives
  \begin{align*}
    \|\eta^2-\hat\eta_n^2\|_{1}&\le \Bigl\|{\eta^2-\int_{\A_n^2}\frac{\mathbb{I}\braces{d(\phi,\phi')\le b_n}}{|\A_n|}
      \rm{Cov}[
      f({\phi }X), f({\phi }'X)|\mathbb{G}]
    |d\phi||d\phi'|}\Bigr\|_1
    \\&+\Big\|\frac{1}{|\A_n|}\int_{\A_n^2}\mathbb{I}(d(\phi,\phi')\le b_n)\big[f(\phi X)f(\phi'(X)-\mathbb{E}(f(\phi X)f(\phi' X)|\group)\big]|d\phi'||d\phi|\Big\|_1
    \\&+|\B_{b_n}|\Big\|\Big(\frac{1}{|\A_n|}\int_{\A_n} f(\phi X)|d\phi|\Big)^2-\mathbb{E}(f(X)|\group)^2\Big\|_{L_2}
    \\&+\|f(X)\|_2^2\frac{|\A_n\triangle \A_n\B_{b_n}|}{|\A_n|}
    \\&=: \text{(a)}+\text{(b)}+\text{(c)}+\|f(X)\|_2^2\mfrac{|\A_n\triangle \A_n\B_{b_n}|}{|\A_n|}
  \end{align*}
  Since ${{|\A_n\triangle \A_n\B_{b_n}|}/{|\A_n|}\rightarrow 0}$
  as ${n\rightarrow\infty}$ by hypothesis (iii), the last term vanishes asymptotically.
  We bound each of the remaining terms individually. For term (a),
  abbreviate ${\eta_b^2:=\int_{\phi\in
      \B_b}\eta^2(\phi)|d\phi|}$. Then
  \begin{align*}
  &\text{(a)}\le
  \mathbb{E}\big[\big|{\eta^2-\eta_{b_n}^2 }\big|\big]
  +
  \Big\|{\eta_{b_n}^2-\int_{\A_n^2}|\A_n|^{-1}\mathbb{I}\braces{d(\phi,\phi')\le b_n}\rm{Cov}[
      f({\phi }X), f({\phi }'X)|\mathbb{G}]
    |d\phi||d\phi'|}\Big\|_1
 \\&\le
 4\|f(X)\|_{2+\epsilon}^2\tau(b_n)
   +
  4 \big\|{f(X)}\big\|^2_{2}\mfrac{|\A_n\triangle\A_n\B_{b_n}|}{|\A_n|}.
  \end{align*}
  Since ${\tau(b_n)\rightarrow 0}$ by hypothesis (i), that
  implies $(a)\rightarrow 0$ as ${n\rightarrow\infty}$.
  Term (c) satisfies 
  \begin{align*}
    \text{(c)}
    &\le4\|f(X)\|_2|\B_{b_n}|\Big\|\frac{1}{|\A_n|}\int_{\A_n} f(\phi X)-\mathbb{E}(f(X)|\group)|d\phi|\Big\|_2
    \\
    &\le
    \frac{4\|f(X)\|_{2}\|f(X)\|_{2+\epsilon}\tau(0)|\B_{b_n}|}{\sqrt{|\A_n|}}
  \end{align*}
  and hence ${\text{(c)}\rightarrow 0}$ as ${n\rightarrow\infty}$, by
  hypothesis (ii).

  For term (b), we have to argue similarly as in the proof of
  \cref{theorem:CLT}: Since we do not assume a fourth moment exists,
  we must split $f$. To this end, define ${f^{<\gamma}}$ and
  ${f^{\geq \gamma}}$ as in \eqref{truncated:f}, and choose a sequence
  $(\gamma_n)$ of positive scalars satisfying
  \begin{equation*}
     \gamma_n\rightarrow \infty \qquad\text{ and }\qquad
     \gamma_n^2\,\mfrac{|\B_{b_n}|}{\sqrt{|\A_n|}}\rightarrow 0
     \qquad\text{ as }n\rightarrow\infty\;.
  \end{equation*}
For a pair ${\phi_1,\phi_2\in\group}$, abbreviate
\begin{equation*}
F_{\phi_1\phi_2}^{<\gamma_n}:=\bigl(f^{<\gamma_n}(\phi_1 X)f^{<\gamma_n}(\phi_2 X)-\mathbb{E}[f^{<\gamma_n}(\phi_1 X)f^{<\gamma_n}(\phi_2 X)|\group]\bigr)\;.
\end{equation*}
For any two pairs $(\phi_1,\phi_2)$ and $(\phi_3,\phi_4)$, we then have
\begin{align*}
  \big\|\text{\rm Cov}[F_{\phi_1\phi_2}^{<\gamma_n},F_{\phi_3\phi_4}^{<\gamma_n}|\group]\big\|_1
  \le &\;
  4\big\|{f^{\leq\gamma_n}(X)}\big\|^4_{4+2\varepsilon}
  \alpha
  \big(d((\phi_1,\!\phi_2),\!(\phi_3,\!\phi_4))|\group\big)^{\frac{\varepsilon}{2+\varepsilon}}
\end{align*}
by the conditional mixing bound in \cref{lemma:mixing}. That implies
\begin{align*}
  \Big\|\myint_{\A_n\times \A_n \B_{b_n}}\!\!\!F_{\phi_1\phi_2}^{<\gamma_n}\mfrac{|d\phi_1||d\phi_2|}{|\A_n|}\Big\|_1
&  \le \mfrac{4|\B_{b_n}|}{\sqrt{|\A_n|}}
  \big\|{f^{<\gamma_n}(X)}\big\|_{{4+2\varepsilon}}^2
  (\myint_{\group}\alpha(d(e,\phi)|\mathbb{G})^{\frac{\varepsilon}{2+\varepsilon}} |d\phi|)^{\frac{1}{2}}
\\&\le\mfrac{4|\B_{b_n}|\gamma_n^2}{\sqrt{|\A_n|}}
  (\myint_{\group}\alpha(d(e,\phi)|\mathbb{G})^{\frac{\varepsilon}{2+\varepsilon}} |d\phi|)^{\frac{1}{2}}\;.
\end{align*}
The residual term
${F_{\phi_1\phi_2}^{\leq\infty}}-{F_{\phi_1\phi_2}^{\leq\gamma_n}}$ satisfies
\begin{equation*}
\Big\|\myint_{\A_n
\times \A_n \B_{b_n}}\!\!(F_{\phi_1\phi_2}^{\leq\infty}-F_{\phi_1\phi_2}^{\leq\gamma_n})\tfrac{|d\phi_1||d\phi_2|}{|\A_n|}\Big\|_{\subL1}
 \le2|\B_{b_n}|
 \big\|{f^{\ge\gamma_n}(X)}\big\|_{2}\big\|{f(X)}\big\|_{2}\;,
\end{equation*}
and combining the two shows
\begin{align*}
  \text{(b)}
  \le &\;
  4\mfrac{|\B_{b_n}|}{\sqrt{|\A_n|}}\gamma_n^2
  (\!\myint_{\group}\!\!\alpha(d(e,\phi)|\mathbb{G})^{\frac{\varepsilon}{2+\varepsilon}} |d\phi|)^{\frac{1}{2}}
  +
  2|\B_{b_n}|
  \big\|{f^{\ge\gamma_n}(X)}\big\|_{2}\big\|f(X)\big\|_{2}\rightarrow 0.
\end{align*}
In summary, ${\text{(a)}+\text{(b)}+\text{(c)}\rightarrow 0}$ holds as
${n\rightarrow\infty}$. That implies 
${\|\eta^2-\hat\eta_n^2\|_1\rightarrow 0}$, and we have established
\eqref{eq:proof:CI}.
\\[.5em]
\emph{Step 3}.
By definition of the Wasserstein metric $\dW$, we can find a sequence
of couplings  $(S_n,\hat\eta_nZ)$ that satisfy
${\|S_n-\hat\eta_nZ\|_1\le 2\dW(S_n,\hat\eta_nZ)}$.
The hypothesis ${P(\eta<t)\rightarrow 0}$ for ${t\searrow 0}$ implies
there is a sequence $(t_n)$ of positive reals that satisfies
\begin{equation*}
  t_n\xrightarrow{n\rightarrow \infty}0
  \quad\text{ and }\quad
  \|\eta-\hat\eta_n\|_1=o(t_n)
  \quad\text{ and }\quad
  \dW\big(S_n,~\hat\eta_nZ)=o(t_n\sqrt{P(\hat\eta_n<t_n)})\;.
\end{equation*}
The truncated empirical variance
${\tilde\eta_n:=\min(\hat\eta_n,t_n)}$ then satisfies
$$\|\tilde\eta_n-\hat\eta_n\|_1\;\le\;
t_nP(\hat\eta_n<t_n)\;\xrightarrow{n\rightarrow\infty}\; 0\;.$$
Since also
\begin{equation*}
  P(\hat\eta_n<t_n)
  \;\le\;
  P(\eta<\mfrac{t_n}{2})+P(|\eta-\hat\eta_n|\ge\mfrac{t_n}{2})
  \;\le\;
  P(\eta<\mfrac{t_n}{2})+2\mfrac{\|\eta-\hat\eta_n\|_1}{t_n}
  \;\xrightarrow{n\rightarrow\infty}\;
  0\;,
\end{equation*}
it follows that ${\dW(S_n,\tilde \eta_nZ)\le \dW(S_n,\hat
  \eta_nZ)+t_nP(\hat\eta_n<t_n)\rightarrow 0}$.
To express probabilities in terms of ${\tilde{\eta}_n}$, we
define a sequence $(\epsilon_n)$ as
$\epsilon_n=t_n\sqrt{P(\hat\eta_n<t_n)}$. Then ${\epsilon_n=o(t_n)}$,
and the triangle inequality shows
\begin{align*}&P\Big(\big|\overline{\mathbb{F}}_n
(f,X)\big|>\frac{\hat \eta_n}{\sqrt{|\A_n|}}z_{1-\frac{\alpha}{2}}\Big)
  \\&\le
  P\Big(\big|\sqrt{|\A_n|}\overline{\mathbb{F}}_n
  (f,X)\big|>\tilde \eta_nz_{1-\frac{\alpha}{2}}-\epsilon_n\Big)+P\Big(\big|\hat \eta_n-\tilde \eta_n|z_{1-\frac{\alpha}{2}}>\epsilon_n\Big)
  \\&\le
  P\Big(\big|\tilde \eta_n Z\big|>\tilde \eta_nz_{1-\frac{\alpha}{2}}-2\epsilon_n\Big)+P(\big|\sqrt{|\A_n|}\overline{\mathbb{F}}_n-\tilde \eta_n Z\big|>\epsilon_n)+P\Big(\big|\hat \eta_n-\tilde \eta_n|z_{1-\frac{\alpha}{2}}>\epsilon_n\Big)
  \\&\le
  \text{(a')}+\text{(b')}+\text{(c')}
\;.
\end{align*}
We again bound each term successively. By Markov's inequality,
\begin{align*}
  \text{(c')}
  \;=\;
  P\Big(\big|\hat \eta_n-\tilde
  \eta_n|z_{1-\frac{\alpha}{2}}>\epsilon_n\Big)
  \;\le\;
  \frac{z_{1-\frac{\alpha}{2}}\|\hat\eta_n-\tilde \eta_n\|_1}{\epsilon_n}
  \;&\le\;
  \frac{t_nP(\hat\eta_n<t_n)z_{1-\frac{\alpha}{2}}}{\epsilon_n}
  \\
  &=\;z_{1-\frac{\alpha}{2}}\sqrt{P(\hat\eta_n<t_n)}
  \;\rightarrow\; 0
\;.
\end{align*}
Since the coupling ${(S_n,\hat\eta_nZ)}$ is chosen to  satisfy ${\|\!\sqrt{|\A_n|}\,\overline{\mathbb{F}}_n-\hat\eta_n Z\|_1\le 2\dW(S_n,\hat\eta_n)}$,
\begin{align*}
  \text{(b')}
  \;&=\;
  P(\big|\sqrt{|\A_n|}\overline{\mathbb{F}}_n(f,X)-\tilde \eta_n Z\big|>\epsilon_n)
  \;\le\;
  \frac{\big\|\sqrt{|\A_n|}\overline{\mathbb{F}}_n(f,X)-\tilde \eta_n Z\big\|_1}{\epsilon_n}
  \\
  &\le\;
  \frac{\big\|\sqrt{|\A_n|}\overline{\mathbb{F}}_n(f,X)-\hat\eta_n
    Z\big\|_1+\|\tilde \eta_n-\hat\eta_n\|_1}{\epsilon_n}
  \;\le\;
  \frac{2\dW(S_n,\hat\eta_n)+ t_nP(\hat\eta_n<t_n)}{\epsilon_n}\rightarrow 0.
\end{align*}
Finally,
\begin{align*}
  \text{(a')}\;=\;
  P\Big(\big|\tilde \eta_n Z\big|>\tilde
  \eta_nz_{1-\frac{\alpha}{2}}-2\epsilon_n\Big)
  &\;=\;
  P\Big(\big|Z\big|> z_{1-\frac{\alpha}{2}}-\frac{2\epsilon_n}{\tilde \eta_n}\Big)
  \;\le\; P\Big(\big|Z\big |>z_{1-\frac{\alpha}{2}}-\frac{2\epsilon_n}{ t_n}\Big)\\
  &\;\le\; P\Big(\big|Z\big
  |>z_{1-\frac{\alpha}{2}}-2\sqrt{P(\hat\eta_n<t_n)}\Big)
  \;=\;
  \alpha+o_n(1)
  \;.
\end{align*}
Substituting the upper bounds on (a'), (b') and (c') into the bound
above, we obtain
\begin{equation*}
  \limsup_nP\Big(\big|\overline{\mathbb{F}}_n
(f,X)\big|>\frac{\hat
    \eta_n}{\sqrt{|\A_n|}}z_{1-\frac{\alpha}{2}}\Big)\le \alpha\;,
\end{equation*}
which is the statement of the theorem.
\end{proof}

\newpage

\section{Proofs of the general limit theorems}
\label{proofs:general}

We next prove \cref{theorem:CLT:g,theorem:BE:g}.
Recall that the proof in the basic case adapts Stein's inequality
in \cref{lemma:basic:stein:bound:initial}, bounds the constituent terms individually, and
then deduces both limit theorems from this bound.
The structure in the general case is similar:
\cref{neuchatel} below substitutes for
\cref{lemma:basic:stein:bound:initial}, and
the main work is again to upper-bound each term on its right-hand side, which
we do in Sections
\ref{sec:proofs:individual:bounds}--\ref{sec:term:4}.
The theorems are then deduced in Sections \ref{sec:proof:CLT:g} and \ref{sec:proof:BE:g}.
Although the steps remain similar, the terms in the bounds
change:\\[.5em]
\myitem\tab The generalization of invariance
to \eqref{eq:generalized:invariance} makes the dependency
neighborhoods (which above were balls of radius $b_n$
around group elements) more complicated.
\\[.2em]
\myitem\tab
The fact that $k_n$ may grow with $n$ complicates terms involving
$f_n$. Their moments are handled using 
telescopic sums $\bar{h}_n^i$, defined below.\\[.2em]
\myitem\tab
Large values of $f$ were previously controlled using
${f(x)\mathbb{I}\braces{|f(x)|<\gamma}}$ and its
remainder. Similar quantities now have to be phrased in terms of
$\bar{h}_n^i$ and the coefficients $c_{i,p}$.\\[.2em]
\myitem\tab
Randomized averages have to be phrased in terms of $\mu_n$,
see the definitions of $P_{\mu_n}$ and $\mean_{\mu_n}$ below.\\[.2em]
\myitem\tab
Since we have to control the influence of randomization by
$\mu_n$, spreading coefficients $\mathcal{S}^n$ or $\mathcal{S}^n_w$
appear in the bounds.\\[.2em]
\myitem\tab
Since we make no specific restrictions on how a group action may apply
the entries of a vector ${\bphi\in\group^{k_n}}$, arguments that
compare pairs of such vectors often have to compare all possible
combinations of coordinates.
\\[.5em]
As a result, the bounds become lengthy, and we first introduce
some additional notation to summarize quantities that occur frequently.

\subsection{Notation}
Recall that sequences $(k_n)$ and $(b_n)$ are given by hypothesis.
In addition, we will use a non-decreasing integer sequence $(k_n')$
with ${k_n'\leq k_n}$.
In the proofs, the functions $f_n$ always appear in a centered form,
which is the (random) function
\begin{equation*}
  h_n(X_n):=f_n(X_n)-\mean[f_n(X_n)|\group]\;.
\end{equation*}
We frequently have to restrict random measures to subsets.
If $\mu$ is a random measure on $\group^{k_n}$ and $A$ a measurable subset, write
\begin{equation*}
  P_\mu(\argdot|A):=\frac{\mu(\argdot\cap A)}{\mu(A)}
\end{equation*}
provided ${\mu(A)>0}$ almost surely. Since ${P_\mu(\argdot|A)}$ is almost surely a probability measure 
even if $\mu$ is not, the usual rules of conditioning apply and explain expressions such
as ${P_\mu(\argdot|A,Y)}$ for a random quantity $Y$. If $f$ is a measurable function on
${\group^{k_n}}$, set
\begin{equation*}
  \mathbb{E}_{\mu}[f(\boldsymbol{\phi}) |A]:=\int f(\bphi)P_{\mu}(d\bphi|A)=\frac{1}{\mu(A)}\int_{A}f(\boldsymbol{\phi}) \mu(d\boldsymbol{\phi})\;.
\end{equation*}
The distance $\dW(W_n,Z)$ in Stein's inequality is then applied
to
\begin{equation*}
  W_n:= \tfrac{\sqrt{|\A_n|}}{\eta(n)}
  \mathbb{E}_{\mu_n}[h_n(\boldsymbol{\phi}X_n)|\A_n^{k_n}]\;.
\end{equation*}
Recall from the proof overview that Stein's method considers dependency
neighborhoods around an index $i$. We generalize
these to sets of coordinates of a vector $\bphi$ that are similar 
to ${\psi\in\group}$,
\begin{equation*}
  \mathcal{I}_{b,k}(\psi,\boldsymbol{\phi}):=\{i\le k:~ d(\psi,\phi_i)\le b\} \quad\text{ for }
  k\leq k_n,b>0\;.
\end{equation*}
Two types of averages of $h_n$ appear in the upper bounds on ${\dW}$.
One holds entries outside a neighborhood
$\mathcal{I}_{b,k}(\psi,\boldsymbol{\phi})$, of size
${I:=|\mathcal{I}_{b,k}(\psi,\boldsymbol{\phi})|}$, fixed,
\begin{equation*}
\bar{h}_n^{\psi, b,k}(\bphi X_n):= \lim_{m\rightarrow \infty} \frac{1}{|\A_m|^{I}}
\int_{
  \braces{\boldsymbol{\theta}\in\A_m^{k_n}|\theta_i=\phi_i\text{ for } i\,\not\in\,\mathcal{I}_{b,k}(\psi,\bphi)}
}h_n(\boldsymbol{\theta}X_n)|d\boldsymbol{\theta}|^{\otimes I}\;.
\end{equation*}
The other appears in particular in the context of moments. It fixes
the first ${k_n-i}$ entries, and can be written as
a telescopic sum
\begin{equation*}
\bar{h}_n^i(\bphi X_n):=
g_n^i(\bphi X_n)-
g_n^{i+1}(\bphi X_n)
\end{equation*}
with summands
\begin{equation*}
g_n^i(\bphi X_n):= \lim_{m\rightarrow \infty}
\frac{1}{|\A_m|^{i}}\int_{\A_m^i}h((\phi_1,\ldots,\phi_{k_n-i},\theta_{1},\ldots,\theta_{i})X_n)|d\theta_{1}|\cdots|d\theta_{i}|\;.
\end{equation*}
Higher moments of ${h_n}/{\eta(n)}$ are controlled using
a sequence ${(\gamma_n)}$ with ${\gamma_n\rightarrow\infty}$. That
leads to bounds involving the terms
\begin{equation*}
  \Gamma_{i,p}(\gamma_n):=\sup_{\phi \in \mathbb{G}^{k_n}}\Big\|\mfrac{\bar{h}_n^i(\boldsymbol{\phi}X_n)
    \mathbb{I}\braces{{|\bar{h}_n^i(\boldsymbol{\phi}X_n)|\le \gamma_nc_{i,2}(h_n)}}}{\eta(n)}\Big\|_p
  \qquad\text{ for }i\le k_n\;.
\end{equation*}
More generally, for any function $f_n$ on $\xspace_n$ and the coefficients $c_{i,p}$ defined in \cref{sec:g}, we write
\begin{equation*}
  M_p(f_n):=\sup_{\;\;\;\boldsymbol{\phi}\in\group^{k_n}}\big\|\tfrac{f_n(\boldsymbol{\phi} X_n)}{\eta(n)}\big\|_{{p}}
  \quad\text{ and }\quad
  C_p(f_n):=\tsum_{i=1}^{\infty}c_{i,p}(f_n)\;.
\end{equation*}
Terms in the bounds that quantify the behavior of $\mu_n$ involve
vectors ${\boldsymbol{\phi}\in \mathbb{G}^{k_n}}$ whose entries
are ``not too close'' to each other. To this end, we write
\begin{equation*}
  \dself(\boldsymbol{\phi}):= \min_{i\ne j} d(\phi_i,\phi_j)\;.
\end{equation*}
In particular, we must consider
${\mu_n^{\ast}(\argdot):=
\mu_n(\argdot\cap\braces{\bphi|\dself(\bphi)\geq b_n})}$. This is
again a random measure on $\group^{k_n}$, with
\begin{equation}\begin{split}
  \label{gruyere}&
  P_{\mu^*_n}(\bphi\in\argdot|\A_n^{k_n})=\mathbb{E}_{\mu_n}\big[\mathbb{I}\braces{\boldsymbol{\phi}\in\argdot, \dself(\boldsymbol{\phi})\ge b_n}\big|\A_n^{k_n}\big]\;.
\end{split}\end{equation}
Moments of $\mu_n$ are controlled using a
sequence $(\beta_n)$ with ${\beta_n\rightarrow\infty}$. They lead to
rather complicated terms, which we encapsulate using
the sets
\begin{equation*}
  V_{i,\beta_n}(n) := \Bigbraces{\bpsi\in \mathbb{G}^{k_n}\,\Big\vert\,\sup_{j\le k'_n}\frac{|\A_n|}{|\B_{b_n}|}P_{\mu^*_n}(d(\phi_i,\psi_j)\le b_n|\A_n^{k_n},\bpsi)\le k'_n \beta_n}\;.
\end{equation*}
In words, a random vector $\bphi$ is generated by $P_{\mu_n}$, 
conditionally on its entries not being too similar (hence $P_{\mu_n^{\ast}}$), and the set contains
those vectors $\bpsi$ unlikely to have an entry similar to $\phi_i$.
Finally, for a strongly well-spread sequence, the spreading coefficient $\mathcal{S}^n$ was defined in
\cref{sec:g}. A similar coefficients in the well-spread case is
\begin{align*}
  \mathcal{S}^n_w&:=
  \sup_{{A\in \Sigma_{n},n\in\mathbb{N}} }
  \mean\Big[\mfrac{1}{\mathbb{T}_n(A,|\argdot|^{\otimes k_n})}P_{\mu_n\otimes \mu_n}\big((\bphi,\bphi')\in A
  \big| \A_n^{2k_n}\big)\Big]\;.
\end{align*}

\subsection{Main lemmas}

We first bound the error incurred by excluding vectors whose entries are close to each other,
i.e.\ of substituting $\mu^{\ast}_n$ for $\mu_n$:
\begin{lemma}
  \label{aix_la_chapelle}
  For a positive random variable $\eta(n)$ with
  ${\eta(n)\,\condind_\group\, X_n}$ and a standard normal variable
  $Z^*$, write
  \begin{equation*}
    E(\mu_n):=\tfrac{\sqrt{|\A_n|}}{\eta(n)}
    \mathbb{E}_{\mu_n}[h_n(\boldsymbol{\phi}X_n)|\A_n^{k_n}]\;.
    \end{equation*}
      Then
      \begin{equation*}
    \begin{split}
      \|\dW(E(\mu_n),Z^*|\group) - \dW(E(\mu_n^*),Z^*|\group)\|_1
     \le \frac{k_n^2 C_{1}(\frac{h_n}{\eta(n)})|\B_{b_n}|  \mathcal{S}^n_w}{\sqrt{|\A_n|}}\;.
  \end{split}
  \end{equation*}
\end{lemma}
\begin{proof}
By definition of the Wasserstein distance,
\begin{align*}
  &\|\dW(E(\mu_n),Z^*|\group) -
  \dW(E(\mu_n^*),Z^*|\group)\|_1\leq\|\dW(E(\mu_n),E(\mu_n^*)|\group)\|_1\\
  &\leq\; \|E(\mu_n)-E(\mu_n^*)\|_1
\;\le\; \Big\| \tfrac{\sqrt{|\A_n|}}{\eta(n)} \mathbb{E}_{\mu_n}[\mathbb{I}\braces{\dself(\boldsymbol{\phi})\le b_n}h_n(\boldsymbol{\phi}X_n)|\A_n^{k_n}]\Big\|_{\subL1}.
  \end{align*}
We bound the final term:  Since $\mu_n$ and $X_n$ are independent, we can apply the definition of
 the spreading coefficient $\mathcal{S}_w^n$ to obtain
  \begin{align*}
    &  \Big\| \tfrac{\sqrt{|\A_n|}}{\eta(n)} \mathbb{E}_{\mu_n}[\mathbb{I}\braces{\dself(\boldsymbol{\phi})\le b_n}h_n(\boldsymbol{\phi}X_n)|\A_n^{k_n}]\Big\|_{\subL1}
\\ & \le M_1\bigl(\tfrac{h_n}{\eta(n)}\bigr)\mathbb{E}[\sqrt{|\A_n|}P_{\mu_n}(\dself(\boldsymbol{\phi})\le b_n|\A_n^{k_n})]
\\ & \le \frac{ k_n^2 M_1\bigl(\frac{h_n}{\eta(n)}\bigr)|\B_{b_n}|}{\sqrt{|\A_n|}} \sup_{i\ne j}\mathbb{E}\big[\tfrac{ {|\A_n|}}{|\B_{b_n}|}P_{\mu_n^*}(\mathbb{I}\braces{\boldsymbol{\phi}_i^{-1}\boldsymbol{\phi}_j\in B_{b_n}}|\A_n^{k_n})\bigr]
\\ & \le \frac{ k_n^2 M_1\bigl(\tfrac{h_n}{\eta(n)}\bigr)|\B_{b_n}|  \mathcal{S}^n_w}{\sqrt{|\A_n|}}\;,
  \end{align*}
  which yields the desired result.
\end{proof}
\noindent The main bound on the Wasserstein distance is formulated in terms of $\mu_n^*$:
\begin{lemma}\label{neuchatel}
  Let $\eta(n)$ be a positive random variable with ${\eta(n)\condind_{\group}X_n}$, and $\mathcal{F}$ the function
  class \eqref{stein_class}. Let
  \begin{equation*}
    W^*:= {\textstyle\tfrac{\sqrt{|\A_n|}}{\eta(n)}\tsum_i}\mathbb{E}_{\mu^*_n}[\bar{h}_n^i(\boldsymbol{\phi}X_n)|\A_n^{k_n}]
    \;.
  \end{equation*}
  For given sequences $(b_n)$ and $(k_n')$, abbreviate
  \begin{equation*}
    W^{\boldsymbol{\phi}}_{in}:={\textstyle\tfrac{\sqrt{|\A_n|}}{\eta(n)}}\mathbb{E}_{\mu^*_n}[\bar{h}_n^{\boldsymbol{\phi}_i,b_n,k'_n}(\boldsymbol{\phi}'X_n)|\A_n^{k_n}]
    \quad\text{ and }\quad
    \Delta^{\boldsymbol{\phi}}_{in}=W^*-W^{\boldsymbol{\phi}}_{in}\;.
  \end{equation*}
  Then, for an independent variable $Z^*\sim N(0,1)$,
  \begin{equation*}\begin{split}
      &\big\|\dW(W^*, Z^*|\group)\big\|_1
    \\&\le
    \sup_{t\in \mathcal{F}}\Big\|\mathbb{E}\big[\tfrac{\sqrt{|\A_n|}}{\eta(n)}\tsum_i \mathbb{E}_{\mu_n^*}\big[\bar{h}_n^i(\boldsymbol{\phi}X_n) t(W^{\boldsymbol{\phi}}_{in})|\A_n^{k_n}\big]\big|\group\big]\Big\|_1
    \\&+\sup_{t\in \mathcal{F}}
    \Big\|\mathbb{E}\big[\tfrac{\sqrt{|\A_n|}}{\eta(n)}\tsum_i\mathbb{E}_{\mu_n^*}\big[\bar{h}_n^i(\boldsymbol{\phi}X_n) (t(W^*)- t(W^{\boldsymbol{\phi}}_{in})-\Delta^{\boldsymbol{\phi}}_{in} t'(W^*))|\A_n^{k_n}\big]\big|\group\big]\Big\|_1
    \\&+\sqrt{\tfrac{2}{\pi}} \Big\|1-\tfrac{\sqrt{|\A_n|}}{\eta(n)}\tsum_i\mathbb{E}[\mathbb{E}_{\mu_n^*}[\bar{h}_n^i(\boldsymbol{\phi}X_n) \Delta^{\boldsymbol{\phi}}_{in}|\A_n^{k_n}]|\mathbb{G}]\Big\|_{\subL1}
    \\&+ \sqrt{\tfrac{2}{\pi}}\tsum_i \Big\|\tfrac{\sqrt{|\A_n|}}{\eta(n)}\mathbb{E}_{\mu_n^*}\big[ \bar{h}_n^i(\boldsymbol{\phi}X_n) \Delta^{\boldsymbol{\phi}}_{in}- \mathbb{E}[ \bar{h}_n^i(\boldsymbol{\phi}X_n) \Delta^{\boldsymbol{\phi}}_{in}|\mathbb{G}] \big|\A_n^{k_n}\big]\Big\|_{\subL1}\;.
  \end{split}\end{equation*}
\end{lemma}

\begin{proof}
  By Stein's inequality, ${\dW(W^*, Z^*)\le \sup |\mathbb{E}[W^*t(W^*)- t'(W^*)]|}$.
  We decompose the right-hand side: Since
  ${h_n=\sum_i \bar{h}_n^i}$,
  \begin{multline*}
    \|\mathbb{E}[W^*t(W^*)- t'(W^*)|\group]\|_1
    \le
   \big\| \mathbb{E}\big[\tfrac{\sqrt{|\A_n|}}{\eta(n)}\tsum_i \mathbb{E}_{\mu_n^*}\big[\bar{h}_n^i(\boldsymbol{\phi}X_n) t(W^{\boldsymbol{\phi}}_{in})|\A_n^{k_n}\big]\big|\group\big]\big\|_1
    \\
    +\Big\|\mathbb{E}\Big[\tfrac{\sqrt{|\A_n|}}{\eta(n)}\tsum_i
      \mathbb{E}_{\mu_n^*}\big[\bar{h}_n^i(\boldsymbol{\phi}X_n)(t(W^*)- t(W^{\boldsymbol{\phi}}_{in}))
        \big|\A_n^{k_n}\big]-t'(W^*)\Big|\group\Big]\Big\|_1\;.
  \end{multline*}
  The final term can be bounded further using the triangle inequality as
  \begin{align*}
    &\;\Big\|\mathbb{E}\Big[
      {\textstyle\frac{\sqrt{|\A_n|}}{\eta(n)}\sum_i }
      \mathbb{E}_{\mu_n^*}\big[\bar{h}_n^i(\boldsymbol{\phi}X_n)\big (t(W^*)- t(W^{\boldsymbol{\phi}}_{in})\big) \big|\A_n^{k_n}\big]-t'(W^*)\Big|\group\Big]\Big\|_1
    \\
    \le&\;\Big\|\mathbb{E}\big[
    {\textstyle\frac{\sqrt{|\A_n|}}{\eta(n)}\sum_i}
    \mathbb{E}_{\mu_n^*}\big[\bar{h}_n^i(\boldsymbol{\phi}X_n) (t(W^*)- t(W^{\boldsymbol{\phi}}_{in})-\Delta^{\boldsymbol{\phi}}_{in} t'(W^*))|\A_n^{k_n}\big]\big|\group\big]\Big\|_1
    \\+&\;\Big\|\mathbb{E}\big[
      {\textstyle\sum_it'(W^*)}\big(1-\mathbb{E}_{\mu_n^*}\big[{\textstyle\frac{\sqrt{|\A_n|}}{\eta(n)}}\bar{h}_n^i(\boldsymbol{\phi}X_n) \Delta^{\boldsymbol{\phi}}_{in} |\A_n^{k_n}\big]\big)\big|\group\big]\Big\|_1
\\ \overset{(*)}{\le}&\; \Big\|\mathbb{E}\big[
{\textstyle\frac{\sqrt{|\A_n|}}{\eta(n)}\sum_i}\mathbb{E}_{\mu_n^*}\big[\bar{h}_n^i(\boldsymbol{\phi}X_n) (t(W^*)- t(W^{\boldsymbol{\phi}}_{in})-\Delta^{\boldsymbol{\phi}}_{in} t'(W^*))|\A_n^{k_n}\big]\big|\group\big]\Big\|_1
\\+&\; {\textstyle\sqrt{\frac{2}{\pi}}} \big\|
1-{\textstyle\frac{\sqrt{|\A_n|}}{\eta(n)}\sum_i}\mathbb{E}[\mathbb{E}_{\mu_n^*}[ \bar{h}_n^i(\boldsymbol{\phi}X_n) \Delta^{\boldsymbol{\phi}}_{in}|\A_n^{k_n}]|\mathbb{G}\big]\big\|_{\subL1}
\\+&\; {\textstyle\sqrt{\frac{2}{\pi}}\sum_i} \Big\|{\textstyle\frac{\sqrt{|\A_n|}}{\eta(n)}}\mathbb{E}_{\mu_n^*}\big[\bar{h}_n^i(\boldsymbol{\phi}X_n) \Delta^{\boldsymbol{\phi}}_{in}- \mathbb{E}[\bar{h}_n^i(\boldsymbol{\phi}X_n) \Delta^{\boldsymbol{\phi}}_{in}|\mathbb{G}]\big|\A_n^{k_n}\big]\Big\|_{\subL1}\;,
  \end{align*}
where $(*)$ uses the fact that $\sup_{x\in \mathbb{R}}|t'(x)|\le \sqrt{2/\pi}$.
\end{proof}

\subsection{Bounding the first term in Lemma \ref{neuchatel}}
\label{sec:proofs:individual:bounds}

We proceed to bound each term on the right-hand side of \cref{neuchatel}.
For the first term, we observe:\nolinebreak
\begin{lemma} Assume the conditions of \cref{theorem:BE:g}, and define
  a random measure ${\mu_n^{i-j}(\argdot):=|\A_n|P_{\mu_n\otimes\mu_n}(\boldsymbol{\phi}_j^{-1}\boldsymbol{\phi}'_i\in
    \argdot|\A_n^{2k_n})}$ on $\group$. Then
  \begin{equation*}
    \|\widehat{\mathbb{F}}_{\infty,i}(h_n,X_n,e)\|_{p} \le c_{i,p}(h_n)
    \quad\text{ and }\quad
    \mathbb{E}[\mu_n^{i-j}(\mathbb{I}_{\B_{b}})]\le   \mathcal{S}^n_w |\B_b|
  \end{equation*}
hold for ${i,n,b\in\mathbb{N}}$ and ${p\in\mathbb{R}}$.
\end{lemma}
\begin{proof}
  The first statement follows from the definition of $\widehat{\mathbb{F}}$, as
  \begin{equation*}
    \begin{split}&
      \|\widehat{\mathbb{F}}_{\infty,i }(h_n,X_n,e)\|=\big\|\!\lim_{m}\tfrac{1}{|\A_m|^{k_n}}\!
      \!\myint_{\A_m^{k_n}}\!\!h_n(\boldsymbol{\phi}_{1:i-1}e\boldsymbol{\phi}_{i+1:k_n}X_n)\!-\!h_n(\boldsymbol{\phi}X_n)|d\boldsymbol{\phi}|\big\|_{p}
      \\&\le \lim_{m}\frac{1}{|\A_m|^{k_n}} \int_{\A_m^{k_n}}\|h_n(\boldsymbol{\phi}_{1:i-1}e\boldsymbol{\phi}_{i+1:k_n}X_n)- h_n(\boldsymbol{\phi}X_n)\|_p|d\boldsymbol{\phi}|
      \;\le\; c_{i,p}(h_n)\;.
    \end{split}
  \end{equation*}
  Since
  ${\mathbb{E}[\mathbb{E}_{\mu_n^{i-j}}[\mathbb{I}_{\B_b}]]
  =
  |\A_n|\mathbb{E}[\mathbb{E}_{\mu_n\otimes\mu_n}[\mathbb{I}_{\boldsymbol{\phi}_j^{-1}\boldsymbol{\phi}'_i \in \B_b}|\A_n^{2k_n}]]
  \le   \mathcal{S}^n_w |\B_b|}$,
  the second statement also holds.
\end{proof}

\begin{lemma}\label{stein:term1}
  Assume hypothesis \eqref{H1:g}. Then
  \begin{equation*}
    \sup_{t\in \mathcal{F}}\Big\|\mathbb{E}\big[{\textstyle\frac{\sqrt{|\A_n|}}{\eta(n)}\sum_i}
    \mathbb{E}_{\mu_n^*}[\bar{h}_n^i(\boldsymbol{\phi}X_n) t(W^{\boldsymbol{\phi}}_{in})|\A_n^{k_n}]\big|\group\big]\Big\|_1
    \le K_1  C_{2}\bigl({\textstyle\frac{h_n}{\eta(n)}}\bigr)
    {\textstyle\sum_{k'_n<i}} c_{i,2}\bigl({\textstyle\frac{h_n}{\eta(n)}}\bigr),
  \end{equation*}
where  $K_1=O(|\B_K|  \mathcal{S}^n_w)$.
If hypothesis \eqref{H2:g} holds instead,
\begin{multline*}
        \sup_{t\in \mathcal{F}}\Big\|\mathbb{E}\big[{\textstyle\frac{\sqrt{|\A_n|}}{\eta(n)}\sum_i}
    \mathbb{E}_{\mu_n^*}[\bar{h}_n^i(\boldsymbol{\phi}X_n) t(W^{\boldsymbol{\phi}}_{in})|\A_n^{k_n}]\big|\group\big]\Big\|_1
    \\ \le  K_2 C_{2+\varepsilon}\bigl({\textstyle\frac{h_n}{\eta(n)}}\bigr)
    \big[{\textstyle\frac{k_n}{\sqrt{|\A_n|}}}
      +
      C_{2+\varepsilon}\bigl({\textstyle\frac{h_n}{\eta(n)}}\bigr)
      \big] \mathcal{R}_{n}(b_n)\\
      +K_3|\B_{b_n}|  C_{2}\bigl({\textstyle\frac{h_n}{\eta(n)}}\bigr)
    {\textstyle\sum_{k'_n<i}} c_{i,2}\bigl({\textstyle\frac{h_n}{\eta(n)}}\bigr)\;,
\end{multline*}
where $K_2=O(  \mathcal{S}^n_w)$ and $K_3=O(  \mathcal{S}^n_w )$.
\end{lemma}

\begin{proof}
  We prove the (harder) case of hypothesis \eqref{H2:g} first, and then modify it for 
  \eqref{H1:g}. Similar to $W^{\boldsymbol{\phi}}_{in}$, we abbreviate
  \begin{equation*}
    W^{\boldsymbol{\phi}}_{ibk}:={\textstyle\tfrac{\sqrt{|\A_n|}}{\eta(n)}}\mathbb{E}_{\mu^*_n}[\bar{h}_n^{\boldsymbol{\phi}_i,b,k}(\boldsymbol{\phi}'X_n)|\A_n^{k_n}]\;,
  \end{equation*}
  so that in particular ${W^{\boldsymbol{\phi}}_{in}=W^{\boldsymbol{\phi}}_{ib_nk_n'}}$.
  For all ${t\in\mathcal{F}}$,
  \begin{align}
    \label{fribourg}
    &\; \tsum_i\Big\|\mathbb{E}\big[\tfrac{\sqrt{|\A_n|}}{\eta(n)} \mathbb{E}_{\mu^*_n}\big[\bar{h}_n^i(\boldsymbol{\phi}X_n)  t(W^{\boldsymbol{\phi}}_{in})\big|\A_n^{k_n}\big]\big|\group\big]\Big\|_1
    \\
    \overset{(\ast)}{\le} &\; \tsum_i\mathbb{E}\Big[\tfrac{\sqrt{|\A_n|}}{\eta(n)} \mathbb{E}_{\mu^*_n}\big[|\bar{h}_n^i(\boldsymbol{\phi}X_n)||W^{\boldsymbol{\phi}}_{in}-W^{\boldsymbol{\phi}}_{ib_nk_n}| \big|\A_n^{k_n}\big]\Big]
    \notag\\ +& \;
    \tsum_i\Big\|\mathbb{E}\Big[\tfrac{\sqrt{|\A_n|}}{\eta(n)} \mathbb{E}_{\mu_n^*}\big[\bar{h}_n^i(\boldsymbol{\phi}X_n)  t(W^{\boldsymbol{\phi}}_{ib_nk_n})\big|\A_n^{k_n}\big]\big|\group\Big]\Big\|_1
    \notag
  \end{align}
  where ($\ast$) holds since $t$ is 1-Lipschitz.
  To bound the first term on the right-hand side, we use the
  definition the Lipschitz coefficients
  of $h_n$ to obtain
  \begin{align*}
    &\; \tsum_i\mathbb{E}\Big[\tfrac{\sqrt{|\A_n|}}{\eta(n)} \mathbb{E}_{\mu_n^*}\big[\big|\bar{h}_n^i(\boldsymbol{\phi}X_n)\big| \big|W^{\boldsymbol{\phi}}_{in}-W^{\boldsymbol{\phi}}_{ib_nk_n}\big| \big|\A_n^{k_n}\big]\Big]
    \\
    \le &\; \tsum_i\mathbb{E}\Big[|\A_n|\mathbb{E}_{\mu_n^{\otimes^2}}\big[\tsum_{j\in\mathcal{J}_n(\boldsymbol{\phi}_i,\boldsymbol{\phi}')}c_{i,2}\bigl(\tfrac{h_n}{\eta(n)}\bigr) c_{j,2}\bigl(\tfrac{h_n}{\eta(n)}\bigr)\big|\A_n^{2k_n}\big]\Big]
    \\
    \le &\; |\B_{b_n}| \tsum_i\tsum_{k'_n<j\le k_n} c_{i,2}\bigl(\tfrac{h_n}{\eta(n)}\bigr) c_{j,2}\bigl(\tfrac{h_n}{\eta(n)}\bigr)   \mathcal{S}^n_w
    \;,
  \end{align*}
  where ${\mathcal{J}_n(\boldsymbol{\phi}_i,\boldsymbol{\phi}')= \mathcal{I}_{b_n,k_n}(\boldsymbol{\phi}_i,\boldsymbol{\phi}')\setminus \mathcal{I}_{b_n,k'_n}(\boldsymbol{\phi}_i,\boldsymbol{\phi}')}$.

  To bound the second term, consider the vector ${\boldsymbol{\phi}\in\group^{k_n}}$ in \eqref{fribourg}.
  We define a sequence $(\boldsymbol{\phi}^{i,j})_{j\in \mathbb{N}}$ in ${\group^{k_n}}$ whose coordinates
  differ increasingly from the $i$th coordinate of $\boldsymbol{\phi}$ as $j$ increases: Set
  ${\boldsymbol{\phi}^{i,0}=\boldsymbol{\phi}}$. For ${j\geq 1}$, choose
  \begin{equation*}
    \boldsymbol{\phi}^{i,j}_k:=\;
    \begin{cases}
      \boldsymbol{\phi}^{i,j-1}_k & \text{if}\;
      d(\boldsymbol{\phi}_k,\boldsymbol{\phi}_i)\not\in [j, j+1)
        \\ \text{any }\boldsymbol{\phi}^{i,j}_k\text{ with }d(\boldsymbol{\phi}^{i,j}_k,\boldsymbol{\phi}_i)>\text{diam}(\A_n) &  \text{if}\; d(\boldsymbol{\phi}_k,\boldsymbol{\phi}_i)\in [j, j+1)
    \end{cases}
  \end{equation*}
  for each ${k\leq k_n}$. By definition of $\mu_n^*$,
  we have $\boldsymbol{\phi}^{i,j}=\boldsymbol{\phi}$ for $j\le b_n$. Then
  \begin{align*}
    &\; \tsum_i\Big\|\mathbb{E}\bigl[\tfrac{\sqrt{|\A_n|}}{\eta(n)} \bar{h}_n^i(\boldsymbol{\phi}X_n)  t(W^{\boldsymbol{\phi}}_{b_n,k_n,i})\big|\group\bigr]\Big\|_1
    \\\le &\; \tsum_i\tsum_{j\ge b_n}\Big\|\mathbb{E}\bigl[\tfrac{\sqrt{|\A_n|}}{\eta(n)} \bar{h}_n^i(\boldsymbol{\phi}^{i,j+1}X_n)  [t(W^{\boldsymbol{\phi}}_{ijk_n})-t(W^{\boldsymbol{\phi}}_{i(j+1)k_n})]\big|\group\bigr]\Big\|_1
    \\ + &\; \tsum_i\tsum_{j\ge b_n}\Big\|\mathbb{E}\bigl[\tfrac{\sqrt{|\A_n|}}{\eta(n)}[\bar{h}_n^i(\boldsymbol{\phi}^{i,j+1}X_n) -\bar{h}_n^i(\boldsymbol{\phi}^{i,j}X_n)]t(W^{\boldsymbol{\phi}}_{ijk_n})\big|\group\bigr]\Big\|_1
    \\\overset{(*)} {\le} &\; 4 \sqrt{\tfrac{2}{\pi}}\tsum_{j,~j\ge b_n}\tsum_i c_{i,2+\varepsilon}\bigl(\tfrac{h_n}{\eta(n)}\bigr)  |\A_n| \big\|W^{\boldsymbol{\phi}}_{ijk_n}-W^{\boldsymbol{\phi}}_{i(j+1)k_n}\big\|_{2+\varepsilon} \alpha_n^{\frac{\varepsilon}{2+\varepsilon}}(j|\mathbb{G})
\\+ &\; 4\tsum_i c_{i,2+\varepsilon}\bigl(\tfrac{h_n}{\eta(n)}\bigr)\tsum_{j\ge b_n} \alpha_n^{\frac{\varepsilon}{2+\varepsilon}}(j|\mathbb{G}) \sqrt{|\A_n|}\mathbb{I}\braces{d(\boldsymbol{\phi}_{i}, \boldsymbol{\phi}_{\setminus i})\in [j,{j+1}]}\;.
  \end{align*}
  Here, $(*)$ is obtained using \cref{lemma:mixing}, and the fact that
  \begin{equation*}
    \sup_{x\in \mathbb{R}}|t'(x)|\le \sqrt{2/\pi}
    \quad\text{ and }\quad
    \sup_{x\in \mathbb{R}}|t(x)|\le 1\;
  \end{equation*}
  Since that is true for any ${\boldsymbol{\phi}\in\mathbb{G}^{k_n}}$,
  using the definition of $  \mathcal{S}^n_w $ we conclude
\begin{align*}
  &\tsum_i\Big\|\mathbb{E}\bigl[
    \tfrac{\sqrt{|\A_n|}}{\eta(n)} \mathbb{E}_{\mu_n^*}\big(\bar{h}_n^i(\boldsymbol{\phi}X_n)
    t(W^{\boldsymbol{\phi}}_{b_n,k_n,i})\big|\A_n^{k_n}\big)\big|\group\bigr]\Big\|_1
  \\&
  \le
  4 \sqrt{\tfrac{2}{\pi}}\tsum_i c_{i,2+\varepsilon}(\tfrac{h_n}{\eta(n)})\tsum_j c_{j,2+\varepsilon}(\tfrac{h_n}{\eta(n)})
  \\&\quad \times
\mathbb{E}\bigl[\mathbb{E}_{\mu_n^{\otimes 2}}\big( \mathbb{I}\braces{j\not\in \mathcal{I}_{b_n,k_n}({\boldsymbol{\phi}_i,\boldsymbol{\phi}'})}|\A_n|\alpha^{\tfrac{\varepsilon}{2+\varepsilon}}(d(\boldsymbol{\phi}_i,\boldsymbol{\phi}'_j)|\mathbb{G})|\A_n^{2k_n}\big)\bigr]
\\& + 4\tsum_i c_{i,2+\varepsilon}(\tfrac{h_n}{\eta(n)})\tsum_{j\ne i}\mathbb{E}\bigl[\mathbb{E}_{\mu_n^*}( \sqrt{|\A_n|}\alpha_n^{\frac{\varepsilon}{2+\varepsilon}}(d(\boldsymbol{\phi}_i,\boldsymbol{\phi}_j)|\mathbb{G})|\A_n^{k_n}\big)\bigr]
\\& \le4 \tsum_i c_{i,2+\varepsilon}(\tfrac{h_n}{\eta(n)})  \mathcal{S}^n_w
\Bigl(\tfrac{k_n}{\sqrt{|\A_n|}}
+\sqrt{\tfrac{2}{\pi}}\tsum_i c_{i,2+\varepsilon}(\tfrac{h_n}{\eta(n)})\Bigr)\\&\quad\tsum_{i\ge b_n} \alpha_n^{\frac{\varepsilon}{2+\varepsilon}}(i|\mathbb{G}) | \B_{i+1}\!\setminus\!\B_{i}|\;.
\end{align*}
That establishes the result under \eqref{H2:g}. If \eqref{H1:g} is assumed instead, the second term of \cref{fribourg} vanishes
by \cref{lemma:mixing}. We hence have
\begin{equation*}\begin{split}&
  \sup_{t\in\mathcal{F}}\Big\|\mathbb{E}\bigl[
    \tfrac{\sqrt{|\A_n|}}{\eta(n)}\tsum_i
    \mathbb{E}_{\mu_n^*}\bigl[\bar{h}_n^i(\boldsymbol{\phi}X_n) t(W^{\boldsymbol{\phi}}_{in})|\A_n^{k_n}\big]\big|\group\bigr]\Big\|_1
  \\&
  \le \;  |\B_{K}|   \mathcal{S}^n_w  \tsum_i\tsum_{k'_n<j\le k_n} c_{i,2}\big(\tfrac{h_n}{\eta(n)}\big) c_{j,2}\big(\tfrac{h_n}{\eta(n)}\big)\;,
\end{split}\end{equation*}
which shows result also holds under \eqref{H1:g}.
\end{proof}

\subsection{The second term in Lemma \ref{neuchatel}}

The strategy is to use a Taylor expansion, and to bound
\begin{equation*}
  \big| \bar{h}_n^i(\boldsymbol{\phi}X_n) \big(t(W^*)- t(W^{\boldsymbol{\phi}}_{in})-\Delta^{\boldsymbol{\phi}}_{in}t'(W^*)\big)\big|
  \le
  \frac{\sup_{x\in \mathbb{R}}|t''(x)|}{2} \big| \bar{h}_n^i(\boldsymbol{\phi}X_n) \big|
  (\Delta^{\boldsymbol{\phi}}_{in})^2
\end{equation*}
As $\bar{h}^i_n(X_n)$ might not admit a third moment we first upper-bound it using the sequence $(\gamma_n)$.
To bound  $| \bar{h}_n^i(\boldsymbol{\phi}X_n) \mathbb{I}(\bar{h}_n^i(\boldsymbol{\phi}X_n)\le \gamma_n c_{i,2}(h_n))|({\Delta^{\boldsymbol{\phi}}_{in}})^2$, we must control the  probability that random triples
${\boldsymbol{\phi},\boldsymbol{\phi}',\boldsymbol{\phi}''\in\group^{k_n}}$ satisfy
\begin{equation}
  \label{eq:interaction:condition:1}
  d(\boldsymbol{\phi}_i,\boldsymbol{\phi}'_j),
  d(\boldsymbol{\phi}_i,\boldsymbol{\phi}''_l)\le b_n
  \quad
  \text{ and }
  \quad
  \boldsymbol{\phi}'\in V_{i,\beta_n}(n)
\end{equation}
for some ${i\leq k_n}$ and $j,l\le k'_n$, and either
\begin{equation}
  \label{eq:interaction:condition:2}
  \text{(i)}\; \min_{l\leq k_n}d(\boldsymbol{\phi}'_j,\boldsymbol{\phi}''_l)\in [k,{k+1}]
  \quad\text{ or }\quad
  \text{(ii)}\; \min_{\substack{l\leq k_n\\l\ne i}}d(\boldsymbol{\phi}'_j,\boldsymbol{\phi}_l)\in [k,{k+1}]\;.
\end{equation}
We quantify these as follows:
The upper bound on the term in \cref{neuchatel} must be established for fixed values of
$n$ and $\beta_n$. Given such values, we choose a constant
$S^*_2(k_n)$ that satisfies
\begin{equation*}
  \begin{split}
  \frac{|\A_n|^2\,
  \big\|\mathbb{E}_{\mu_n^{\otimes
      3}}\big[\mathbb{I}\braces{\boldsymbol{\phi},\boldsymbol{\phi}'',\boldsymbol{\phi}''\text{
      satisfies }\eqref{eq:interaction:condition:1}\text{ and } (\ref{eq:interaction:condition:2}i)}\big\vert\A_n^{3k_n}\big]\big\|_1
}{|\B_{{k+1}}\setminus \B_{k}||\B_{b_n}|k_n}
  & \le S^*_2(k_n)\\
  \text{ and }\quad
  \frac{|\A_n|^2\,
  \big\|\mathbb{E}_{\mu_n^{\otimes 3}}\big[\mathbb{I}\braces{\boldsymbol{\phi},\boldsymbol{\phi}'',\boldsymbol{\phi}''\text{
      satisfies }\eqref{eq:interaction:condition:1}\text{ and } (\ref{eq:interaction:condition:2}ii)}\big\vert\A_n^{3k_n}\big]\big\|_1
}{|\B_{{k+1}}\setminus \B_{k}||\B_{b_n}|k_n}
  & \le S^*_2(k_n)\;.
  \end{split}
\end{equation*}
Similarly, we choose a constant $S^{\ast}_0$ such that
\begin{equation*}
  \frac{|\A_n|}{|\B_{m}|}\big\|\mathbb{E}_{\mu_n^{\otimes 2}}[\mathbb{I}\braces{d(\boldsymbol{\phi}_i,\boldsymbol{\phi'}_j)\le m\text{ and }\phi'\not \in V_{i,\beta_n}(n)}|\A_n^{2k_n}] \big\|_1\le S^{\ast}_0
\end{equation*}
for all ${n,m\in\mathbb{N}}$ and ${i,j\le k_n}$.
\begin{lemma}
  \label{stein:term2}
  Assume \eqref{H1:g} holds. Then for ${t\in\mathcal{F}}$, and any $p,q>0$ satisfying
  $\frac{1}{p}+\frac{1}{q}=1$,
  \begin{align*}
    & \sup_{H\in \mathcal{F}} \Big\|\mathbb{E}\Big(\tfrac{\sqrt{|\A_n|}}{\eta(n)}\mathbb{E}_{\mu_n^*}\big(\bar{h}_n^i(\boldsymbol{\phi}X_n) (t(W^*)- t(W^{\boldsymbol{\phi}}_{in})-\Delta^{\boldsymbol{\phi}}_{in} t'(W^*))|\A_n^{k_n}\big)\Big|\group\Big)\Big\|_1
      \\&\le K_1  \tfrac{k_nk'_n}{\sqrt{|\A_n|}}C_{2q}\big(\tfrac{h_n}{\eta(n)}\big)^2~
        S^*_2(k'_n)\tsum_i \Gamma_{i,p}(\gamma_n)
        \\& +K_2S^{\ast}_0 C_{2}(\tfrac{h_n}{\eta(n)})^2
        +K_3C_2(\tfrac{h_n}{\eta(n)})
        \tsum_i c_{i,2}\big(\tfrac{\bar{h}_n^{i}(\boldsymbol{\phi}X_n)}{\eta(n)}\mathbb{I}
        \braces{\tfrac{\bar{h}_n^{i}(\boldsymbol{\phi}X_n)}{\eta(n)}\ge \gamma_n}\big)\;,
  \end{align*}
        where $K_1=O(|\B_k|^2)$ and $K_2=O(|\B_{K}|)$ and $K_3=O(   \mathcal{S}^n_w |\B_{K}|)$.
        If \eqref{H2:g} holds instead,
        \begin{align*}
          &\sup_{H\in \mathcal{F}} \Big\|\mathbb{E}\Big(\tfrac{\sqrt{|\A_n|}}{\eta(n)}\mathbb{E}_{\mu_n^*}\big(\bar{h}_n^i(\boldsymbol{\phi}X_n) (t(W^*)- t(W^{\boldsymbol{\phi}}_{in})-\Delta^{\boldsymbol{\phi}}_{in} t'(W^*))|\A_n^{k_n}\big)\Big|\group\Big)\Big\|_1
            \\&
            \le
            K_1 \frac{ k_nk'_n|\B_{b_n}|S^*_2(k'_n)}{\sqrt{|\A_n|}}C_{(2+\varepsilon)q}\big(\tfrac{h_n}{\eta(n)}\big)^2~
            \sum_i \Gamma_{i,p(1+\frac{\varepsilon}{2})}(\gamma_n)
            \\&
            + K_2{|\B_{b_n}|}S^{\ast}_0 C_{2}\big(\tfrac{h_n}{\eta(n)}\big)^2
            + K_3|\B_{b_n}|C_2\big(\tfrac{h_n}{\eta(n)}\big)
            \tsum_i c_{i,2}\big(\tfrac{\bar{h}_n^{i}(\boldsymbol{\phi}X_n)}{\eta(n)}\mathbb{I}\braces{\tfrac{\bar{h}_n^{i}(\boldsymbol{\phi}X_n)}{\eta(n)}\!\ge\!\gamma_n}\big)
        \end{align*}
      where $K_1=O(\mathcal{R}_n(0))$ and $K_2=O(1)$ and $K_3=O(   \mathcal{S}^n_w )$.
\end{lemma}
\begin{proof}
  Suppose first \eqref{H1:g} holds. Since $h_n(X_n)$ may not have a third moment, we upper-bound it using the sequence $(\gamma_n)$.
  By the triangle inequality,
  \begin{align*}
    \label{3rd_moment_break}
    \big\|\mathbb{E}\big[\tfrac{\sqrt{|\A_n|}}{\eta(n)}&\tsum_i\mathbb{E}_{\mu_n^*}\big[
        \underbrace{\bar{h}_n^i(\boldsymbol{\phi}X_n)  (t(W^*)- t(W^{\boldsymbol{\phi}}_{in})-\Delta^{\boldsymbol{\phi}}_{in} t'(W^*))}_{:=T}
        |\A_n^{k_n}\big]\big|\group\big]\big\|_1\\
    \le &\;
    \big\|\mathbb{E}\big[\tfrac{\sqrt{|\A_n|}}{\eta(n)}
      \tsum_i\mathbb{E}_{\mu_n^*}\big[
        T\,\mathbb{I}\braces{|\tfrac{\bar{h}_n^i(\boldsymbol{\phi}X_n)}{c_{i,2}(h_n)}|> \gamma_n}
        \,|\,\A_n^{k_n}\big]\big|\group\big]\big\|_1
    \\+ &\;
    \big\|\mathbb{E}\big[\tfrac{\sqrt{|\A_n|}}{\eta(n)}
      \tsum_i\mathbb{E}_{\mu_n^*}\big[
        T\,\mathbb{I}\braces{|\tfrac{\bar{h}_n^i(\boldsymbol{\phi}X_n)}{c_{i,2}(h_n)}|\le \gamma_n}
        \,|\,\A_n^{k_n}\big]\big|\group\big]\big\|_1\;.
  \end{align*}
We again bound each term separately.
Since ${t\in\mathcal{F}}$, it satisfies \eqref{eq:stein:taylor}, hence
\begin{align*}
     \big\|\mathbb{E}\big[&\tfrac{\sqrt{|\A_n|}}{\eta(n)}
      \tsum_i\mathbb{E}_{\mu_n^*}\big[
        T\,\mathbb{I}\braces{|\tfrac{\bar{h}_n^i(\boldsymbol{\phi}X_n)}{c_{i,2}(h_n)}|> \gamma_n}
        \,|\,\A_n^{k_n}\big]\big|\group\big]\big\|_1\\
    \le &\;
    2 \tsum_i\mathbb{E}\big[\tfrac{\sqrt{|\A_n|}}{\eta(n)}
      \mathbb{E}_{\mu_n^*}\big[
        |\bar{h}_n^i(\boldsymbol{\phi}X_n)|\mathbb{I}\braces{
          \tfrac{|\bar{h}_n^i(\boldsymbol{\phi}X_n)|}{c_{i,2}(h_n)}> \gamma_n
        }|\Delta^{\boldsymbol{\phi}}_{in}||\A_n^{k_n}
        \big]\big];.
\end{align*}
For all ${\boldsymbol{\phi}\in \group^{k_n}}$ and ${i\in \mathbb{N}}$ we have, by definition of $\Delta^{\boldsymbol{\phi}}_{in}$,
\begin{align*}&
\tsum_i\mathbb{E}\big[\tfrac{\sqrt{|\A_n|}}{\eta(n)}
      \mathbb{E}_{\mu_n^*}\big[
        |\bar{h}_n^i(\boldsymbol{\phi}X_n)|\mathbb{I}\braces{
          \tfrac{|\bar{h}_n^i(\boldsymbol{\phi}X_n)|}{c_{i,2}(h_n)}> \gamma_n
        }|\Delta^{\boldsymbol{\phi}}_{in}||\A_n^{k_n}
        \big]\big]
    \\
    &\le\;
    |\A_n|\Bigl(
    c_{j,2}\big(\tfrac{\bar{h}_n^i(\boldsymbol{\phi}X_n)}{\eta(n)}
    \mathbb{I}\braces{\tfrac{|\bar{h}_n^i(\boldsymbol{\phi}X_n)|}{c_{i,2}(h_n)}> \gamma_n}
    \big)
    \mathbb{E}\big[\mathbb{E}_{\mu_n^{*}}[\mathbb{I}\braces{d(\boldsymbol{\phi}_i,\boldsymbol{\phi}'_{j})\le b_n}|\A_n^{k_n},\boldsymbol{\phi}]\big]\Bigl)\\&\quad \tsum_{j\le k'_n} c_{i,2}\big(\tfrac{\bar{h}_n^i}{\eta(n)}\big)\;.
\end{align*}
Using the definition of $  \mathcal{S}^n_w $, this implies
\begin{align*}
     \big\|\mathbb{E}\big[&\tfrac{\sqrt{|\A_n|}}{\eta(n)}
      \tsum_i\mathbb{E}_{\mu_n^*}\big[
        T\,\mathbb{I}\braces{|\tfrac{\bar{h}_n^i(\boldsymbol{\phi}X_n)}{c_{i,2}(h_n)}|> \gamma_n}
        \,|\,\A_n^{k_n}\big]\big|\group\big]\big\|_1\\
    \le &\;
      \mathcal{S}^n_w |\B_{b_n}|C_2\big(\tfrac{h_n}{\eta(n)}\big)
    \tsum_i c_{i,2}\big(\tfrac{\bar{h}_n^i(\boldsymbol{\phi}X_n)}{\eta(n)}
    \mathbb{I}\braces{\tfrac{|\bar{h}_n^i(\boldsymbol{\phi}X_n)|}{c_{i,2}(h_n)}> \gamma_n}
    \big)\;.
\end{align*}
To bound the second term, we abbreviate
\begin{align*}
  & \tilde{W}^{\boldsymbol{\phi}}_{in}:=
  \mathbb{E}_{\mu_n^*}\big[\mathbb{I}\braces{\boldsymbol{\phi}'\in V_i(\beta_n)}\bar{h}_n^{\boldsymbol{\phi}_i,b_n,k'_n}(\boldsymbol{\phi}'X_n)\big|\A_n^{k_n},\boldsymbol{\phi}\big]\\
  \text{ and }\quad &
  \tilde{\Delta}^{\boldsymbol{\phi}}_{in}:= \mathbb{E}_{\mu_n^*}\big[\mathbb{I}\braces{\boldsymbol{\phi}'\in V_i(\beta_n)}
    \bigl(h_n(\boldsymbol{\phi}'X_n)-\bar{h}_n^{\boldsymbol{\phi}_i,b_n,k'_n}(\boldsymbol{\phi}'X_n)\bigr)\big|\A_n^{k_n},\boldsymbol{\phi}\big]\;.
\end{align*}
Again using the triangle inequality, we have
\begin{align*}
&\big\|\mathbb{E}\big[\tfrac{\sqrt{|\A_n|}}{\eta(n)}
      \tsum_i\mathbb{E}_{\mu_n^*}\big[
        T\,\mathbb{I}\braces{|\tfrac{\bar{h}_n^i(\boldsymbol{\phi}X_n)}{c_{i,2}(h_n)}|\le \gamma_n}
        \,|\,\A_n^{k_n}\big]\big|\group\big]\big\|_1\\ \le
&
\big\|\mathbb{E}\big[\tfrac{\sqrt{|\A_n|}}{\eta(n)}\!\tsum_i\!\mathbb{E}_{\mu_n^*}\!\big[\bar{h}_n^i(\boldsymbol{\phi}X_n)
\mathbb{I}\braces{|\tfrac{\bar{h}_n^i(\boldsymbol{\phi}X_n)}{c_{i,2}(h_n)}|\!\!\le\!\!\gamma_n\!}
\big(t(\tilde{W}^{\boldsymbol{\phi}}_{in})-t(W^{\boldsymbol{\phi}}_{in})\big)\big|\A_n^{k_n}\big]\big|\group\big]\big\|_1
\\+
&
\big\|\mathbb{E}\big[\tfrac{\sqrt{|\A_n|}}{\eta(n)}\!\tsum_i\!\mathbb{E}_{\mu_n^*}\!\big[\bar{h}_n^i(\boldsymbol{\phi}X_n)
\mathbb{I}\braces{|\tfrac{\bar{h}_n^i(\boldsymbol{\phi}X_n)}{c_{i,2}(h_n)}|\!\!\le\!\!\gamma_n\!}
(\Delta^{\boldsymbol{\phi}}_{in}-\tilde{\Delta}^{\boldsymbol{\phi}}_{in}) t'(W^{*})|\A_n^{k_n}\big]\big|\group\big]\big\|_1
\\+
&
\big\|\mathbb{E}\big[\tfrac{\sqrt{|\A_n|}}{\eta(n)}\!\tsum_i\!\mathbb{E}_{\mu_n^*}\!\big[\bar{h}_n^i(\boldsymbol{\phi}X_n)
\\&\hspace{6.8em}\mathbb{I}\braces{|\tfrac{\bar{h}_n^i(\boldsymbol{\phi}X_n)}{c_{i,2}(h_n)}|\!\!\le\!\!\gamma_n\!}
(t(\tilde{W}^{\boldsymbol{\phi}}_{in})\!-\! t(\tilde{W}^{*})\!-\!\tilde{\Delta}^{\boldsymbol{\phi}}_{in} t'\!({W}^{*}))\big|\A_n^{k_n}\big]\big|\group\big]\big\|_1
\\
&=: \text{(a)} + \text{(b)} +\text{(c)}\;,
\end{align*}
and must bound (a)---(c) further. Since $t$ is 1-Lipschitz,
\begin{align*}
& \text{(a)}\le \mathbb{E}\big[\tfrac{\sqrt{|\A_n|}}{\eta(n)}\!\tsum_i\!\mathbb{E}_{\mu_n^*}\!\big[\big|\bar{h}_n^i(\boldsymbol{\phi}X_n)\big|
\mathbb{I}\braces{|\tfrac{\bar{h}_n^i(\boldsymbol{\phi}X_n)}{c_{i,2}(h_n)}|\!\!\le\!\!\gamma_n\!}
\big|\tilde{W}^{\boldsymbol{\phi}}_{in}-W^{\boldsymbol{\phi}}_{in}\big|\big|\A_n^{k_n}\big]\big]
\\&\le 2\tsum_i c_{i,2}\big(\tfrac{h_n}{\eta(n)}\big) \tsum_{j\le k'_n}c_{j,2}\big(\tfrac{h_n}{\eta(n)}\big)
  \\
  &\quad\quad\quad\; \mathbb{E}[|\A_n|\mathbb{E}_{\mu_n^{\otimes 2}}[\mathbb{I}\braces{d(\boldsymbol{\phi}_i,\boldsymbol{\phi'}_j)\!\le\! b_n,\boldsymbol{\phi'}\! \not\in V_{i\beta_n}}|\A_n^{2k_n}]]
\\& \le 2C_{2}\big(\tfrac{h_n}{\eta(n)}\big) ^2\sup_{i,j} \mathbb{E}[|\A_n|\mathbb{E}_{\mu_n^{\otimes 2}}\big[\mathbb{I}\braces{d(\boldsymbol{\phi}_i,\boldsymbol{\phi'}_j)\!\le\! b_n,\boldsymbol{\phi'}\! \not\in V_{i\beta_n}}|\A_n^{2k_n}]\big] .
\end{align*}
Analogously, we have
\begin{align*}
\text{(b)}\le 2|\B_{b_n}|\tsum_i& c_{i,2}\big(\tfrac{h_n}{\eta(n)}\big)\tsum_{j\le k_n} c_{j,2}\big(\tfrac{h_n}{\eta(n)}\big)\\
&\sup_{i,j}\tfrac{1}{|\B_{b_n}|}
\mathbb{E}[|\A_n|\mathbb{E}_{\mu_n^{\otimes 2}}[\mathbb{I}\braces{d(\boldsymbol{\phi}_i,\boldsymbol{\phi'}_j)\!\le\! b_n,\phi'\! \not\in V_{i\beta_n}}|\A_n^{2k_n}]]\;.
\end{align*}
To bound (c), we first observe
\begin{align*}
  \text{(c)}
 \le&\; \frac{1}{2}\sup_{x\in \mathbb{R}}|h''(x)|\mathbb{E}\big[\tfrac{\sqrt{|\A_n|}}{\eta(n)}\tsum_i \mathbb{E}_{\mu_n^*}\big[ |\bar{h}_n^i(\boldsymbol{\phi}X_n)|\mathbb{I}\braces{|\tfrac{\bar{h}_n^i(\boldsymbol{\phi}X_n)}{c_{i,2}(h_n)}|\!\!\le\!\!\gamma_n\!}(\tilde{\Delta}^{\boldsymbol{\phi}}_{in})^2\,\big|\,\A_n^{k_n}\big] \big]\;.
\end{align*}
We again have to control interactions between elements of $\group^{k_n}$. In addition to the element
$\boldsymbol{\phi}$ in (c), fix two further elements $\boldsymbol{\phi}'$ and $\boldsymbol{\phi}''   $, and
a list ${\bpsi^0,\ldots,\bpsi^{b_n}}$ constructed for ${b=0,\ldots,b_n}$ as follows:
\begin{itemize}
\item Set $\bpsi^{0}=\boldsymbol{\phi}''$.\\[-.8em]
\item Choose ${\bpsi^b_k:=\bpsi_k^{b-1}}$ if either
\begin{equation*}
  \min\braces{\bar{d}(\bpsi^{b-1}_k,\bphi),\bar{d}(\bpsi^{b-1}_k,\bphi')}\not\in[b,b+1)
\quad\text{ or }\quad
k\not\in\mathcal{I}_{b_n,k_n'}(\bphi,\bphi'')\;.
\end{equation*}
\item Otherwise, choose $\bpsi^b_k$ such that ${\bar{d}(\bpsi^b_k,\bphi)>b_n}$ and
  ${\bar{d}(\bpsi^b_k,\bphi')>b_n}$.
\end{itemize}
Note such a sequence always exists.
Abbreviate $${G(\boldsymbol{\phi}'):=\big[h_n(\boldsymbol{\phi}'X_n)-\bar{h}_n^{\boldsymbol{\phi}'_{i},b_n,k'_n}(\boldsymbol{\phi}'X_n)}\big]\mathbb{I}\braces{\boldsymbol{\phi}'\in V_i(\beta_n)}.$$
An application of the triangle inequality and of \cref{lemma:mixing}
yields
\begin{align*}
  &\big\|\mathbb{E}[
    \tfrac{|\bar{h}_n^i(\boldsymbol{\phi}X_n)|}{\eta(n)^{3}}
    \mathbb{I}\braces{\tfrac{|\bar{h}_n^{i}(\boldsymbol{\phi}X_n)|}{c_{i,2}(h_n)}\!\le\!\gamma_n}
    G(\boldsymbol{\phi}')G(\boldsymbol{\phi}'')
    |\group]\big\|_1
\\&
\le\tsum_l\Big\|
\mathbb{E}[
    \tfrac{|\bar{h}_n^i(\boldsymbol{\phi}X_n)|}{\eta(n)^{3}}
    \mathbb{I}\braces{\tfrac{|\bar{h}_n^{i}(\boldsymbol{\phi}X_n)|}{c_{i,2}(h_n)}\!\le\!\gamma_n}
    G(\boldsymbol{\phi}')G(\boldsymbol{\psi}^l)
  \big|\group  ]\\
&\quad\quad \quad-
\mathbb{E}[
    \tfrac{|\bar{h}_n^i(\boldsymbol{\phi}X_n)|}{\eta(n)^{3}}
    \mathbb{I}\braces{\tfrac{|\bar{h}_n^{i}(\boldsymbol{\phi}X_n)|}{c_{i,2}(h_n)}\!\le\!\gamma_n}
    G(\boldsymbol{\phi}')G(\boldsymbol{\psi}^{l-1})
  \big|\group  ]
\Big\|_1
 \\&
 \le 4 \Gamma_{i,q(1+\frac{\varepsilon}{2})}(\gamma_n)
 \tsum_{j,k}
 c_{k,2p(1+\frac{\varepsilon}{2})}\big(\tfrac{h_n}{\eta(n)}\big)
 c_{j,2p(1+\frac{\varepsilon}{2})}\big(\tfrac{h_n}{\eta(n)}\big)
 \\&
 \quad\;\mathbb{I}\braces{\boldsymbol{\phi}''\in V_i(\beta_n)}\alpha_n^{\frac{\varepsilon}{2+\varepsilon}}(\min\braces{\bar{d}(\bphi''_{k},\bphi),\bar{d}(\bphi''_{k},\bphi')}|\mathbb{G})\;,
\end{align*}
where the final term sums over
${j\in\mathcal{I}_{b_n,k'_n}(\bphi,\bphi')}$ and ${l\in\mathcal{I}_{b_n,k'_n}(\bphi,\bphi'')}$.
By Taylor expansion, we hence obtain
\begin{align*}
  \text{(c)}
 \le&\; \frac{1}{2}\sup_{x\in \mathbb{R}}|h''(x)|\mathbb{E}\big[\tfrac{\sqrt{|\A_n|}}{\eta(n)}\tsum_i \mathbb{E}_{\mu_n^*}\big[ |\bar{h}_n^i(\boldsymbol{\phi}X_n)|\mathbb{I}\braces{|\tfrac{\bar{h}_n^i(\boldsymbol{\phi}X_n)}{c_{i,2}(h_n)}|\!\!\le\!\!\gamma_n\!}(\tilde{\Delta}^{\boldsymbol{\phi}}_{in})^2\,\big|\,\A_n^{k_n}\big] \big]
\\\le &\;  4 \tsum_{i\le k_n,j,k\le k'_n}\Gamma_{i,q(1+\frac{\varepsilon}{2})}(\gamma_n)
  c_{k,2p(1+\frac{\varepsilon}{2})}\big(\tfrac{h_n}{\eta(n)}\big)
 c_{j,2p(1+\frac{\varepsilon}{2})}\big(\tfrac{h_n}{\eta(n)}\big)\tsum_b \alpha_n^{\frac{\varepsilon}{2+\varepsilon}}(b|\group)
 \\&
 \; \big\|\mathbb{E}_{\mu_n^{\otimes 3}}\big[\bar{d}(\bphi''_k,\bphi)\in [b,b+1], ~\boldsymbol{\phi}'',\boldsymbol{\phi}'\in \B_{b_n}(\boldsymbol{\phi}), ~\mathbb{I}\braces{\boldsymbol{\phi}''\in V_i(\beta_n)}\big|A_{n}^{3k_n}\big]\big\|_1
\\+ &\;  4 \tsum_{i\le k_n,j,k\le k'_n}\Gamma_{i,q(1+\frac{\varepsilon}{2})}(\gamma_n)
  c_{k,2p(1+\frac{\varepsilon}{2})}\big(\tfrac{h_n}{\eta(n)}\big)
 c_{j,2p(1+\frac{\varepsilon}{2})}\big(\tfrac{h_n}{\eta(n)}\big)\tsum_b \alpha_n^{\frac{\varepsilon}{2+\varepsilon}}(b|\group)
 \\&
 \;\big\|\mathbb{E}_{\mu_n^{\otimes 3}}\big[\bar{d}(\bphi''_{k},\bphi')\in [b,b+1], \boldsymbol{\phi}'',\boldsymbol{\phi}'\in \B_{b_n}(\boldsymbol{\phi}), \mathbb{I}\braces{\boldsymbol{\phi}''\in V_i(\beta_n)}\big|A_{n}^{3k_n}\big]\big\|_1
     \\ \overset{}{\le} &\;
  \tfrac{1}{\sqrt{|\A_n|}}
  \Bigl(8{k'_n}k_n|\B_{b_n}| \bigl(
    \tsum_i c_{i,q(1+\frac{\varepsilon}{2})}(\tfrac{h_n}{\eta(n)})\bigr)^2
    \bigl(\tsum_i\Gamma_{i,p(1+\frac{\varepsilon}{2})}(\gamma_n)\bigr)
    S^*_2(k'_n)\mathcal{R}_n(0)\Bigr)\;.
\end{align*}
This establishes the result under hypothesis \eqref{H2:g}. If \eqref{H1:g} holds instead, we modify
the proof above as follows: There is now some ${K\in\mathbb{N}}$
such that ${b_n=K}$ for all $n$, and
that any two elements separated by a distance of at least $K$ are conditionally independent.
In this case,
\begin{align*}
  \big|\mathbb{E}&\big[\tfrac{\sqrt{|\A_n|}}{\eta(n)}\mathbb{E}_{\mu_n^*}\big[
      \bar{h}_n^i(\boldsymbol{\phi}X_n) (t(W^*)- t(W^{\boldsymbol{\phi}}_{in})-\Delta^{\boldsymbol{\phi}}_{in} t'(W^*))|\A_n^{k_n}
      \big]\big]\big|
\\&\le\;
 \frac{4 {k'_n}k_n|\B_{K}|^2}{\sqrt{|\A_n|}}
 \bigl(\tsum_i c_{i,2q}\big(\tfrac{h_n}{\eta(n)}\big)\bigr)^2
 \bigl(\tsum_i\Gamma_{i,{p(1+\frac{\varepsilon}{2})}}(\gamma_n)\bigr) S^*_2(k'_n)
\\&
+2\tsum_i c_{i,2}\big(\tfrac{h_n}{\eta(n)}\big)
\tsum_j c_{j,2}\big(\tfrac{h_n}{\eta(n)}\big)
\\&\hspace{9.5em}
\mathbb{E}[|\A_n|\mathbb{E}_{\mu_n^{\otimes 2}}[\mathbb{I}\braces{\bar{d}(\boldsymbol{\phi}_i,\boldsymbol{\phi}'_{1:j})\!\le\! K,\boldsymbol{\phi}\! \not\in V_{i\beta_n}}|\A_n^{2k_n}]]
\\& +2   \mathcal{S}^n_w |\B_{K}|k'_n \tsum_i c_{i,2}\big(\mfrac{\bar{h}_n^i(\boldsymbol{\phi}X_n)}{\eta(n)}
\mathbb{I}\braces{|\mfrac{\bar{h}_n^i(\boldsymbol{\phi}X_n)}{c_{i,2}(h_n)}|\!>\!\gamma_n\!}
\big)M_2\big(\tfrac{h_n}{\eta(n)}\big)\;,
\end{align*}
and the result holds under \eqref{H1:g}.
\end{proof}

\subsection{The third term in Lemma \ref{neuchatel}}
\begin{lemma}\label{stein:term3}
  Fix ${p,q>0}$ such that ${\frac{1}{p}+\frac{1}{q}=1}$. If \eqref{H1:g} holds,
  \begin{equation*}
    \begin{split}&
      \big\|
      1-\tfrac{\sqrt{|\A_n|}}{\eta(n)}\mathbb{E}\big[\mathbb{E}_{\mu_n^*}\big[
          h_n(\boldsymbol{\phi}X_n) \Delta^{\boldsymbol{\phi}}_{in}\big|\A_n^{k_n}\big]\big|\mathbb{G}\big]
      \big\|_1
      \\& \le \mathbb{E}\big[\big|\tfrac{\eta(n)^2-\hat{\eta}^2_{n,K} }{\eta(n)^2}\big|\big]
      + K_1C_2\big(\tfrac{h_n}{\eta(n)}\big)\tsum_{j>k'_n}c_{j,2}\big(\tfrac{h_n}{\eta(n)}\big)
      +\tfrac{K_2k_n^4}{|\A_n|} C_2\big(\tfrac{h_n}{\eta(n)}\big)^2\;,
  \end{split}
  \end{equation*}
  where ${K_1=O(  \mathcal{S}^n_w  |\B_{K}| )}$ and ${K_2=O(  \mathcal{S}^n_w  |\B_{K}|^2)}$.
  If \eqref{H2:g} holds instead,
  \begin{align*}
    &
      \big\|
      1-\tfrac{\sqrt{|\A_n|}}{\eta(n)}\mathbb{E}\big[\mathbb{E}_{\mu_n^*}\big[
          h_n(\boldsymbol{\phi}X_n) \Delta^{\boldsymbol{\phi}}_{in}\big|\A_n^{k_n}\big]\big|\mathbb{G}\big]
      \big\|_1
      \\& \le
  K_2|\B_{b_n}| C_2\big(\tfrac{h_n}{\eta(n)}\big) \tsum_{j>k'_n}c_{j,2}\big(\tfrac{h_n}{\eta(n)}\big)+ K_1  \mathcal{S}^n_w \frac{k_n^4|\B_{b_n}|^2}{|\A_n|}C_{2}(\frac{h_n}{\eta(n)})^2
      \\&
      +K_3 C_{2+\varepsilon}\big(\tfrac{h_n}{\eta(n)}\big)^2\tfrac{k_n^2|\B_{b_n}|}{|\A_n|}   \mathcal{S}^n_w \mathcal{R}_{n}(b_n)
      +\mathbb{E}\big[\big|\tfrac{\eta(n)^2-\hat{\eta}^2_{n,b_n} }{\eta(n)^2}\big|\big]\;,
  \end{align*}
where $K_1=O(1)$ and $K_2=O(  \mathcal{S}^n_w )$ and $K_3=O(1)$.
\end{lemma}
\begin{proof}
Assume first that \eqref{H2:g} holds. We use the abbreviation
\\${G^k(\boldsymbol{\phi'}):=h_n(\boldsymbol{\phi'}X_n)-\bar{h}_n^{\boldsymbol{\phi}_{i},b_n,k}(\boldsymbol{\phi'}X_n)}$.
By the triangle inequality,
\begin{align}
  \label{proof:eq:third:term:main}
  \begin{split}
    & \;
    \Big\|\frac{\eta(n)^2-|\A_n|\sum_i\mathbb{E}_{\mu_n^{\otimes 2}}[
    \mathbb{E}[\bar{h}_n^i(\boldsymbol{\phi}X_n)G^{k'_n}(\bphi')|\group]
    |\A_n^{2k_n}]}{\eta(n)^2}\Big\|_1
    \\\le &\; \Big\|\frac{|\A_n|\sum_i\mathbb{E}_{\mu_n^{\otimes 2}}[\mathbb{E}[
          \bar{h}_n^i(\boldsymbol{\phi}X_n)(\bar{h}_n^{\boldsymbol{\phi}_i,b_n,k'_n}(\boldsymbol{\phi}'X_n)-\bar{h}_n^{\boldsymbol{\phi}_i,b_n,k_n}(\boldsymbol{\phi}'X_n))|\mathbb{G}]\big|\A_n^{2k_n}]}{\eta(n)^2}\Big\|_1
    \\ + &\;
    \Big\|\frac{\hat{\eta}^2_{n,b_n} -|\A_n|\sum_i\mathbb{E}_{\mu_n^{\otimes 2}}[\mathbb{E}[
          \bar{h}_n^i(\boldsymbol{\phi}X_n) G^{k_n}(\bphi')|\mathbb{G}]|\A_n^{2k_n}]}{\eta(n)^2}\Big\|_1
     +
    \mathbb{E}\big[\big|\tfrac{\eta(n)^2-\hat{\eta}^2_{n,b_n} }{\eta(n)^2}\big|\big]
 \\ =: &\; \text{(a)} + \text{(b)} + \text{(c)} \;.
  \end{split}
\end{align}
We can further bound terms (a) and (b).
By definition of the Lipschitz coefficients,
\begin{align*}
\text{(a)}
 & \le \sum_i\Big\|\frac{|\A_n|\mathbb{E}_{\mu_n^{\otimes 2}}[\sum_{j\in \mathcal{I}_{b_n,k_n}(\boldsymbol{\phi}_i,\boldsymbol{\phi}')\setminus\mathcal{I}_{b_n,k'_n}(\boldsymbol{\phi}_i,\boldsymbol{\phi}')}c_{i,2}(\frac{h_n}{\eta(n)})c_{j,2}(\frac{h_n}{\eta(n)})|\A_n^{2k_n}\big]}{\eta(n)^2}\Big\|_1
\\& \le   \mathcal{S}^n_w |\B_{b_n}| \tsum_{i}\tsum_{j>k'_n}c_{i,2}\big(\tfrac{h_n}{\eta(n)}\big) c_{j,2}\big(\tfrac{h_n}{\eta(n)}\big)\;.
\end{align*}
To bound (b), abbreviate
${H(\bphi,\bphi'):= \bar{h}_n^i(\boldsymbol{\phi}X_n)\bigl(h_n(\boldsymbol{\phi}'X_n)-\bar{h}_n^{\boldsymbol{\phi}_i,b_n,k_n}(\boldsymbol{\phi}'X_n)\bigr)}$, and consider the index set
\begin{equation}
  \label{index:set:J}
  \mathcal{J}(\bphi,\bphi'):=\braces{i,j|d(\bphi_i,\bphi'_j)\leq b_n}\;.
\end{equation}
Then for all $\bphi,\bphi'\in \mathbb{G}^{k_n}$ such that $\mathcal{J}(\bphi,\bphi')=\braces{i,j}$, let
 $\bpsi,\bpsi'$ be two elements of $\group^{k_n}$ such that, for the same index pair $(i,j)$,
\begin{equation}\label{lausanne}
  \bpsi_i=\bphi_i\quad\text{ and }\quad\bpsi'_j=\bphi'_j,\qquad \mathcal{J}(\bpsi,\bpsi')=\braces{i,j}\;.
\end{equation}
For the remainder of the proof, denote the concatenation of two vectors as
\begin{equation*}
  [\bphi,\bpsi]:=(\phi_1,\ldots,\phi_m,\psi_1,\ldots,\psi_n)\quad\text{ for } \bphi=(\phi_1,\ldots,\phi_m),\bpsi=(\psi_1,\ldots,\psi_n)\;.
\end{equation*}
Using a telescopic sum, we have
\begin{align*}
&
 \big\|\mathbb{E}\big[\tfrac{1}{\eta(n)^2} H(\bphi,\bphi')\big|\group\big]-\mathbb{E}\big[\tfrac{1}{\eta(n)^2}
   H(\bpsi,\bpsi')\big|\group\big]\big\|_{\subL1}
 \\&
 \le \tsum_{l=0}^{k_n-1}\big\|\mathbb{E}\big[\tfrac{1}{\eta(n)^2}
   \big(H([\bpsi_{1:l},\bphi_{l+1:k_n}],\bphi')-
   H([\bpsi_{1:l+1},\bphi_{l+2:k_n}],\bphi')\big)\big|\group\big]\big\|_{\subL1}\\
 & +
 \tsum_{l=0}^{k_n-1}\big\|\mathbb{E}\big[\tfrac{1}{\eta(n)^2}
   \bigl(H(\bpsi,[\bpsi_{1:l}',\bphi_{l+1:k_n}'])-
   H(\bpsi,[\bpsi_{1:l+1}',\bphi_{l+2:k_n}'])\bigr)\big|\group\big]\big\|_{\subL1}
 \\& \overset{(*)}{\le} 8\tsum_{l\ne i} c_{l,2+\varepsilon}\big(\tfrac{h_n}{\eta(n)}\big) c_{j,2+\varepsilon}\big(\tfrac{h_n}{\eta(n)}\big)
 \alpha_n^{\frac{\varepsilon}{2+\varepsilon}}\big(\bar{d}([\bpsi_l,\boldsymbol{\phi}_l],~[\boldsymbol{\phi}',\boldsymbol{\phi}_{l+1:k_n},\bpsi_{1:l-1}])\big|\group\big)
\\& + 8\tsum_{l\ne j}  c_{l,2+\varepsilon}\big(\tfrac{h_n}{\eta(n)}\big) c_{i,2+\varepsilon}\big(\tfrac{h_n}{\eta(n)}\big)
\alpha_n^{\frac{\varepsilon}{2+\varepsilon}}\bigl(\bar{d}([\bpsi'_l,\boldsymbol{\phi}'_l],~[\boldsymbol{\phi},\boldsymbol{\phi}'_{l+1:k_n},\bpsi'_{1:l-1}])|\group\bigr)\;.
\end{align*}
where (*) is follows from \cref{lemma:mixing} and inequality
\begin{align*}&\Big\|\tfrac{1}{\eta(n)^2}
   \big(H([\bpsi_{1:l},\bphi_{l+1:k_n}],\bphi')-
   H([\bpsi_{1:l+1},\bphi_{l+2:k_n}],\bphi')\big)\Big\|_{1+\frac{\epsilon}{2}}\\&\le 2c_{l,2+\varepsilon}\big(\tfrac{h_n}{\eta(n)}\big) c_{j,2+\varepsilon}\big(\tfrac{h_n}{\eta(n)}\big)\;.\end{align*}
By definition, ${\widehat{\mathbb{F}}_{\infty,i}(h_n,X_n,\boldsymbol{\phi}_i) \widehat{\mathbb{F}}_{\infty,j}(h_n,X_n,\boldsymbol{\phi}'_j)}$
is the average of $H(\bpsi,\bpsi')$ over the set of pairs $(\bpsi,\bpsi')$ satisfying \eqref{lausanne}.
Therefore, for ${(i,j)=\mathcal{J}(\bphi,\bphi')}$,
\begin{align*}
 &\Big\| \mathbb{E}\big[\tfrac{1}{\eta(n)^2} H(\boldsymbol{\phi},\boldsymbol{\phi}')\big|\mathbb{G}\big]
  -
  \mathbb{E}\big[\tfrac{1}{\eta(n)^2} \widehat{\mathbb{F}}_{\infty,i}(h_n,X_n,\boldsymbol{\phi}_i) \widehat{\mathbb{F}}_{\infty,j}(h_n,X_n,\boldsymbol{\phi}'_j)|\mathbb{G}\big] \Big\|_{\subL1}
\\& \le 8\tsum_{l\ne i} c_{l,2+\varepsilon}\big(\tfrac{h_n}{\eta(n)}\big) c_{j,2+\varepsilon}\big(\tfrac{h_n}{\eta(n)}\big)
\alpha_n^{\frac{\varepsilon}{2+\varepsilon}}\big(\bar{d}(\boldsymbol{\phi}_l,[\boldsymbol{\phi}',\boldsymbol{\phi}_{l+1:k_n}])\big|\group\big)
\\& + 8\tsum_{l\ne j}  c_{l,2+\varepsilon}\big(\tfrac{h_n}{\eta(n)}\big) c_{i,2+\varepsilon}\big(\tfrac{h_n}{\eta(n)}\big)
\alpha_n^{\frac{\varepsilon}{2+\varepsilon}}\big(\bar{d}(\boldsymbol{\phi}'_l,[\boldsymbol{\phi}'_{l+1:k_n},\boldsymbol{\phi}_{i}])|\group\bigr)\;.
\end{align*}
For all $i,j\le k_n$, we hence obtain
{\begin{align*}
\label{basel}
&  \Big\| \mathbb{E}_{\mu_n^{\otimes 2}}\big[\tfrac{\mathbb{I}\braces{\mathcal{J}(\bphi,\bphi')=\{i,j\}}
    (H(\boldsymbol{\phi},\boldsymbol{\phi}')- \widehat{\mathbb{F}}_{\infty,i}(h_n,X_n,\boldsymbol{\phi}_i) \widehat{\mathbb{F}}_{\infty,j}(h_n,X_n,\boldsymbol{\phi}'_j))}{\eta(n)^2}\big|\A_n^{2k_n}\big]\Big\|_1
\\& \le 32\bigl(\tsum_{l} c_{l,2+\varepsilon}\big(\tfrac{h_n}{\eta(n)}\big)\big)^2
\tfrac{k_n^2|\B_{b_n}|}{|\A_n|} \tsum_{m\ge b_n}|\B_{m+1}\!\setminus\!\B_m|  \mathcal{S}^n_w \alpha_n^{\frac{\varepsilon}{2+\varepsilon}}(m|\mathbb{G})\;.
\end{align*}}
We can then upper-bound (b) as
\begin{align*}
&\big\||\A_n| \mathbb{E}_{\mu_n^{\otimes 2}}\big[\tfrac{\mathbb{I}\braces{\mathcal{J}(\bphi,\bphi')=\{i,j\}}
    (H(\boldsymbol{\phi},\boldsymbol{\phi}')- \widehat{\mathbb{F}}_{\infty,i}(h_n,X_n,\boldsymbol{\phi}_i) \widehat{\mathbb{F}}_{\infty,j}(h_n,X_n,\boldsymbol{\phi}'_j))}{\eta(n)^2}\big|\A_n^{2k_n}\big]\big\|_1
\\&+\big\||\A_n| \mathbb{E}_{\mu_n^{\otimes 2}}\big[\tfrac{\mathbb{I}\braces{\mathcal{J}(\bphi,\bphi')\subsetneq\{i,j\}}
    (H(\boldsymbol{\phi},\boldsymbol{\phi}')- \widehat{\mathbb{F}}_{\infty,i}(h_n,X_n,\boldsymbol{\phi}_i) \widehat{\mathbb{F}}_{\infty,j}(h_n,X_n,\boldsymbol{\phi}'_j))}{\eta(n)^2}\big|\A_n^{2k_n}\big]\big\|_1
\\& =: b^1_{ij}+ b^2_{ij} \; \ge \text{(b)}\;.
\end{align*}
We have already obtained a bound for $b^1_{ij}$ above.
For $b^2_{ij}$, the Cauchy-Schwartz inequality yields
\begin{align*}
  \tsum_{ij}b^2_{ij}&\le  4|\A_n| M_2\big(\tfrac{h_n}{\eta(n)}\big)^2\tsum_{ij}\mathbb{E}\Big[\mathbb{E}_{\mu_n^{\otimes 2}}\big[\mathbb{I}\braces{\mathcal{J}(\bphi,\bphi')\subsetneq\{i,j\}}\big|A_n^{2k_n}\big]\Big]
\\& \le 4\tfrac{  \mathcal{S}^n_w  |\B_{b_n}|^2k_n^4}{|\A_n|} M_2\big(\tfrac{h_n}{\eta(n)}\big)^2\;.
\end{align*}
Substituting the bounds for (a) and (b) so obtained back into \eqref{proof:eq:third:term:main}
then completes the proof under hypothesis \eqref{H2:g}. If \eqref{H1:g} holds instead,
correlations between elements separated by a distance exceeding some constant $K$ have no effect.
In this case,
\begin{equation*}\begin{split}&
 \big\|t'(W^*)\bigl(1-\tfrac{\sqrt{|\A_n|}}{\eta(n)}
 \mathbb{E}\big[\mathbb{E}_{\mu_n^*}\big[
     h_n(\boldsymbol{\phi}X_n) \Delta^{\boldsymbol{\phi}}_{in}\big|\A_n^{k_n}\big]\big|\mathbb{G}\big]\big\|_1
\\&
\le \sqrt{\tfrac{2}{\pi}} \Bigl(
\mathbb{E}\big[\big|\tfrac{\eta(n)^2-\hat{\eta}^2_{n,K} }{\eta(n)^2}\big|\big]
+  \mathcal{S}^n_w  |\B_{K}|\tsum_{i}\tsum_{k'_n<j\le k_n}  c_{i,2}\big(\tfrac{h_n}{\eta(n)}\big)
c_{j,2}\big(\tfrac{h_n}{\eta(n)}\big)
\\&+4\tfrac{  \mathcal{S}^n_w  |\B_{K}|^2k_n^4}{|\A_n|} M_2\big(\tfrac{h_n}{\eta(n)}\big)^2 \Bigr)\;,
 \end{split}\end{equation*}
which completes the proof.
\end{proof}

\subsection{The fourth term in Lemma \ref{neuchatel}}
\label{sec:term:4}

\newcommand\mydots{\hbox to 1em{.\hss.\hss.}}

The final term in \cref{neuchatel} represents variation of $\eta(n)$, and we upper-bound it in terms of its variance.
As in the proof of the basic case, $\eta(n)$ can be thought of as an empirical variance, and its variance is a fourth-order quantity.
Since the fourth moment of $h_n(X_n)$ may not exist,
we control it using the sequence $(\gamma_n)$. 
\\To bound the standard deviation, we have to consider interactions between
quadruples ${\bphi_1,\ldots,\bphi_4}$ of random elements of $\group^{k_n}$.
Once again, $n$, $b_n$, $\beta_n$ and $k_n$ are fixed. For a quadruple of indices ${i,j,l,m}$, we consider the events
\begin{align}
  \label{eq:interaction:condition:4}
  d(\boldsymbol{\phi}_{1,i}, \boldsymbol{\phi}_{2,j}) \le b_n
  \quad&\quad
  d(\boldsymbol{\phi}_{3,l}, \boldsymbol{\phi}_{4,m}) \le b_n
  \\
  \label{eq:interaction:condition:5}
  \text{and}\quad\boldsymbol{\phi}_1\in V_{i,\beta_n}
  \quad
  \boldsymbol{\phi}_2\in V_{j,\beta_n}
  \;\; & \;\;
  \boldsymbol{\phi}_3\in V_{l,\beta_n}
  \quad
  \boldsymbol{\phi}_4\in V_{m,\beta_n}\;.
\end{align}
Since $n$ is fixed, we can then choose a constant $S^{\ast}_4$ such that
\begin{equation*}
  \tfrac{|\A_n|^3}{|A||\B_{b_n}|^2} \big\|\mathbb{E}_{\mu_n^{\otimes 4}}
  \big[
    \mathbb{I}\braces{\bphi_1,\!\mydots,\bphi_4\text{ satisfy }\eqref{eq:interaction:condition:4},\eqref{eq:interaction:condition:5}
    \text{ and }\boldsymbol{\phi}_{2,j}^{-1}\boldsymbol{\phi}_{3,m}\!\in\! A}\big|\A_n^{4k_n}\big]\big\|\le S^*_4
\end{equation*}
holds for every Borel set ${A\subset\group^{k_n}}$ with ${|pr_j(A)|\geq 1}$ for all $j\le k_n$.
\begin{lemma}
  \label{stein:term4}
  Fix ${p,q>0}$ with $\frac{1}{p}+\frac{1}{q}=1$.
  Assume \eqref{H1:g} holds. Then
  \begin{align*}
    &
    \tsum_i \big\|
    \tfrac{\sqrt{|\A_n|}}{\eta(n)}\mathbb{E}_{\mu_n^*}\big[
      \bar{h}_n^i(\boldsymbol{\phi}X_n) \Delta^{\boldsymbol{\phi}}_{in}- \mathbb{E}[\bar{h}_n^i(\boldsymbol{\phi}X_n) \Delta^{\boldsymbol{\phi}}_{in}|\mathbb{G}]\big|\A_n^{k_n}
      \big]\big\|_{\subL1}
    \\& \le K_1 \tfrac{{k'_n}^2}{\sqrt{|\A_n|}} \Gamma^2_{4(1+\frac{\varepsilon}{2})}\big(\gamma_n)\sqrt{S^*_4}+\tfrac{K_2k_n^4}{|\A_n|}
    C^2_2\big(\tfrac{h_n}{\eta(n)}\big)+K_3\big[ C_2\big(\tfrac{h_n}{\eta(n)}\big)^2|\B_{k}|S^{\ast}_0
    \\&+  C_{2}\big(\tfrac{h_n}{\eta(n)}\big)\!\tsum_i
    \bigl(\mathbb{E}[\frac{|\widehat{\mathbb{F}}_{\infty,i}(h_n,X_n,e)|^{2}}{\eta(n)^2}
      \mathbb{I}\braces{|\widehat{\mathbb{F}}_{\infty,i}(h_n,X_n,e)|> \gamma_nc_{i,2}(h_n)}]\bigr)^{\frac{1}{2}}\Big]\;,
  \end{align*}
  where $K_1=O(|\B_{K}|^{\frac{3}{2}})$ and  $K_2=O(  \mathcal{S}^n_w  |\B_{K}|^2)$ and $K_3=O(  \mathcal{S}^n_w |B_K|)$.
  If \eqref{H2:g} holds instead, then
  \begin{align*}
    &
    \tsum_i \big\|
    \tfrac{\sqrt{|\A_n|}}{\eta(n)}\mathbb{E}_{\mu_n^*}\big[
      \bar{h}_n^i(\boldsymbol{\phi}X_n) \Delta^{\boldsymbol{\phi}}_{in}- \mathbb{E}[\bar{h}_n^i(\boldsymbol{\phi}X_n) \Delta^{\boldsymbol{\phi}}_{in}|\mathbb{G}]\big|\A_n^{k_n}
      \big]\big\|_{\subL1}
    \\& \le K_1\Bigl(|\B_{b_n}|S^{\ast}_0C^2_{2+\varepsilon}\big(\tfrac{h_n}{\eta(n)}\big)
    +\tfrac{  |\B_{b_n}|^2k_n^4}{|\A_n|}C_{2+\varepsilon}(\tfrac{h_n}{\eta(n)})^2\Bigr)
    \\&+ K_2      \mathcal{S}^n_w \tfrac{k_n^2|\B_{b_n}|\mathcal{R}_{n}(b_n) C^2_{2+\varepsilon}\big(\tfrac{h_n}{\eta(n)}\big)}{|\A_n|}
    \\&+
    K_3|\B_{b_n}| C_{2+\varepsilon}\big(\tfrac{h_n}{\eta(n)}\big)\!\tsum_i
    \bigl(\mathbb{E}\big[
    \tfrac{|\widehat{\mathbb{F}}_{\infty,i}(h_n,X_n,e)|^{2}
    \mathbb{I}\braces{|\widehat{\mathbb{F}}_{\infty,i}(h_n,X_n,e)|> \gamma_n
    c_{i,2}(h_n)}}{\eta(n)^2}
    \big]\bigr)^{\frac{1}{2}}
    \\& +
    K_4\frac{|\B_{b_n}|{k'_n}^2 \Gamma^2_{4(1+\frac{\varepsilon}{2})}\big({ \gamma_n})}{\sqrt{|\A_n|}} \sqrt{S^*_4}\;,
  \end{align*}
  for  $K_1=O(  \mathcal{S}^n_w )$ and $K_2=O(1)$ and $K_3=O(  \mathcal{S}^n_w )$ and $K_4=O( \mathcal{R}^{\frac{1}{2}}_n(0))$.
\end{lemma}

\begin{proof}
First suppose \eqref{H2:g} holds.
As in \cref{stein:term3}, we abbreviate
\begin{equation*}
  H(\bphi,\bphi',i)= \bar{h}_n^i(\boldsymbol{\phi}X_n)\bigl(h_n(\boldsymbol{\phi}'X_n)-\bar{h}_n^{\boldsymbol{\phi}_i,b_n,k_n'}(\boldsymbol{\phi}'X_n)\bigr)\;,
\end{equation*}
where we now keep track of the
index $i$. We conditionally center $H$,
\begin{equation*}
  \overline{H}(\bphi,\bphi',i):=H(\bphi,\bphi',i)-\mean[H(\bphi,\bphi',i)|\group]\;.
\end{equation*}
Interactions between  $\hat{\mathbb{F}}_{\infty,i}$ for different values of $i$
involve terms of the form
\begin{equation*}
  \begin{split}
  F_{ij}(\phi,\phi',\tau) =\ &\,
  \widehat{\mathbb{F}}_{\infty,i}(h_n,X_n,\phi)
  \mathbb{I}\braces{\widehat{\mathbb{F}}_{\infty,i}(h_n,X_n,\phi)\leq \tau c_{i,2}({h_n(X_n)})}\\
  \times &\,
  \widehat{\mathbb{F}}_{\infty,j}(h_n,X_n,\phi')
  \mathbb{I}\braces{\widehat{\mathbb{F}}_{\infty,j}(h_n,X_n,\phi')\leq \tau c_{j,2}({h_n(X_n))}}
  \end{split}
\end{equation*}
for ${\phi,\phi'\in\group}$ and some threshold ${\tau\in(0,\infty]}$. We again center conditionally,
\begin{equation*}
  \overline{F}_{ij}(\phi,\phi',\tau)=
  F_{ij}(\phi,\phi',\tau)-\mathbb{E}[F_{ij}(\phi,\phi',\tau)|\group]
\end{equation*}
Abbreviate ${J_{ij}=\mathbb{I}\braces{j\in\mathcal{I}_{b_n,k'_n}(\boldsymbol{\phi}_i,\boldsymbol{\phi}'),(\boldsymbol{\phi},\boldsymbol{\phi}')\!\in\! V_{i,\beta_n}\!\times\! V_{j,\beta_n}}}$.
Using the triangle inequality, we obtain:
\begin{align*}
 & \tsum_i \big\|\tfrac{\sqrt{|\A_n|}}{\eta(n)}\mathbb{E}_{\mu_n^*}\big[
    \bar{h}_n^i(\boldsymbol{\phi}X_n) \Delta^{\boldsymbol{\phi}}_{in}-
    \mathbb{E}[\bar{h}_n^i(\boldsymbol{\phi}X_n) \Delta^{\boldsymbol{\phi}}_{in}|\mathbb{G}]\big|\A_n^{k_n}\big]\big\|_{\subL1}
  \\& \le
 \tsum_{i,j} \big\|\tfrac{|\A_n|}{\eta(n)^2}\mathbb{E}_{\mu_n^{\otimes 2}}\bigl[
 J_{ij}
    \overline{H}(\bphi,\bphi',i) \big|\A_n^{2k_n}\bigr]\big\|_{\subL1}
  \\& +
\tsum_{i,j}  \big\|\tfrac{|\A_n|}{\eta(n)^2}\mathbb{E}_{\mu_n^{\otimes 2}}\bigl[
(1-J_{ij})
    \overline{H}(\bphi,\bphi',i) \big|\A_n^{2k_n}\bigr]\big\|_{\subL1}
\\&=: \tsum_{i,j} a_{ij} + \tsum_{i,j} b_{ij}\;.
\end{align*}
Consider $a_{ij}$ first. By the triangle inequality
\begin{align*}
a_{ij}
 & \le\Big|
 \mean\! \Big[\!\tfrac{|\A_n|}{\eta(n)^2}\mathbb{E}_{\mu_n^{\otimes 2}}\!\Bigl[\mathbb{I}\braces{j\!\in\!\mathcal{I}_{b_n,k'_n}(\boldsymbol{\phi}_i,\boldsymbol{\phi}')}
      \big| \mathbb{E}[|\overline{H}(\bphi,\bphi',i)|\,|\group]
   \!-\!\mean[|\overline{F}_{ij}(\bphi_i,\bphi'_j,\infty)|\,|\group]\big|\Big|\A_n^{2k_n}\Bigr]\Big]\Big|
  \\& \quad+\Big\|\tfrac{|\A_n|}{\eta(n)^2}\mathbb{E}_{\mu_n^{\otimes 2}}\bigl[ J_{ij}
 \overline{F}_{ij}(\bphi_i,\bphi'_j,\infty)\bigr]\Big\|_{1}
 \\& =: a'_{ij} + a''_{ij}\;.
\end{align*}
To bound $a'_{ij}$, we proceed similarly as in the proof of \cref{stein:term3}:
Recall the index set $\mathcal{J}(\bphi,\bphi')$ in \eqref{index:set:J}.
If ${\bphi,\bphi'\in \group^{k_n}}$ satisfy $\mathcal{J}(\bphi,\bphi')=\{i,j\}$, we have
\begin{align*}
 &\Big\|\tfrac{1}{\eta(n)^2}\big(
   \mean\big[|\overline{H}(\bphi,\bphi',i)|\big|\group\big]-   \mean\big[|\overline{F}_{ij}(\bphi_i,\bphi'_j,\infty)
 | \big|\group\big]\big) \Big\|_{\subL1}
\\& \le   8\tsum_{l\ne i} c_{l,2+\varepsilon}\big(\tfrac{h_n}{\eta(n)}\big)
c_{j,2+\varepsilon}\big(\tfrac{h_n}{\eta(n)}\big)
\alpha_n^{\frac{\varepsilon}{2+\varepsilon}}\big(d(\boldsymbol{\phi}_l,[\boldsymbol{\phi}',\boldsymbol{\phi}_{l+1:k_n}])\big|\mathbb{G}\big)
\\&+ 8\tsum_{l\ne j}  c_{l,2+\varepsilon}\big(\tfrac{h_n}{\eta(n)}\big)
c_{i,2+\varepsilon}\big(\tfrac{h_n}{\eta(n)}\big)
\alpha_n^{\frac{\varepsilon}{2+\varepsilon}}\bigl(d(\boldsymbol{\phi}'_l,[\boldsymbol{\phi}'_{l+1:k_n},\boldsymbol{\phi}_{i}])|\mathbb{G}\bigr)\;.
\end{align*}
Applying \cref{lemma:mixing} and the definition of the random measure $\mu_n^*$ gives
\begin{align*}
  &
 \tsum_{i,j}\mean\Big[\tfrac{|\A_n|}{\eta(n)^2}\mathbb{E}_{\mu_n^{\otimes 2}}\!\Big[
 \mathbb{I}\braces{\mathcal{J}(\boldsymbol{\phi},\boldsymbol{\phi}')\!=\!\{i,j\}\!}
 \\&\hspace{8em}\big|\mean[|\overline{H}(\bphi,\bphi',i)|\big|\group]\!-\!\mean[|\overline{F}_{ij}(\bphi_i,\bphi'_j,\infty)|\big|\group]\big|
 \;\Big|\A_n^{2k_n}\!\Big]\Big]
 \\&\le  32\,  \mathcal{S}^n_w \,\tfrac{k_n^2|\B_{b_n}|}{|\A_n|}\bigl(\tsum_l c_{l,2+\varepsilon}\big(
   \tfrac{h_n}{\eta(n)}\big)\bigr)^2\tsum_{i\ge b_n}|\B_{i+1}\!\setminus\! \B_{i}|
   \alpha_n^{\frac{\varepsilon}{2+\varepsilon}}(i|\mathbb{G})\;.
\end{align*}
Again similarly to the proof of  \cref{stein:term3}, we obtain
\begin{align*}
  &
\tsum_{i,j} \mean\Big[\tfrac{|\A_n|}{\eta(n)^2}\mathbb{E}_{\mu_n^{\otimes 2}}\!\big[
 \mathbb{I}\braces{\mathcal{J}(\boldsymbol{\phi},\boldsymbol{\phi}')\!\not\subset\!\{i,j\}\!}\\&\qquad \qquad \times
 \big|\mean[|\overline{H}(\bphi,\bphi',i)|\big|\group]\!-\!\mean[|\overline{F}_{ij}(\bphi_i,\bphi'_j,\infty)|\big|\group]\big|
 |\A_n^{2k_n}\!\big]\Big]
\\&\le {4 |\B_{b_n}|^2}M_2\big(\tfrac{h_n}{\eta(n)}\big)^2|\A_n|\mathbb{E}\Big[\mathbb{E}_{\mu_n^{\otimes 2}}\big[ \mathbb{I}\braces{\mathcal{J}(\boldsymbol{\phi},\boldsymbol{\phi}')\!\not\subset\!\{i,j\}}|A_n^{2k_n}\big]\Big]
\\
 & \le \tfrac{4   \mathcal{S}^n_w  |\B_{b_n}|^2k_n^4}{|\A_n|}M_2\big(\tfrac{h_n}{\eta(n)}\big)^2\;.
\end{align*}
Hence
\begin{align*}
 \tsum_{i,j}a'_{ij} &\le\; 32\,  \mathcal{S}^n_w \tfrac{k_n^2|\B_{b_n}|}{|\A_n|}\bigl(\tsum_l c_{l,2+\varepsilon}\big(
   \tfrac{h_n}{\eta(n)}\big)\bigr)^2\tsum_{i\ge b_n}|\B_{i+1}\!\setminus\! \B_{i}|
   \alpha_n^{\frac{\varepsilon}{2+\varepsilon}}(i|\mathbb{G})\\&+
  \tfrac{4   \mathcal{S}^n_w  |\B_{b_n}|^2k_n^4}{|\A_n|}M_2\big(\tfrac{h_n}{\eta(n)}\big)^2\;.
\end{align*}
To bound $a''_{ij}$, abbreviate ${J_{ij}':=\mathbb{I}\braces{\bphi\in V_{i,\beta_n},\boldsymbol{\phi}'\in V_{j,\beta_n},d(\boldsymbol{\phi}_i,\boldsymbol{\phi}'_j)\le {b_n}}}$.
Then
\begin{align*}
  a''_{ij} &\le \big\||\A_n|\mathbb{E}_{\mu_n^{\otimes 2} }\big[
  J'_{ij}
  \tfrac{\overline{F}_{ij}(\bphi_i,\bphi'_j,\infty)-\overline{F}_{ij}(\bphi_i,\bphi'_j,\gamma_n)}{\eta(n)^2}
  \big|\A_n^{2k_n}\big]\big\|_{\subL1}
  \\&+\big\||\A_n|\mathbb{E}_{\mu_n^{\otimes 2} }\big[
  J'_{ij}
  \tfrac{\overline{F}_{ij}(\bphi_i,\bphi'_j,\gamma_n)}{\eta(n)^2}
  \big|\A_n^{2k_n}\big]\big\|_{2}
\end{align*}
The first term can be bounded using Cauchy-Schwartz, as
\begin{align*}
  &\big\|\msum_{i,j}|\A_n|\mathbb{E}_{\mu_n^{\otimes 2} }\big[
  J'_{ij}
  \tfrac{\overline{F}_{ij}(\bphi_i,\bphi'_j,\infty)-\overline{F}_{ij}(\bphi_i,\bphi'_j,\gamma_n)}{\eta(n)^2}
  \big|\A_n^{2k_n}\big]\big\|_{\subL1}
  \\& \le
  4\tsum_{\min\braces{i,j}\le k'_n}
  \Bigl(
  \mathbb{E}[\mfrac{|\widehat{\mathbb{F}}_{\infty,i}(h_n,X_n,e)|^{2}
    \mathbb{I}\braces{|\widehat{\mathbb{F}}_{\infty,i}(h_n,X_n,e)|> \gamma_nc_{i,2}(h_n)}}{\eta(n)^2}]
  \Bigr)^{\frac{1}{2}}\\
  & \qquad \quad\;c_{j,2}\big(\tfrac{h_n}{\eta(n)}\bigr)]
      \mathbb{E}[\mathbb{E}_{\mu_n^{\otimes 2}}\mathbb{I}\braces{\bphi_i^{-1}\bphi'_j\in \B_{b_n}}|\A_n^{2k_n}]]
        \\&
        \le 8   \mathcal{S}^n_w |\B_{b_n}|\big(\tsum_{j} c_{j,2}\big(\tfrac{h_n}{\eta(n)}\big)\big)
        \\&
        \quad\;
        \tsum_i\bigl(\mathbb{E}[|\mfrac{1}{\eta(n)^2}\widehat{\mathbb{F}}_{\infty,i}(h_n,X_n,e)|^{2}
              \mathbb{I}\braces{|\widehat{\mathbb{F}}_{\infty,i}(h_n,X_n,e)|> \gamma_nc_{i,2}(h_n)}]
        \bigr)^{\frac{1}{2}}\;.
\end{align*}
The second term involves four-way interactions, so some abbreviations are helpful:
Set ${\zeta_i:=\|\widehat{\mathbb{F}}_{\infty,i}(h_n,X_n,e)\mathbb{I}\braces{|
    \widehat{\mathbb{F}}_{\infty,i}(h_n,X_n,e)|\le \gamma_n c_{i,2}(h_n)}\|_{4+2\varepsilon}}$
and ${\widehat{\mathbb{F}}_{\infty,i}^{\gamma_n}:=\min\braces{\widehat{\mathbb{F}}_{\infty,i},\gamma_n}}$.
For ${\phi,\phi',\psi,\psi'\in\group}$ and indices ${i,j,l,m}$, we have
\begin{align*}
&\big\|\text{\rm Cov}[
        \widehat{\mathbb{F}}^{\gamma_n}_{\infty,i}(h_n,X_n,\phi)\widehat{\mathbb{F}}^{\gamma_n}_{\infty,l}(h_n,X_n,\phi'),\;
        \widehat{\mathbb{F}}^{\gamma_n}_{\infty,j}(h_n,X_n,\psi)\widehat{\mathbb{F}}^{\gamma_n}_{\infty,m}(h_n,X_n,\psi')
]\big\|_1
\\& \hspace{5em}\le 4\,\zeta_i\,\zeta_j\,\zeta_l\,\zeta_m\,
\alpha_n^{\frac{\varepsilon}{2+\varepsilon}}\bigl(\bar{d}((\phi,\phi'),(\psi,\psi'))\big|\mathbb{G}\bigr)
\end{align*}
Therefore, by definition of $S_4^{\ast}$, we have
\begin{equation*}\begin{split}&
\sum_{i\le k_n,j\le k'_n}\big\||\A_n|\mathbb{E}_{\mu_n^{\otimes 2} }\big[
  \mathbb{I}\braces{\boldsymbol{\phi}'\in V_{j,\beta_n},d(\boldsymbol{\phi}_i,\boldsymbol{\phi}'_j)\le {b_n}}
  \tfrac{\overline{F}_{ij}(\bphi_i,\bphi'_j,\gamma_n)}{\eta(n)^2}
  \big|\A_n^{2k_n}\big]\big\|_{2}\\&\le 8 \tfrac{|\B_{b_n}|{k'_n}^2}{\sqrt{|\A_n|}}
\Bigl(S^*_4\tsum_i |\B_{i+1}\!\setminus\!\B_{i}|\alpha_n^{\frac{\varepsilon}{2+\varepsilon}}(i|\mathbb{G})\Bigr)^{\frac{1}{2}}
\tsum_{i\le k_n,j\le k'_n} \zeta_i\,\zeta_j\;.
\end{split}\end{equation*}
In summary, we can upper-bound $a''_{ij}$ as
\begin{equation*}\begin{split}
\tsum_{i\le k_n,j\le k'_n}&a''_{ij}\le 8 \tfrac{|\B_{b_n}|{k'_n}^2}{\sqrt{|\A_n|}}
\Bigl(S^*_4\tsum_i |\B_{i+1}\!\setminus\!\B_{i}|\alpha_n^{\frac{\varepsilon}{2+\varepsilon}}(i|\mathbb{G})\Bigr)^{\frac{1}{2}}
\tsum_{i\le k_n,j\le k'_n} \zeta_i\,\zeta_j\\
&+ 8\,  \mathcal{S}^n_w \,|\B_{b_n}|\big(\tsum_{j} c_{j,2}\big(\tfrac{h_n}{\eta(n)}\big)\big)
\\&
\quad\;
\tsum_i\bigl(\mathbb{E}[|\widehat{\mathbb{F}}_{\infty,i}(h_n,X_n,e)|^{2}
\mathbb{I}\braces{|\widehat{\mathbb{F}}_{\infty,i}(h_n,X_n,e)|> \gamma_nc_{i,2}(h_n)}]
\bigr)^{\frac{1}{2}}\;.
\end{split}\end{equation*}
An upper bound on the final term $\sum_{i,j}(b_{ij})$ is, by Cauchy-Schwartz,
\begin{align*}
2\bigl(\tsum_i c_{i,2}\big(\tfrac{h_n}{\eta(n)}\big)\bigr)^2
\sup_{i,j}\mathbb{E}[\mathbb{E}_{\mu_n^{\otimes 2}}[|\A_n|\mathbb{I}\braces{\boldsymbol{\phi}'\not\in V_i(\beta_n),d(\boldsymbol{\phi}_i,\boldsymbol{\phi}'_j)\le b_n}|\A_n^{2k_n}]]\;,
\end{align*}
and we have hence obtained all terms in the bound under hypothesis \eqref{H2:g}. If \eqref{H1:g} holds instead, there
is again a constant distance $K$ beyond which correlations vanish, and
\begin{align*}
  & \text{(a)} \le\tfrac{8   \mathcal{S}^n_w  |\B_{K}|^2k_n^4}{|\A_n|}M^2_2\big(\tfrac{h_n}{\eta(n)}\big)
 + |\B_{b_n}|S^{\ast}_0  \mathcal{S}^n_w \bigl(\tsum_i c_{i,2}\big(\tfrac{h_n}{\eta(n)}\big)\big)^2
  \\ &
  \text{(b)} \le 2 \tfrac{|\B_{K}|^{\frac{3}{2}}{k'_n}^2}{\sqrt{|\A_n|}}\sqrt{S_4^{\ast}}\tsum_{i\le k_n,j\le k'_n}\zeta_i\zeta_j\;,
\end{align*}
which completes the proof of the lemma.
\end{proof}

\subsection{Proof of the central limit theorem}
\label{sec:proof:CLT:g}
We complete the proof of \cref{theorem:CLT:g} by showing ${\dW(\sqrt{|\A_n|}\,\widehat{\mathbb{F}}_n(h_n,X_n), \eta Z\bigl)\rightarrow 0}$.  We first note that
\begin{equation}
  \label{eq:variance:convergence}
  \|\widehat{\eta}_{m,n}^2-\eta^2_m\|_{\subL1}\xrightarrow{n\rightarrow\infty}0
  \qquad\text{ for all }m\in\mathbb{N}\;.
\end{equation}
That is the case since, for every $\varepsilon>0$, we have
\begin{align*}
  \mathbb{E}[|\widehat{\eta}_{m,n}^2-\eta^2_m|]
  \le&\;
  \varepsilon +\mathbb{E}[\eta_m^2\mathbb{I}\braces{|\widehat{\eta}_{m,n}^2-\eta^2_m|>\varepsilon})
    +
    \mathbb{E}\big[\widehat{\eta}_{m,n}^2\mathbb{I}\braces{|\widehat{\eta}_{m,n}^2-\eta^2_m|>\varepsilon}\big]
    \\
    \le&\; \varepsilon +\mathbb{E}[\eta_m^2 \mathbb{I}\braces{|\widehat{\eta}_{m,n}^2-\eta^2_m|>\varepsilon}]\\
    &\quad+
    |\B_m|   \mathcal{S}^n_w \bigl(\tsum_i c_{i,2}(h_n\mathbb{I}\braces{|\widehat{\eta}_{m,n}^2-\eta^2_m|>\varepsilon})\bigr)^2\;,
\end{align*}
and  \eqref{eq:variance:convergence} follows by uniform integrability of
  $(h_n(\boldsymbol{\phi}X_n)^{2})_{\boldsymbol{\phi},n}$.

We next must
specify suitable sequences of coefficients $\gamma_n$, $\beta_n$, $k_n$, $k_n'$, and $b_n$
for which the relevant terms in the bounds in \cref{aix_la_chapelle} and \ref{neuchatel}
converge to $0$
as ${n\rightarrow\infty}$.
We first choose $(\gamma_n)$ and $(\beta_n)$ to satisfy
\begin{equation*}
  r^1_n:={\beta_n\gamma_n^2 k_n^2}/{\sqrt{|\A_n|}}  \;\longrightarrow\;0\;.
\end{equation*}
Such sequences exist, since ${k_n^2/\sqrt{|\A_n|}\rightarrow 0}$. Because of \eqref{eq:variance:convergence}, $(b_n)$ can be chosen to additionally satisfy
\begin{equation*}
  \|\widehat{\eta}_{b_n,n}^2-\eta^2_{b_n}\|_{1}
  \xrightarrow{n\rightarrow\infty}0\;.
\end{equation*} In addition we ask that
$(k_n')$ and $(b_n)$ satisfy
\begin{align*}
  r^2_n
  \,:=&\;
  |\B_{b_n}|k_n'S_0^{\ast}
  \;\longrightarrow\;0\\
  r^3_n
  \,:=&\;
  |\B_{b_n}|k_n'
  \tsum_i c_{i,2}\big(\bar{h}_n^i(\boldsymbol{\phi}X_n)\mathbb{I}\braces
       {\tfrac{|\bar{h}_n^i(\boldsymbol{\phi}X_n)|}{ c_{i,2}(h_n)}>\gamma_n}\bigr)
  \;\longrightarrow\;0\\
  r^4_n
  \,:=&\;
  |\B_{b_n}|\Bigl(\tsum_{k'_n<i} c_{i,2+\varepsilon}(h_n)\Bigr)
  \;\longrightarrow\;0\\
  r^5_n
  \,:=&\;
  |\B_{b_n}|\tfrac{k_n^2\gamma_n^2}{\sqrt{|\A_n|}}+\mathcal{R}_{n}(b_n)+k'_n r_n^1
  \;\longrightarrow\;0
\end{align*}
as ${n\rightarrow\infty}$, which is possible since ${S_0^{\ast}\rightarrow 0}$ as ${\beta_n\rightarrow\infty}$.
Consequently, we can choose sequences ${(\delta_n)}$ and ${(\varepsilon_n)}$, with
${\delta_n\rightarrow\infty}$ and
${\varepsilon_n\rightarrow\infty}$ such that
\begin{equation*}
  \delta_n/\varepsilon_n^3\rightarrow 0
  \quad\text{ and }\quad
  \delta_n r_n^j/\varepsilon_n^3
  \xrightarrow{n\rightarrow\infty} 0
  \quad
  \text{ for }j=1,\ldots,5\;.
\end{equation*}
 Because of \eqref{eq:variance:convergence}, these sequences can be chosen to additionally satisfy
\begin{equation*}\frac{ \delta_n}{\varepsilon_n^3} \|\widehat{\eta}_{b_n,n}^2-\eta^2_{b_n}\|_{1}
  \xrightarrow{n\rightarrow\infty}0\;.\end{equation*}
Let $\eta$ be the asymptotic variance, as in the hypothesis of the theorem.
Given $(\varepsilon_n)$ and $(\delta_n)$, we construct the sequence $(\eta(n))_n$ as\begin{equation*}
  \eta(n):=
  \eta\mathbb{I}\braces{\eta\in[u_n,v_n]}
  +
  \varepsilon_n\mathbb{I}\braces{\eta\not\in[u_n,v_n]}\;.
\end{equation*} Then
using \cref{lemma:random:scaling} we obtain
\begin{equation*}
\dW(S_n ,\eta(n)Z) \le \delta_n\mean\big[ \dW\big(\tfrac{S_n }{\eta(n)},Z\big|\group\big)\big]\quad\text{ for }\quad
S_n:=\sqrt{|\A_n|}\,\hF_n(h_{n},X_n)\;.
\end{equation*}
To apply \cref{aix_la_chapelle} and \cref{neuchatel}, we note that
\begin{equation*}
  \sup_n \tsum_i c_{i,2}\bigl(
  \bar{h}_n^i(\boldsymbol{\phi}X_n)\mathbb{I}\braces{
    |\tfrac{\bar{h}_n^i(\boldsymbol{\phi}X_n)}{c_{i,2}(h_n)}|>\gamma_n
  }\bigr)\rightarrow 0
  \qquad\text{ as }\gamma_n\rightarrow\infty\;.
\end{equation*}
Recall that the constants $S_0^{\ast}$, $S_2^{\ast}$, etc by definition depend
on the specific choice of the sequence $(k_n')$ and $(\beta_n)$. With the sequences satisfying:
\begin{equation*}
S_2^*\le k'_n \beta_n  \mathcal{S}^n_w 
\qquad
S_4^*\le {k'_n}^2 \beta_n^2  \mathcal{S}^n_w 
\qquad S_0^{\ast} \rightarrow 0\;.
\end{equation*}
Moreover, we have
${\sum_{i\le k_n,j\le k'_n} \zeta_i\,\zeta_j\le \frac{\gamma_n^2}{\varepsilon_n^2}\big[ \sum_i c_{i,2+\varepsilon}(h_n)}\big]^2$
and
\begin{equation*}
  \tsum_i \big\|\bar{h}_n^i(\boldsymbol{\phi}X_n)
  \mathbb{I}\braces{|\bar{h}_n^i(\boldsymbol{\phi}X_n)| \le \gamma_n c_{i,2}\big(\tfrac{h_n}{\eta(n)}\big)}
  \big\|_{L_{\infty}}\le \gamma_n \tsum_i c_{2,i}(h_n)\;.
\end{equation*}
Substituting into
\cref{aix_la_chapelle} and \ref{neuchatel}, we then obtain an upper bound
on ${\mean\big[ \dW\big(\tfrac{S_n }{\eta(n)},Z\big|\group\big)\big]}$ and hence, as shown above, on ${\dW(S_n,Z)}$ as claimed.

\subsection{Proof of the Berry-Esseen theorem}
\label{sec:proof:BE:g}

To prove \cref{theorem:BE:g}, let $\mu_n^*$ be the random measure defined in \cref{gruyere}.
We consider the variable
\begin{equation*}
  W:=
  \tfrac{\sqrt{|\A_n|}}{\eta}
  \mathbb{E}_{\mu_n}[h_n(\boldsymbol{\phi}X_n)|\A_n^{k_n}]
  =
  \tfrac{\sqrt{|\A_n|}}{\eta}
  \tsum_i\mathbb{E}_{\mu_n}[\bar{h}^i_n(\boldsymbol{\phi}X_n)|\A_n^{k_n}]\;,
\end{equation*}
and similarly define $W^{\ast}$ by substituting $\mu_n^{\ast}$ for $\mu_n$, as in
\cref{neuchatel}. If $(b_n)$ is the increasing sequence chosen in the theorem,
\cref{aix_la_chapelle} shows
\begin{equation*}
  \big|\dW(W,Z)-\dW(W^{\ast},Z)\big|
  \;\le\; \frac{k_n^2 C_{1}(\frac{h_n}{\eta(n)})|\B_{b_n}|  \mathcal{S}^n_w }{\sqrt{|\A_n|}}\;.
\end{equation*}
(If hypothesis \cref{H1:g} is assumed, we can in particular choose $b_n=K$ for all $n$ and
some $K$.)
We can apply \cref{neuchatel}, where we choose ${\eta(n):=\eta}$ and ${k_n':=k_n}$ for all $n$.
In \cref{stein:term1}--\ref{stein:term4}, we can set
$p=\frac{3}{2}$ and $q=\frac{1}{3}$. The constants ${S^{\ast}_2,S^{\ast}_4}$ and the
weak spreading coefficient $  \mathcal{S}^n_w $ can then be bounded in terms of the (strong)
spreading coefficients as
\begin{equation*}
  S_{2}^* \le \mathcal{S}^n
  \qquad
  S_{4}^*\le \mathcal{S}^n
  \qquad
    \mathcal{S}^n_w  \le \mathcal{S}^n\;,
\end{equation*}
and substitute these into the bounds in \cref{stein:term1}--\ref{stein:term4}.
The sequences $(\beta_n)$ and $(\gamma_n)$, which respectively
controls moments of $(\mu_n)$ and $\frac{h_n}{\eta(n)}$, are relevant in the proof of the central limit theorem;
here, we can set ${\beta_n=\gamma_n=\infty}$ for all $n$, and note that
\begin{align*}
  \|\bar{h}_n^i(\boldsymbol{\phi}X_n)\mathbb{I}\braces
        {|\bar{h}_n^i(\boldsymbol{\phi}X_n)|\! \le \!\gamma_n c_{i,2}\big(\tfrac{h_n}{\eta(n)}\big)}
        \|_{3(1+\frac{\epsilon}{2})}
        \;=\;
        \|\bar{h}_n^i(\boldsymbol{\phi}X_n)\|_{3(1+\frac{\epsilon}{2})}
        \le\; c_{i,{3(1+\frac{\epsilon}{2})}}\big(\tfrac{h_n}{\eta}\big)
\end{align*}
and ${\zeta_i\le c_{4+2\epsilon,i}\big(\tfrac{h_n}{\eta}\big)}$.
Substituting all terms into \cref{neuchatel} completes the proof.

\newpage
\section{Other Proofs}
\label{proofs:other}
This appendix collects the proofs of all results aside from the main
limit theorems---on mixing coefficients,
concentration, and applications---in the order they appear in the text.

\subsection{Properties of mixing coefficients}
\begin{proof}[Proof of \cref{lemma:hypotheses}] Fix ${n\in\mathbb{N}}$
and ${(A,B)\in \mathcal{C}(n)}$. Using the triangle
inequality,
\begin{equation*}
    \begin{split}
        &\mathbb{E}\big[|P(A|\group)P(B|\group)-P(A\cap B|\group)|\big]\\[.2em]
        &\le 2\!\!\sup_{C\in \sigma({\group})}\mathbb{E}\big[\mathbb{I}(C)\big(P(A|\group)P(B|\group)-P(A\cap B|\group)\big)\big]
        \;\le\;2\!\!\sup_{C\in \sigma({\group})} \bigl( a + b\bigr)
      \end{split}
    \end{equation*}
    where we have abbreviated
    \begin{equation*}
      \begin{split}
        a:=&\,\mathbb{E}\big[\mathbb{I}(C)P(A|\group)P(B|\group)-P(A) P(B\cap C)\big]\\
        \text{and\;\;}b:=&\,\mathbb{E}\big[P(A)P(B\cap C)-\mathbb{I}(C)P(A\cap B|\group)\big]\;.
\end{split}\end{equation*}
It follows from the tower property that
\begin{equation*}
b\;\le\;
\big|P(A\cap B\cap C)-P(A)P(B\cap C)\big|\le \alpha(n)\;,
\end{equation*}
and therefore ${b\le \alpha(n)}$. Similarly,
\begin{equation*}\begin{split}
a\;&\le\;
\big|\mathbb{E}\big[P(A)P(B\cap C)-\mathbb{I}(A)P(B\cap C|\group)\big]\big|\\
&\le\big| P(A)P(B\cap C)-\mathbb{E}\big[\mathbb{I}(A)P(B\cap C|\group)\big]\big|
\le \lim_{k\rightarrow \infty} \alpha(k)=0\;.
\end{split}\end{equation*}
In summary, ${\mathbb{E}[|P(A|\group)P(B|\group)-P(A\cap B|\group)|]\le 4 \alpha(n)}$. Since that is the case for all
${n\in\mathbb{N}}$ and ${(A,B)\in\mathcal{C}(n)}$, we conclude ${\alpha(n|\group)\le 4 \alpha(n)}$
\end{proof}
To relate marginal and conditional mixing coefficients, we use \cref{lemma:mixing}:
\begin{proof}[Proof of \cref{prop_mix_b}]
  Fix ${i,j\leq k}$.
  We can choose a subset ${G\subset\group}$ and vectors
  ${\bphi,\bphi',\bpsi,\bpsi'\in\group^k}$ such that
  ${\delta_{i,j}(\bphi,\bphi',G)\ge t}$ and ${\delta_{i,j}(\bpsi,\bpsi',G)\ge t}$ and
  \begin{equation*}
    \psi_l=\begin{cases}\pi\phi_i &\text{if}~l= i\\ \phi_l & \rm{otherwise}\end{cases}
    \quad
    \psi'_l=\begin{cases}\pi\phi'_j &\text{if}~l= j\\ \phi'_l
      & \rm{otherwise}\end{cases}
    \quad
    \text{ for some }\pi\in\group\;.
  \end{equation*}
  For Borel sets ${A\subset\mathbb{R}^2}$ and ${B\subset\mathbb{R}^G}$, \cref{lemma:mixing} shows
\begin{equation*}\begin{split}&
 \big\|\mathbb{E}\big[\mathbb{I}[(X_{\bphi},X_{\bphi'})\!\in\! A]\mathbb{I}[X_{G}\!\in\! B]|\group\big]-\mathbb{E}\big[\mathbb{I}[(X_{\bpsi},X_{\bpsi'})\!\in\! A]
 \mathbb{I}[X_{G}\!\in\! B]|\group\big]\big\|_{1}\\&\le \alpha(t|\group).
\end{split}
\end{equation*}
Substituting into the definition of $P_{i,j}(\cdot)$ gives
\begin{equation*}
|P(A,B|\group)-\mean[P_{i,j}(A)\mathbb{I}\braces{X_n\in B}|\group_n]|\le \alpha(t|\group)
\end{equation*}
for all ${i,j\leq k}$, and hence ${\alpha_n(t|\group)\le \alpha(t|\group)}$ as claimed.
\end{proof}

\subsection{Concentration}
To prove concentration, we use the
``exchangeable pairs'' variant of Stein's method, in this form due to
Chatterjee \citep{Chatterjee:2005}.
\nolinebreak
\begin{proof}[Proof of \cref{theorem:concentration}]
  The proof strategy is to regard
  ${\mathbb{E}_{\mu_n}[h_n(\boldsymbol{\phi} X_n)|\A_n^{k_n}]}$ as an
  integral, approximate this integral by sums,
  and establish concentration of each sum. These sums are constructed as follows:
  For each ${m\in\mathbb{N}}$, let $C_m$ be an $\epsilon_m$-net with ${\epsilon_m=1/m}$. Let $\lambda_m$ be a partition of
  $\group$ into a countable number of measurable sets; we write $\lambda_m(\phi)$ for the set containing a given ${\phi\in\group}$.
  Clearly, this partition can be chosen such that
  \begin{equation*}
    \text{ each }\phi\in C_m\text{ is in a separate set of }\lambda_m
    \quad\text{ and }\quad
    \lambda_m(\phi)\subset\B_{1/m}(\phi)\;.
  \end{equation*}
  Since ${\lambda_m}$ partitions $\group$, the product
  ${\lambda_m^{k_n}:=\lambda_m\times\ldots\times\lambda_m}$
  partitions $\group^{k_n}$, and we discretize the integral as
  \begin{equation*}
    \Sigma_{nm} := \tsum_{\boldsymbol{\phi}\in C_m^{k_n} }\mathbb{E}_{\mu_n}\big[
      \lambda_m^{k_n}(\bphi)| \A_n^{k_n}\big]h_n(\boldsymbol{\phi}X_n)\;.
  \end{equation*}
  For each fixed ${n\in\mathbb{N}}$, the approximation error satisfies
  \begin{equation*}
    \big\|\Sigma_{nm}-\mathbb{E}_{\mu_n}[h_n(\boldsymbol{\phi}X_n)\big|\A_n^{k_n}]\big\|_{1}
    \le
    \sup_{\substack{\bphi,\bphi'\in\mathbb{G}^{k_n}\\d(\bphi'_i,\bphi_i)\le \epsilon_m,~i\le k_n}}
    \|h_n(\boldsymbol{\phi}X_n)-h_n(\boldsymbol{\phi}'X_n)\|_{\subL1}
    \;\xrightarrow{m}\; 0.
  \end{equation*}
  Thus,
  ${\|\Sigma_{nm}-\mathbb{E}_{\mu_n}[h_n(\boldsymbol{\phi}X_n)\big|\A_n^{k_n}]\|\rightarrow
    0}$ as ${m\rightarrow\infty}$.  
  Since $h_n$ is $\L_1$-uniformly continuous,
  \begin{equation*}
    \mathbb{P}({|\mathbb{E}_{\mu_n}[h_n(\boldsymbol{\phi}X_n)|\A_n^{k_n}]|> t}\,|\,\mu_n)
    \;\le \;
    \limsup_{m}\mathbb{P}({|\Sigma_{nm}|\ge t}\,\big|\,\mu_n)\quad\text{ for }t>0\;.
  \end{equation*}  Now apply the method of exchangeable pairs:
  Consider the sets of vectors $${\lambda_m^{-i}(\phi):=\{(\psi_1,\ldots,\psi_{k_n})\in\A_n^{k_n}\vert \psi_i\in \lambda_m(\phi)\}.}$$
  Recall that ${\Sigma_{nm}}$ is self-bounded by hypothesis. For each
  ${i\leq k_n}$, the self-bounding coefficient is
  ${\sum_{\bphi\in C_m^{k_n}}c_i \mathbb{E}_{\mu_n}[\lambda_m^{-i}(\bphi_i)|\A_n^{k_n}]}$.
  Using \cite[][Theorem 4.3]{Chatterjee:2005}, we obtain
  \begin{align*}
    \mathbb{P}(|\Sigma_{mn}|\ge t|\mu_n)
    & \quad\le \quad
    2\, \mathbb{E}\Big[\exp\Bigl(
    {-\frac{\big(1-\Lambda[(X_{\phi})_{\phi\in C_m}]\big)t^2}
      {\sum_{\phi\in C_m}(\sum_{i}c_i\mathbb{E}_{\mu_n}[\lambda^{-i}_m(\bphi_i)|\A_n^{k_n}])^2}}
      \Bigr)\Big]
    \\ & \quad \le \quad
    2\, \mathbb{E}\Big[ \exp\Bigl(-|\A_n|\frac{(1-\Lambda[(X_{\phi})_{\phi\in C_m}])t^2}
          {\tau^2_n|\B_{1/m}| \big(\sum_{i}c_i\big)^2}\Bigr)\Big]\;,
  \end{align*}
  where the second inequality uses the definition of $\tau_n$. That holds
  for any $m$, and any decreasing sequence $(C_m)$ of nets. For ${m\rightarrow\infty}$, we hence obtain
  \begin{equation*}
    \mathbb{P}(|\mathbb{E}_{\mu_n}(h_n(\boldsymbol{\phi}X_n)|\A_n^{k_n})|\ge t\,|\,\mu_n)
    \le
    2\,\mathbb{E}\big[ \exp\bigl({-|\A_n|\tfrac{(1-\rho_n)t^2}{[\sum_{i}c_i]^2 \tau^2_n }}\bigr)\big]
  \end{equation*}
  as claimed, where we have substituted in the definition of $\rho_n$.
\end{proof}

\subsection{Approximation by subsets of transformations}
Recall that we may assume ${\mean[f(X)|\group]=0}$ without loss of generality, by \cref{lemma:centering}.
\begin{proof}[Proof of \cref{result:discretization}]
Set ${f':=f-\mean[f(X)|\group]}$. By \cref{theorem:CLT:g},
\begin{equation*}
\myint_{\A_n}\tfrac{f'(\phi X)}{\sqrt{|\A_n|}}
|d\phi|\darrow\eta Z\;.
\end{equation*}
For the measures $(\mu_n)$ chosen as $\mu_n(A):=|A\bigcap \mathbb{H}|$, the theorem shows
\begin{equation*}
\myint_{\A_n\cap \mathbb{H}} \tfrac{f'(\phi X)}{\sqrt{|\A_n\cap \mathbb{H}|}}|d\phi|\,\darrow\,\eta_H Z
\quad\text{ and }\quad
\myint_{\A_n} \tfrac{f'(\phi X)}{\sqrt{|\A_n|}}
|d\phi|\,\darrow\,\eta Z\;.
\end{equation*}
Since the random variables ${\eta}$ and ${\eta_H}$ satisfy
\begin{align*}
|\mathbb{K}|\eta_H^2-\eta^2 \; &= \;
|\mathbb{K}|\int_{\mathbb{H}}\mathbb{E}[f(X)f(\phi X)|\group]|d\phi|-\eta^2\\
\; &= \;
\int_{\mathbb{H}}\int_{\mathbb{K}}\mathbb{E}[f(X)[f(\phi X)-f(\phi\theta X)]|\group]|d\theta||d\phi|
\end{align*}
almost surely, the result follows.
\end{proof}

\subsection{Applications}

We first establish \cref{theorem:exchangeable}, on exchangeable structures. The idea of the proof is
to represent $(f(\phi X))_{\phi\in\mathbb{S}_n}$ approximately,
by a certain random field $X_n$ on $\mathbb{Z}^{k_n}$ that is
invariant under diagonal action of shifts. That allows us to apply
\cref{theorem:CLT:g,theorem:BE:g}. That can be read as an example of the generalized U-statistics in
\cref{corollary:U:statistics}.
\begin{proof}[Proof of \cref{theorem:exchangeable}]
For $i\in \mathbb{N}$, we denote
\begin{equation*}
  d_i:=\limsup_{j}\|f(X)-f(\tau_{ij}X)\|_{2}
  \quad\text{ and }\quad
  d_i(\eta):=\limsup_{j}\big\|\tfrac{f(X)-f(\tau_{ij}X)}{\eta}\big\|_{2}\;.
\end{equation*}
Consider the segment ${[i]=\braces{1,\ldots,i}}$, and write
${\mathbb{S}^{[i]}_m=\{\phi\in \mathbb{S}_m\vert\phi[i]=[i]\}}$ for the set of permutations that leave it invariant.\\[.5em]
\noindent\emph{Step 1: Approximation}. We define
\begin{equation*}
  \bar{f}^i(x):=\lim_{m\rightarrow\infty}\tfrac{1}{|\mathbb{S}^{[i]}_m|}\tsum_{\psi\in\mathbb{S}^{[i]}_m}f(\psi x)\;,
\end{equation*}
and use $\bar{f}^i(\phi X)$ as a surrogate of $f(\phi X)$ that depends only on the image $\phi[i]$.
Averaging out the $k$th coordinate gives
\begin{equation*}
  \bar{f}^{i,k}(x):=\lim_{m\rightarrow\infty}\tfrac{1}{m}\tsum_{l\le m}\bar{f}^i(\tau_{l,k}x)\;.
\end{equation*}
We will show that for any increasing, divergent sequence $(k_n)$,
\begin{equation*}
 \tfrac{\sqrt{n}}{|\mathbb{S}_n|}\tsum_{\phi\in \mathbb{S}_n}\big(f(\phi X)- \bar{f}^{k_n}(\phi X)\big)
 \;\xrightarrow{\;\L_1\;}\;0\qquad\text{ as }n\rightarrow\infty\;.
\end{equation*}
Indeed, since ${(f\!-\!\bar{f}^{k_n})=\sum_{k\ge k_n}(\bar{f}^{k+1}\!-\! \bar{f}^{k})}$, we have
\begin{equation*}
  \begin{split}
    &
    \big\| \tfrac{\sqrt{n}}{|\mathbb{S}_n|}\tsum_{\phi\in\mathbb{S}_n}\!\big(f(\phi X)- \bar{f}^{k_n}(\phi X)\big)\big\|^2_{1}
    \;\le\;
    \big\|\tfrac{\sqrt{n}}{|\mathbb{S}_n|}\tsum_{\phi\in\mathbb{S}_n}\!\big(f(\phi X)- \bar{f}^{k_n}(\phi X)\big)\big\|^2_{2}
    \\&\le
    \tfrac{n}{|\mathbb{S}_n|^2}
    \tsum_{\phi,\psi\in\mathbb{S}_n}
    \mathbb{E}\big[\big({f}(\phi X)- \bar{f}^{k_n}(\phi X)\big)\big(f(\psi X)- \bar{f}^{k_n}(\psi X)\big)\big]
    \\&\le
    \tfrac{{n}}{|\mathbb{S}_n|^2}
    \tsum_{k\ge k_n}\tsum_{\phi,\psi\in\mathbb{S}_n}\mathbb{E}\big[\big(\bar{f}^{k+1}(\phi X)- \bar{f}^{k}(\phi X)\big)\big(f(\psi X)- \bar{f}^{k_n}(\psi X)\big)\big]
\end{split}
\end{equation*}
Consider the summands on the right-hand side. Observe that
\begin{equation*}
  \begin{split}
    \mathbb{E}\big[(\bar{f}^{k}(\phi X)- \bar{f}^{k-1}(\phi X))\bar{f}^{\infty,m}(\psi X)\big] =&\; 0
    \\
    \text{ and }\quad
    \mathbb{E}\big[(\bar{f}^{k}(\phi X)- \bar{f}^{k-1}(\phi X))\bar{f}^{k_n,m}(\psi X)\big] =&\; 0
\end{split}\end{equation*}
whenever ${\psi(m)=\phi(k)}$ for ${k\le m}$. Each summand is hence bounded as
\begin{equation*}
  \begin{split}
    \big|\mathbb{E}\big[&
    (\bar{f}^{k+1}(\phi X)- \bar{f}^{k}(\phi X))(f(\psi X)- \bar{f}^{k_n}(\psi X))
    \big]\big|
    \\=&\;
    \big|\mathbb{E}\big[
    (\bar{f}^{k+1}(\phi X)\!-\! \bar{f}^{k}(\phi X))(f(\psi X)\!-\!\bar{f}^{\infty,m}(\psi X)\!-\!\bar{f}^{k_n}(\psi X)\!+\!\bar{f}^{k_n,m}(\psi X))
    \big]\big|
    \\\le&\;
    \big\|\bar{f}^{k+1}(\phi X)\!-\! \bar{f}^{k}(\phi X)\big\|_{2}\,
    \big\|f(\psi X)\!-\!\bar{f}^{\infty,m}(\psi
    X)\!-\! \bar{f}^{k_n}(\psi X)\!+\!\bar{f}^{k_n,m}(\psi
    X)\big\|_{2}\\
    \le&\;
    2 d_k d_m\;.
\end{split}\end{equation*}
Substituting into the bound yields
\begin{equation*}
  \begin{split}
    &
    \big\| \tfrac{\sqrt{n}}{|\mathbb{S}_n|}\tsum_{\phi\in\mathbb{S}_n}(f(\phi X)- \bar{f}^{k_n}(\phi X))\big\|^2_{1}
    \\&\le\;
    \tfrac{{n}}{|\mathbb{S}_n|^2}\tsum_{k\ge k_n}\tsum_{m\in \mathbb{N}}\tsum_{\phi,\psi\in\mathbb{S}_n}\mathbb{I}\braces{\phi(k)=\psi(m)} d_k d_m
    \\&\le\;
    2\big(\tsum_{k\ge k_n} d_k\big)\big(\tsum_{m\in \mathbb{N}} d_m\big)\;\longrightarrow\; 0.
\end{split}\end{equation*}
It hence suffices to show ${\frac{\sqrt{n}}{|\mathbb{S}_n|}\sum_{\phi\in\mathbb{S}_n}\bar{f}^{k_n}(\phi X)}$ is asymptotically normal
if ${k_n=o(n^{1/4})}$.\\[.5em]
\noindent\emph{Step 2: Representation by random fields}.
For each ${n\in\mathbb{N}}$, we construct a scalar random field $X_n$ on
$\mathbb{Z}^{k_n}$ as follows: For ${\mathbf{j}=(j_1,\ldots,j_k)\in\mathbb{Z}^k}$, define
the permutation ${\phi_\mathbf{j}:=\tau_{1,j_1}\circ\cdots\circ\tau_{k,j_k}}$.
Note that ${\phi_\mathbf{j}[k]=\mathbf{j}}$. Then 
\begin{equation*}
  X_n:=\big(Y_{\mathbf{j}}\big)_{\mathbf{j}\in\mathbb{Z}^{k_n}}
  \quad\text{ where }\quad
  Y_{\mathbf{j}}:=\begin{cases}
        \bar{f}^{k_n}(\phi_{\mathbf{j}}X) & \text{if
  }\mathbf{j}_l\ne \mathbf{j}_k\text{ for all } l\ne k\\
  0 & \text{otherwise}\end{cases}
\end{equation*}
is a random element of ${\xspace_n:=\mathbb{R}^{\mathbb{Z}^{k_n}}}$.
The shift
${(\mathbf{i},(x_{\mathbf{j}})_{\mathbf{j}\in\mathbb{Z}^{k_n}})\mapsto(x_{\mathbf{j}+\mathbf{i}})_{\mathbf{j}\in\mathbb{Z}^{k_n}}}$
is an action of the group $\mathbb{Z}^{k_n}$ on $\xspace_n$.
Since $X$ is exchangeable, $X_n$ is by construction invariant under the diagonal action of $\mathbb{Z}^{k_n}$, and its marginal
mixing coefficients satisfy ${\alpha_n(t|\group)=0}$ for all
${t>0}$. \cref{theorem:CLT:g} then shows convergence as
in \eqref{eq:CLT:exchangeable} holds, for ${\eta\condind Z}$.\\[.5em]
\noindent\emph{Step 3: Berry-Esseen bound}. The reasoning is similar: For ${k\in \mathbb{N}}$, we have
\begin{equation*}
  \begin{split}
  &
  \dW\big(\tfrac{\sqrt{n}}{\eta|\mathbb{S}_n|}\!\tsum_{\phi\in\mathbb{S}_n}\!f(\phi X),\tfrac{\sqrt{n}}{\eta|\mathbb{S}_n|}\!\tsum_{\phi\in\mathbb{S}_n}\! \bar{f}^{k}(\phi X)\big)
  \le
2\big(\tsum_{l\ge k} d_l(\eta)\big)\big(\tsum_{m\in \mathbb{N}} d_m(\eta)).
\end{split}
\end{equation*}
We denote ${\eta^2(n):=\sum_{i,j\le k}{\rm Cov}[\mathbb{F}^i(X,e) \mathbb{F}^j(X,\phi)|\group]}$, and observe that
\begin{equation*}
  \begin{split}
    \big\|\mfrac{\eta^2(n)-\eta^2}{\eta^2}\big\|
    \;&\le\; \Big\|\mfrac{\sum_{l=k}^{\infty}\!\sum_{ m\in \mathbb{N}}{\rm Cov}[\mathbb{F}^l(X,e) \mathbb{F}^m(X,\phi)|\group]}{\eta^2}\Big\|\\
    &\le\; 2\big(\!\tsum_{m\in \mathbb{N}}d_m(\eta)\big)\!\tsum_{l\ge k}d_l(\eta).
\end{split}
\end{equation*}
Substituting into \cref{theorem:BE:g} gives
\begin{equation*}
    \dW\big( \mfrac{\sqrt{n}}{\eta|\mathbb{S}_n|}\msum_{\phi\in\mathbb{S}_n} f(\phi X), Z\big)
    \;\le\;  C\big(\mfrac{k^2}{\sqrt{n}}+\msum_{l\ge k}d_l(\eta)\big)\;,
\end{equation*}
for some ${C<\infty}$.
\end{proof}

\begin{proof}[Proof of \cref{result:rgm}]
Write ${\mathbb{L}:=\{\phi\in \group|\phi(W)\cap
W\ne \emptyset\}}$. Observe that, if we choose
$\phi$ to be an element of $\mathbb{H}\setminus (\A_n\cap \mathbb{H})$ that is such that
$\phi(W)\cap \A_n W\ne \emptyset$, then we have $\phi\in \A_n\mathbb{L}\cap \mathbb{H}$. 
This implies that
\begin{align*}&
\big\|\sqrt{|\A_n\cap \mathbb{H}|}\big(\nu_n(h)-\frac{1}{|\A_n\cap \mathbb{H}|}\int_{\A_n\cap \mathbb{H}} f(\phi(\Pi))d|\phi|\big)\big\|_2^2\\&\le\frac{|(\A_n\triangle \A_n \mathbb{L})\cap \mathbb{H}|}{|\A_n\cap \mathbb{H}|}\|f(\Pi)\|_{2+\epsilon}^2\sum_{i\in \mathbb{N}}|\B_{i+1}\setminus \B_i|\alpha_i(i|\group)^{\frac{\epsilon}{2+\epsilon}}
\rightarrow 0\;,
\end{align*}
and \cref{theorem:CLT} shows $\frac{1}{\sqrt{|\A_n\cap \mathbb{H}|}}\int_{\A_n\cap \mathbb{H}} f(\phi(\Pi))-\mathbb{E}(f(\Pi)|\group)d|\phi|\xrightarrow{d}\eta Z$.
\end{proof}

\begin{proof}[Proof of \cref{result:polynomially:stable}]
  By hypothesis, ${\sup_{i>0} i^{-r}|\B_i|<\infty}$, polynomial
  stability holds with index $q>\frac{(2+2\epsilon)r}{\epsilon}$, and
  $\Pi$ is a Poisson process. We have to show that 
  \begin{equation*}
    \myint_{\mathbb{G}} \alpha^{(n)}(d(e,\phi)|\mathbb{G})^{\frac{\epsilon}{2+\epsilon}}|d\phi|<\infty\;.
  \end{equation*}
  For ${b\in\mathbb{N}}$, a subset ${G_1\subset\group}$, and
  ${F\in\mathcal{F}}$, define
  \begin{equation*}
  f_{n,b}(F)=f_n(F\cap\B_b)
  \quad
  Y(G_1):=(f_n(\phi(\Pi)))_{\phi\in G_1}
  \quad
  Y_b(G_1):=(f_{n,b}(\phi(\Pi)))_{\phi\in G_1}\;.
  \end{equation*} 
Write $\mathcal{L}(\argdot)$ for the law
of a random variable. Then
 \begin{equation*}
\big\| \Law(Y(G_1))-\Law(Y_b(G_1))\big\|_{\text{\rm\tiny TV}}
     \le
     P(Y(G_1)\neq Y_b(G_1))
\le \mathbb{E}\big[\sum_{(x,y)\in G_1W\cap \Pi}\mathbb{I}(R(x,y,\Pi_n)>b)\big]
\end{equation*}
An application of Campbell's theorem for Poisson
processes \citep{Kallenberg:2001} shows there are constants
${C_1,C_2>0}$ and $\gamma:=|\{\phi\in \mathbb{H}|\phi(W)\cap
G_1W\ne\emptyset\}|$ such that
 \begin{equation}
        \label{mal_1}
\big\| \Law(Y(G_1))-\Law(Y_b(G_1))\big\|_{\text{\rm\tiny TV}}
     \;\le\;
 C_1\gamma\sup_{(x,m)\in W}P( R(x,m,\Pi_n)>{b})
\;\le\; C_2\gamma b^{-q}\;.
\end{equation}
Let $\bar{d}$ be the Hausdorff metric induced by $d$, and denote $\bar{d}$-balls by $\bar{\B}$. Take $G_1:=\{\phi,\phi'\}$ with elements $\phi,\phi'\in \mathbb{G}$ and
let $G_2$ be another subset of $\group$ with $\bar{d}(G_1,G_2)\ge b$. Then there is $C_3<\infty$ such that
 \begin{equation}\begin{split}&\label{mal_2}
 \big\|\Law(Y(G_2))-\Law(Y_{\bar{d}(G_1,G_2)-\frac{b}{2}}(G_2))\big\|_{\text{\rm\tiny TV}}
 \le
 P(Y(G_2)\ne Y_{\bar{d}(G_1,G_2)-\frac{b}{2}}(G_2))
\\&\quad{\le}\;
\tsum_{j\ge b} P\big(Y({\bar{\B}_{j+1}(G_1)\setminus \bar{\B}_j(G_1)})
\ne
Y_{\frac{2j-b}{2}}({\bar{\B}_{j+1}(G_1)\setminus \bar{\B}_j(G_1)}\big)
\\&\quad{\le}\;
C_3\tsum_{j\ge 0}(j+\mfrac{b}{2})^{-q}(j+b)^{r-1}
\end{split} \end{equation}
where the second inequality applies the union bound, and the third follows by substituting the growth rate and the definition of stability into \cref{mal_1}.
Whenever $G_1$ and $G_2$ satisfy $|G_1|\le 2$ and $\bar{d}(G_1,G_2)\ge b$, and $A,B$ are measurable sets, there is hence a constant $C_4$
such that
\begin{equation*}\begin{split}&
 \big| P(Y(G_1)\in A,Y(G_2)\in B) - P(Y(G_1)\in A)P(Y(G_2)\in B)\big|
 \\&\quad{\le}\;
 \|\Law(Y_{b/2}(G_1))-\Law(Y(G_1))\|_{\text{\rm\tiny TV}}
 +
 \|\Law(Y_{\bar{d}(G_1,G_2)-b/2}(G_2))-\Law(Y(G_2))\|_{\text{\rm\tiny TV}}
\\&\quad{\le}\;
C_4 \big(\tfrac{b}{2}\big)^{{r-q}}
 \end{split} \end{equation*}
The first inequality holds by independence of ${Y_{b/2}(G_1)}$ and ${Y_{{\bar{d}(G_1,G_2)-b/2}}(G_2)}$, the second follows from \cref{mal_1} and \eqref{mal_2}.
That implies $\alpha^{(n)}(b|\mathbb{G}_2)\le C_4(b/2)^{r-q}$, and hence the desired result
since $q> 2\frac{1+\epsilon}{\epsilon}r.$
\end{proof}

\begin{proof}[Proof of \cref{theorem:entropy}]
  Since the group is countable, we can define an order $\prec$ on $\group$ by enumerating
  the elements of $\A_n$ as ${\phi_1^n,\phi_2^n\ldots}$ and declaring
  ${\phi_{i-1}^n\prec\phi_{i}^n}$ for all ${i\in\mathbb{N}}$. For the process $(S_{\phi})$, define the $\sigma$-algebras
  \begin{equation*}
    \mathcal{T}_n(\phi):=
    \sigma\braces{
      S_{\phi'}\,|\, \phi'\in\A_n,\phi'\prec \phi
    }
    \quad\text{ and }\quad
    \mathcal{T}(\phi):= \sigma\braces{S_{\phi'}\,|\,\phi'\prec \phi}\;.
  \end{equation*}
  With these in hand, we define functions
  \begin{equation*}
    \begin{split}
    f_n(S,\phi) :=&\ \log P(S_{\phi}|\mathcal{T}_n(\phi))-\mathbb{E}[\log P(S_{\phi}|\mathcal{T}_n(\phi))]
    \\
    g_m(S,\phi) :=&\ \log P(S_{\phi}|\mathcal{T}(\phi)\cap\B_{m})-\mathbb{E}[\log P(S_{\phi}|\mathcal{T}(\phi)\cap\B_{m})\big]
    \end{split}
  \end{equation*}
  An application of the chain rule then yields
  \begin{equation*}\begin{split}&
      \tfrac{1}{\sqrt{|\A_n|}}\bigl(
      \log P(S_{\A_n})-\mathbb{E}[\log P(S_{\A_n})]\bigr)
      =\tfrac{1}{\sqrt{|\A_n|}}\tsum_{\phi \in \A_n}
      f_n(S,\phi)\;.
  \end{split}\end{equation*}
  Now consider a $\phi$ such that ${\mathcal{T}_n(\phi)\cap\B_{m}=\mathcal{T}(\phi)\cap\B_{m}}$. Then
  \begin{equation*}
    \|f_n(S,\phi)-g_m(S,\phi)\|_{2}\le \rho_{m}\;.
  \end{equation*}
  The number of ${\phi\in\A_n}$ for which that is \emph{not} the case is
  \begin{equation*}
    |\braces{
      \phi\in \A_n\,|\,
      \mathcal{T}_n(\phi)\cap\B_{m}\ne\mathcal{T}(\phi)\cap\B_{m}}|
    \;\le\; |\A_n\triangle \B_{m}\A_n|
  \end{equation*}
  Denote ${M_p:=\sup_{\phi\in\group,A\subset\group}\|\log P(X_{\phi}|X_{A})\|_{p}}$.
  For any ${\phi,\phi'\in \mathbb{G}}$ that satisfy
  ${d(\phi,\phi')\ge i}$ and any ${k\in\mathbb{N}}$, we have
  \begin{equation*}
    \begin{split}&
    \text{Cov}\bigl[
      f_n(S,\phi)-g_m(S,\phi),f_n(S,\phi')-g_m(S,\phi')
      \bigr]\\
    &\le 4\min(\rho_m,\rho_k )^2+8\min(\rho_m,\rho_k )M_2+4  M^2_{2+\varepsilon}
    \alpha^{\frac{\varepsilon}{2+\varepsilon}}(i-k,|\B_m|)\;.
    \end{split}
  \end{equation*}
  Therefore  for any sequence $(b_n)$  satisfying
~${\frac{|\A_n\triangle\B_{b_n}\A_n|}{|\A_n|}\rightarrow 0}$ and ${b_n\rightarrow\infty}$
~we have
  \begin{equation*}
    \tfrac{1}{\sqrt{|\A_n|}}
    \tsum_{\phi\in\A_n}f_n(S,\phi)-g_{b_n}(S,\phi)\xrightarrow{L_2} 0\;.
  \end{equation*}
  Let $\alpha^m$ be the mixing coefficient of $g_m$. Then
  ${\alpha^m(i)\le \alpha(i-2m,|\B_m|)}$. \cref{theorem:CLT:g} hence implies
  \begin{equation*}
    \tfrac{1}{\sqrt{|\A_n|}} {\tsum_{\phi\in \A_n}g_m(\phi X)}\xrightarrow{\;d\;} \eta_m Z
    \quad\text{ for }\quad
    \eta_m^2:= \tsum_{\phi} \text{Cov}[g_m(X),g_m(\phi X)]\;.
  \end{equation*}
  Since ${\eta_m\xrightarrow{m\rightarrow \infty} \eta}$, the result follows.
\end{proof}

\pdfinfo{
/Title (Limit theorems for invariant distributions)
/Author (Morgane Austern, Peter Orbanz)
/Subject ()
/Keywords ()
}

\end{document}